\documentclass{scrartcl}

\usepackage{egen}

\usepackage[english]{babel}
\usepackage{tikz}
\usepackage{caption, subcaption}

\numberwithin{equation}{section}

\theoremstyle{plain}

\newtheorem{nthm}[equation]{Theorem}

\newtheorem{nprop}[equation]{Proposition}

\newtheorem{nlemma}[equation]{Lemma}

\newtheorem{ncor}[equation]{Corollary}

\theoremstyle{definition}

\newtheorem{ndefn}[equation]{Definition}

\newtheorem{nconstruction}[equation]{Construction}

\newtheorem{nex}[equation]{Example}

\theoremstyle{remark}

\usepackage{hyperref}
\usepackage{stmaryrd} 
\usepackage{tensor} 
\usepackage{todonotes}
\usepackage{tikz-cd} 
\usepackage{microtype}

\let\wt\relax
\DeclareMathOperator{\wt}{wt}
\newcommand{\res}{\mathrm{res}}
\newcommand{\mfarrows}[2]{\overset{#1}{\underset{#2}{\rightleftarrows}}}
\DeclareMathOperator{\Crit}{Crit}

\newcommand{\Knorrer}{Knörrer}
\newcommand{\opposite}{\mathrm{op}}
\newcommand{\us}{\mathrm{us}}
\newcommand{\sstable}{\mathrm{ss}}
\newcommand{\triv}{\mathrm{triv}}
\DeclareMathOperator{\sEnd}{\mathcal{E}\mathit{nd}}
\DeclareMathOperator{\sExt}{\mathcal{E}\mathit{xt}}

\title{The homological projective dual of $\Sym^2\PP(V)$}
\date{}
\author{Jørgen Vold Rennemo}

\begin{document}
\maketitle
\begin{abstract}
We study the derived category of a complete intersection $X$ of bilinear divisors in the orbifold $\Sym^{2}\PP(V)$.
Our results are in the spirit of Kuznetsov's theory of homological projective duality, and we describe a homological projective duality relation between $\Sym^{2}\PP(V)$ and a category of modules over a sheaf of Clifford algebras on $\PP(\Sym^{2}V^{\vee})$.

The proof follows a recently developed strategy combining variation of GIT stability and categories of global matrix factorisations.
We begin by translating $D^{b}(X)$ into a derived category of factorisations on an LG model, and then apply VGIT to obtain a birational LG model.
Finally, we interpret the derived factorisation category of the new LG model as a Clifford module category.

In some cases we can compute this Clifford module category as the derived category of a variety.
As a corollary we get a new proof of a result of Hosono and Takagi, which says that a certain pair of nonbirational Calabi--Yau 3-folds have equivalent derived categories.
\end{abstract}

\section{Introduction}
Let $V$ be a vector space, let $\Sym^2\PP(V)$ be the quotient stack $\PP(V)^{2}/\ZZ_{2}$, and let $f : \Sym^{2}\PP(V) \to \PP(\Sym^{2}V)$ be the morphism given by
\[
\{[v_{1}], [v_{2}]\} \mapsto [v_{1}\otimes v_{2} + v_{2} \otimes v_{1}],\ \ \ \ v_{1}, v_{2} \in V.
\]
Choose a vector subspace $L \subset \Sym^{2}(V^{\vee})$.
We then get an orthogonal subspace $L^{\perp} = \{v \in \Sym^{2}V\ |\ (v,L) = 0\} \subset \Sym^{2}V$.
The main goal of this paper is to understand the derived category of the stack $X = f^{-1}(\PP(L^{\perp}))$.

Our first result relates this category to a category of modules over a sheaf of Clifford algebras.
We will define a certain $O(2)$-gerbe, $\cY \to \PP(\Sym^{2}V^{\vee})$, equipped with a locally free sheaf $E$, whose rank is $2\dim V$, and a section $q$ of $\Sym^{2}E$.
From this data we define a sheaf of Clifford algebras $C = T(E)/\cI$, where $T(E)$ is the tensor algebra and $\cI$ is the 2-sided ideal generated by $e \otimes e - q(e)$.

Let $\cY_{L}$ be the restriction of $\cY$ to $\PP(L) \subseteq \PP(\Sym^{2}V^{\vee})$, and keep the notation $C$ for the restriction $C|_{\cY_{L}}$.
There is a derived category $D^{b}(\cY_{L},C)$, whose objects are bounded complexes of coherent $C$-modules.
For such a complex $\cE$ and a point $p \in \PP(L)$, the restriction $\cE|_{p}$ is an $O(2)$-equivariant complex of sheaves on $p$, hence splits as a shifted sum of $O(2)$-representations.
We will define a subcategory $D^{b}(\cY_{L},C)_{\res} \subset D^{b}(\cY_{L},C)$ of grade restricted objects, where $\cE$ is grade restricted if for all $p \in \PP(L)$, only certain specified representations occur in the splitting of $\cE|_{p}$.

Let $n = \dim V$.
We say $X$ has the expected dimension if its codimension in $\Sym^{2}\PP(V)$ equals the codimension of $L^{\perp}$ in $\Sym^{2}V$.
Our first result is:
\begin{nthm}
\label{thm:MainTheorem}
If $X$ has the expected dimension and $n$ is odd, then:
\begin{itemize}
\item If $\codim X > n$, there is a fully faithful functor $D^{b}(X) \into D^{b}(\cY_{L},C)_{\res}$.
\item If $\codim X = n$, there is an equivalence $D^{b}(X) \cong D^{b}(\cY_{L},C)_{\res}$.
\item If $\codim X < n$, there is a fully faithful functor $D^{b}(\cY_{L},C)_{\res} \into D^{b}(X)$.
\end{itemize}
If $X$ has the expected dimension and $n$ is even, then:
\begin{itemize}
\item If $\codim X > n$, there is a fully faithful functor $D^{b}(X) \into D^{b}(\cY_{L},C)_{\res}$.
\item If $n/2 < \codim X \le n$, there is a non-trivial triangulated category $\cC$ which is a fully faithful subcategory of both $D^{b}(X)$ and $D^{b}(\cY_{L},C)_{\res}$.
\item If $\codim X \le n/2$, there is a fully faithful functor $D^{b}(\cY_{L},C)_{\res} \into D^{b}(X)$.
\end{itemize}
\end{nthm}
Explicit descriptions of the fully faithful functors and the subcategory $\cC$ will be given in the course of the proof.
In the cases where $D^{b}(\cY_{L},C)_{\res}$ includes into $D^{b}(X)$ or when there is a subcategory $\cC$ common to both of them, Proposition \ref{thm:HPDualityDecomposition} gives a description of the semiorthogonal complement to $D^{b}(\cY_{L}, C)_{\res}$ or $\cC$ inside $D^{b}(X)$.

Our result is an instance of Kuznetsov's theory of homological projective duality \cite{kuznetsov_homological_2007}.
In Section \ref{sec:HPD} we give an introduction to HP duality and explain how our results fit in.

Our second result is that for certain choices of $L$, we can give a more geometric description of the category $D^{b}(\cY_{L}, C)_{\res}$.
The description will depend on the parity of $n$.

Assume first that $n$ is odd.
Interpreting the points of $\PP(\Sym^{2}V^{\vee})$ as symmetric matrices up to scale, we may stratify the space by the ranks of these matrices.
We assume that $\PP(L)$ does not intersect the locus of matrices of corank $\ge 3$, and that the intersection of $\PP(L)$ with the locus of corank $i$ matrices is nonsingular of the expected dimension for $i = 0,1,2$.
This assumption holds for a general $L$ of dimension $\le 6$.

We define a nonsingular variety $Y \to \PP(L)$ as a double cover of the corank 1 locus in $\PP(L)$, ramified in the corank 2 locus.
At a corank 1 point $q \in \PP(L) \subset \PP(\Sym^{2}V^{\vee})$, the 2 points of the fibre $Y|_{q}$ correspond to the 2 connected components of the variety of maximal isotropic subspaces of the quadratic space $(V, q)$; letting this description hold in families of $q$ in the natural way determines $Y$ up to isomorphism (using e.g.\ Lemma \ref{thm:DeterminedVariety}).
\begin{nprop}
\label{thm:InterpretationOfCliffordNOdd}
Under the assumptions above, $D^{b}(\cY_{L},C)_{\res} \cong D^{b}(Y)$.
\end{nprop}

If we take $\dim V = \dim L = 5$ with $L$ generic, then $X$ and $Y$ are nonsingular Calabi--Yau 3-folds.
Combining Theorem \ref{thm:MainTheorem} and Proposition \ref{thm:InterpretationOfCliffordNOdd} gives $D^{b}(X) \cong D^{b}(Y)$. 
This result has been shown previously by Hosono and Takagi \cite{hosono_double_2013}, using completely different methods.
One interesting feature of this example is that $X$ and $Y$ have fundamental groups $\ZZ/2$ and $\{e\}$, respectively, hence are not birational.

Assume now that $n$ is even, that $\PP(L)$ does not intersect the locus of matrices of corank $\ge 2$, and that the intersection of $\PP(L)$ with the locus of corank $i$ matrices is nonsingular of the expected dimension for $i = 0,1$.
This assumption holds for a general $L$ of dimension $\le 3$.
Define the variety $Y \to \PP(L)$ as the double cover of the corank 0 locus in $\PP(L)$, ramified in the corank 1 locus.

\begin{nprop}
\label{thm:InterpretationOfCliffordNEven}
Under the assumptions above, $D^{b}(\cY_{L},C)_{\res} \cong D^{b}(Y)$.
\end{nprop}

The proofs of Propositions \ref{thm:InterpretationOfCliffordNOdd} and \ref{thm:InterpretationOfCliffordNEven} are obtained by combining Proposition \ref{thm:MainPropositionClifford} with Propositions \ref{thm:GeometricInterpretationOfMFCategory} and \ref{thm:InterpretationOfCliffordNEvenLater}, respectively.

\subsection{Proof of Theorem \ref{thm:MainTheorem}}
The strategy of the proof will be to combine categories of matrix factorisations with variation of GIT stability.
This approach was first described in \cite{segal_equivalence_2011}, inspired by the physics paper \cite{herbst_phases_2008}. 
See also \cite{ballard_homological_2013, addington_pfaffian-grassmannian_2014, favero_toric_2014} for other applications of this strategy.

Categories of matrix factorisations will be properly introduced in Section \ref{sec:Factorisations}.
For now it is enough to know that given a stack $\cX$ equipped with a function $W$ and some extra data, one can define the category of factorisations $D^{b}(\cX,W)$, which is a generalisation of the usual derived category $D^{b}(\cX)$.

The diagram below summarises the strategy.
\[
\begin{tikzcd}
D^{b}(\cZ_{+},W) \arrow[hook]{r}{GIT} & D^b(\cZ,W) \arrow[hookleftarrow]{r}{GIT} & D^{b}(\cZ_{-},W)_{\res} \arrow{d}{\cong} \\ 
D^{b}(X) \arrow{u}{\cong}[swap]{\text{\Knorrer{} per.}} & & D^{b}(\cY_{L},C)_{\res}
\end{tikzcd}
\]
The first step is to replace the category $D^{b}(X)$ by an equivalent category of matrix factorisations.
Let $\cO_{\Sym^{2}\PP(V)}(1,1)_{+} \in \Coh(\Sym^{2}\PP(V))$ be the $\ZZ_{2}$-equivariant sheaf $\cO(1,1)$ on $\PP(V)^{2}$, equipped with the $\ZZ_{2}$-action which leaves the restriction of $\cO(1,1)$ to the diagonal of $\PP(V)^{2}$ fixed.
Then $X$ is cut out by a section $s$ of $\cO(1,1)_{+}^{\oplus l}$, where $l = \dim L$.

Let $\cZ_{+}$ be the total space of the stacky vector bundle $\cO(-1,-1)_{+}^{\oplus l} \to \Sym^2\PP(V)$.
Dualising $s$ gives a function $W$ on $\cZ_{+}$.
A result known as \Knorrer{} periodicity (\cite[4.6]{isik_equivalence_2013}, \cite[3.4]{shipman_geometric_2012}) then says that $D^{b}(X) \cong D^{b}(\cZ_{+},W)$. 
This is Proposition \ref{thm:KnorrerPeriodicity}.

The space $\cZ_{+}$ is a GIT quotient for a quotient stack $\cZ = V \times V \times L/G$, where $G = (\CC^{*})^{2} \rtimes \ZZ_{2}$.
We show that there is a full ``window'' subcategory $\cW_{+} \subset D^{b}(\cZ,W)$ such that composing with the restriction functor $D^{b}(\cZ,W) \to D^{b}(\cZ_{+},W)$ we get an equivalence $\cW_{+} \cong D^{b}(\cZ_{+},W)$.

Having translated $D^{b}(X)$ into a window category, we next cross the GIT wall.\footnote{As Lil Jon \& the East Side Boyz put it: ``To the window! (To the window) / To the wall! (To the wall)''}
The stack $\cZ$ has a second GIT quotient $\cZ_{-}$, which geometrically is a vector bundle on an $O(2)$-gerbe $\cY_{L} \to \PP(L)$.
Again we find a full subcategory $\cW_{-} \subset D^{b}(\cZ,W)$ equivalent to $D^{b}(\cZ_{-},W)$.
As $\cZ_{-}$ is not a Deligne--Mumford stack, the category $\cW_{-}$ is too big to be directly compared to $\cW_{+}$ in the way we want.
We therefore define a subcategory $\cW_{-,\res} \subset \cW_{-}$ and get a corresponding subcategory $D^{b}(\cZ_{-},W)_{\res} \subset D^{b}(\cZ_{-},W)$.
These results on window categories are Proposition \ref{thm:windowsAsGIT} and Corollary \ref{thm:GITCorollary}.

Let $\pi : \cZ_{-} \to \cY_{L}$ be the projection.
We can find a $K \in D^{b}(\cZ_{-},W)$ such that $\pi_{*}(\Rhom(K,K)) \cong C$.
The functor $\pi_{*}(\Rhom(K,-)) : D^{b}(\cZ_{-},W) \to D^{b}(\cY_{L},C)$ is then an equivalence, which restricts to give $D^{b}(\cZ_{-},W)_{\res} \cong D^{b}(\cY_{L},C)_{\res}$.
This is Proposition \ref{thm:MainPropositionClifford}.

We thus have
\[
\cW_{+} \cong D^{b}(X)\ \ \ \ \text{and}\ \ \ \ \cW_{-,\res} \cong D^{b}(\cY_{L},C)_{\res}.
\]
Theorem \ref{thm:MainTheorem} now follows from these equivalences, because in each case it is obvious from the definitions that $\cW_{-,\res}$ and $\cW_{+}$ are either contained one in the other as subcategories of $D^{b}(\cZ,W)$ or have $\cC := \cW_{-,\res} \cap \cW_{+}$ non-trivial.

\subsection{Related works}
\label{sec:RelatedWorks}
This project began as an attempt to understand and generalise Hosono and Takagi's work in \cite{hosono_double_2013}, which treats the special case where $\dim V = \dim L = 5$.
They find an equivalence between two Calabi--Yau 3-folds $X$ and $Y$, and also conjecture that this equivalence generalises to a statement in homological projective duality.
Ingalls and Kuznetsov have studied the case where $\dim V = \dim L = 4$ in \cite{ingalls_nodal_2015}.

Our main theorem is inspired by Kuznetsov's description of the derived categories of intersections of quadrics in terms of even Clifford algebras \cite{kuznetsov_derived_2008}, which we informally recall in Section \ref{sec:Quadrics}.
Our category of Clifford modules $D^{b}(\cY_{L},C)_{\res}$ is different from the one in that paper in two important ways.
Firstly, our sheaf of Clifford algebras does not live on $\PP(L)$, but rather on the $O(2)$-gerbe $\cY_{L}$.
In particular, a module over $C$ is locally an $O(2)$-equivariant sheaf on $\PP(L)$.
Secondly, the need to consider the subcategory of grade restricted modules is new to our case.
Both of these features mean that the description in terms of Clifford modules is less useful than in the quadric case, and in proving Propositions \ref{thm:InterpretationOfCliffordNOdd} and \ref{thm:InterpretationOfCliffordNEven} we work mostly with the equivalent category $D^{b}(\cZ_{-},W)_{\res}$ instead of with $D^{b}(\cY_{L},C)_{\res}$.

As the title indicates, our results are motivated by Kuznetsov's theory of homological projective duality \cite{kuznetsov_homological_2007}. 
Our Theorem \ref{thm:MainTheorem} is close to saying that the category $D^{b}(\cY,L)_{\res}$ is the homological projective dual of $\Sym^{2}\PP(L)$.
We will explain this statement further in Section \ref{sec:HPD}, which also contains background on homological projective duality.

As explained above, a crucial step in the proof of Theorem \ref{thm:MainTheorem} is to relate the categories of factorisations on different GIT quotients.
The techniques for doing this were introduced in this context by Segal in \cite{segal_equivalence_2011}, and have since been worked out in great generality by Ballard, Favero and Katzarkov \cite{ballard_variation_2012}, and by Halpern-Leistner \cite{halpern-leistner_derived_2015}.
The main result of these two papers is that if $X\!\sslash\! G \subset X/G$ is a GIT quotient, then there exists a full subcategory $\cW \subset D^{b}(X/G)$ such that the restriction functor $D^{b}(X/G) \to D^{b}(X \!\sslash\! G)$ gives an equivalence $\cW \cong D^{b}(X \!\sslash\! G)$.
When $X/G$ is equipped with a superpotential $W$, it is shown in \cite{ballard_variation_2012} that same results hold for factorisation categories, i.e.\ there is a $\cW \subset D^{b}(X/G,W)$ such that $\cW \cong D^{b}(X\!\sslash\! G,W)$.

To define $\cW$, one first writes down a sequence of 1-parameter subgroups $\lambda_{i} \subset G$ and a sequence of open subvarieties $X_{i}$ of the fix point loci $X^{\lambda_{i}}$.
For any $\cE \in D^{b}_{G}(X)$ (or $D^{b}_{G}(X,W)$), the restriction $\cE|_{X_{i}}$ is then graded by $\lambda_{i}$-weights, and we define $\cW$ by saying $\cE \in \cW$ if the $\lambda_{i}$-weights of $\cE$ are contained in a certain interval $J_{i} \subset \ZZ$ for all $i$.
Unfortunately, the precise results of \cite{ballard_variation_2012, halpern-leistner_derived_2015} are not applicable in our case, as for our GIT quotients $\cZ_{+}, \cZ_{-} \subset \cZ$, the subcategories of $D^{b}(\cZ,W)$ constructed by these papers are not comparable in the way we want.
See Section \ref{sec:StrangeWindows} for a further discussion of this point.

We remedy this by giving an ad hoc definition of the subcategory $\cW_{+}$.
Since we only consider a quotient of an affine space, the technical details are considerably simpler than in the general case, and modifying the arguments of \cite{ballard_variation_2012, halpern-leistner_derived_2015} allows us to give a direct proof of the equivalence $\cW_{+} \cong D^{b}(\cZ_{+},W)$.
A novel feature of our case is that it is necessary to consider weights with respect to a 2-dimensional subtorus of our group $G$, instead of just to 1-parameter subgroups.
The definition of the category $\cW_{-}$ follows \cite{ballard_variation_2012, halpern-leistner_derived_2015}.

As mentioned above, the overall strategy of our proof has been applied successfully to several examples, beginning with \cite{segal_equivalence_2011, shipman_geometric_2012}.
Producing homological projective duals by this method was carried out in certain cases by Ballard et al.\ in \cite{ballard_homological_2013}.
They apply this further to the example of degree $d$ hypersurfaces in \cite{ballard_derived_2014}, in particular recovering Kuznetsov's quadric example \cite{kuznetsov_derived_2008}.
Our proof of the equivalence between the factorisation category $D^{b}(\cZ_{-},W)_{\res}$ and the Clifford module category $D^{b}(\cY,C)_{\res}$ goes along the same lines as parts of their proof in the case $d = 2$.
See also \cite{dyckerhoff_compact_2011}, where a similar equivalence is shown for a single Clifford algebra.

The overall VGIT/LG model approach is also used in Addington, Donovan and Segal's paper \cite{addington_pfaffian-grassmannian_2014}, which reproves the Pfaffian--Grassmannian equivalence of Calabi--Yau 3-folds from \cite{borisov_pfaffian-grassmannian_2009, kuznetsov_homological_2006}.
The fact that we need to take a good subcategory $D^{b}(\cZ_{-},W)_{\res} \subset D^{b}(\cZ_{-},W)$ has a parallel in their paper, as one of their gauged LG models is also an Artin stack.
They speculate that this category corresponds to what physicists call the category of branes in an associated $B$-model \cite[4.1]{addington_pfaffian-grassmannian_2014}.
At present the choice of this subcategory is rather ad hoc, and it will be interesting to see to what extent it can be made in a general way.

The example we consider has been studied from a physical perspective by Hori in \cite{hori_duality_2013}.
See also \cite{hori_linear_2013}, which fits both the Pfaffian--Grassmannian example and the one we study into a long list of similar examples; these await a mathematical treatment.

\subsection{Conventions}
We work over $\CC$.

For objects $\cE, \cF$ in a triangulated category $\cC$, we use the convention that $\Hom(\cE,\cF)$ is the space of maps in $\cC$ and $\RHom(\cE,\cF)$ is the graded space $\oplus_{i \in \ZZ}\Hom(\cE,\cF[i])$.

If $G$ is an algebraic group acting on $X$ and $\rho$ is a representation of $G$, we write $\cO_{X}(\rho)$ for the $G$-equivariant sheaf $\rho \otimes \cO_{X}$.
If $G$ is a $k$-dimensional torus, we denote by $\cO_{X}(i_{1}, \ldots, i_{k})$ the line bundle associated with the character $t_{1}^{i_{1}}\cdots t_{k}^{i_{k}}$.
Finally, if $G = (\CC^{*})^{2} \rtimes \ZZ_{2}$, we write $\cO_X(k,k)_{\pm}$ for $\cO_X(\rho)$, where $\rho$ is the character of $G$ which is $t_{1}^{k}t_{2}^{k}$ on $T$ and which sends the generator of $\ZZ_{2}$ to $\pm 1$.

\subsection{Acknowledgements}
This paper is a modified version of my Ph.D.\ thesis.
I am very grateful to my Ph.D.\ supervisor Richard Thomas for suggesting this topic, for many interesting conversations, and for all the help, advice and encouragement he has provided over four years.

Thanks also to E.\ Segal, who explained to me many of the ideas and tools used here.
I thank N.\ Addington, T.\ Bridgeland, T.\ Coates, D.\ Halpern-Leistner, K.\ Hori, S.\ Hosono, A.\ Kuznetsov and T.\ Pantev for useful conversations; Addington also jointly with Richard suggested the topic.

I thank M.\ Akhtar for thoroughly proofreading the thesis version of this paper; whatever mistakes or typos remain are of course entirely his fault.

\section{Homological projective duality}
\label{sec:HPD}
Theorem \ref{thm:MainTheorem} is motivated by Kuznetsov's theory of homological projective duality, which we explain in this section.
We first present the general definitions and results of the theory, taken from \cite{kuznetsov_homological_2007}.
Next we discuss the example of HP duality for quadric hypersurfaces in $\PP^{n}$.
Finally we explain how our results are a form of HP duality for bilinear divisors in $\Sym^{2}\PP^{n}$.

Note that the proofs of our propositions do not depend on the general results of HP duality, and so logically speaking this section is independent from the rest of the paper.

\subsection{The base locus and the incidence variety}
As a warm-up, we first treat a simple version of HP duality where the derived category results are clear from the geometry.
Let $X$ be a smooth, projective variety with a morphism $f : X \to \PP(V)$ for some vector space $V$, with $f$ not factoring through any linear subspace of $\PP(V)$, and let $\cL = f^{*}(\cO(1))$.
Choose a linear subspace $L \subset V^{\vee}$, which gives a linear system $\PP(L)$ of divisors of class $\cL$.

There are two natural schemes we can construct from this linear system.
Firstly, we can intersect the divisors in the linear system to get the base locus $X_{L^{\perp}} \subset X$.
Secondly, we can construct the incidence variety $\cH_{L} \subset X \times \PP(L)$, which consists of pairs $(x, H)$ such that $x \in H$.

Let us assume that $X_{L^{\perp}}$ has the expected dimension.
The first step towards HP duality is the observation that $D^{b}(X_{L^{\perp}})$ then includes as a full subcategory of $D^{b}(\cH_{L})$.

Consider first the case where $\PP(L) = \PP^{1}$.
Then $\cH_{L}$ is the blowup of $X$ in $X_{L^{\perp}}$, and by \cite[3.4]{bondal_semiorthogonal_1995} we get a semiorthogonal decomposition\footnote{The reference assumes $X_{L^{\perp}}$ to be nonsingular, but by \cite{kuznetsov_homological_2007} this is not necessary.}
\[
D^{b}(\cH_{L}) = \langle D^{b}(X_{L^{\perp}}), D^{b}(X) \rangle.
\]
More generally, if $\PP(L) = \PP^{l}$, $l \ge 1$, then the projection $\cH_{L} \to X$ has fibres $\PP^{l-1}$ over $X \setminus X_{L^{\perp}}$, which jump to $\PP^{l}$ over $X_{L^{\perp}}$.
This gives a semiorthogonal decomposition of $D^{b}(\cH_{L})$ with 1 piece isomorphic to $D^{b}(X_{L^{\perp}})$ and $l$ pieces isomorphic to $D^{b}(X)$.
In general, the inclusion functor $D^{b}(X_{L^{\perp}}) \to D^{b}(\cH_{L})$ is given by $i_{*}p^{*}$ with $p$ and $i$ as in the diagram
\[
\begin{tikzcd}
X_{L^{\perp}} \times \PP(L) \arrow[hook]{r}{i} \arrow{d}{p} & \cH_{L}\\
X_{L^{\perp}} &
\end{tikzcd}
\]

\subsection{Lefschetz decompositions}
Kuznetsov's remarkable discovery is that this relation between the base locus $X_{L^{\perp}}$ and the universal hyperplane $\cH_{L}$ can be turned into something more interesting if we can put a certain extra structure on $D^{b}(X)$.
Namely, assume that the derived category $D^{b}(X)$ admits a semiorthogonal decomposition
\[
D^{b}(X) = \langle \cA_{0}, \cA_{1}(1), \ldots, \cA_{k}(k) \rangle,
\]
where the $\cA_{i}$ are full subcategories of $D^{b}(X)$ satisfying $\cA_{i} \subseteq \cA_{i-1}$ for all $i \ge 1$, and where $\cA_{i}(i)$ denotes the full subcategory whose objects are $\cE \otimes \cL^{\otimes i}$, $\cE \in \cA_{i}$.
Such a decomposition is called a Lefschetz decomposition.

For any hyperplane $H \subset \PP(V)$ inducing a divisor $X_{H} := f^{-1}(H)$, the functor
\[
\cA_{i}(i) \to D^{b}(X) \stackrel{-|_{X_{H}}}{\to} D^{b}(X_{H}).
\]
is full and faithful for $1 \le i \le k$. 
Furthermore, the image subcategories $\cA_{i}(i) \subset D^{b}(X_{H})$ are semiorthogonal.
Both of these facts are easy to show using our assumptions on $\cA_{i}$ and the exact triangle
\[
\cE\otimes \cL^{-1} \to \cE \to \cE|_{X_{H}}\ \ \ \ \ \cE \in D^{b}(X).
\]

We therefore have a full subcategory $\langle \cA_{1}(1), \ldots, \cA_{k}(k) \rangle \subset D^{b}(X_{H})$, and letting $\cC_{H} = \langle \cA_{1}(1), \ldots, \cA_{k}(k) \rangle^{\perp}$, we get a semiorthogonal decomposition
\[
D^{b}(X_{H}) = \langle \cC_{H}, \cA_{1}(1), \ldots, \cA_{k}(k) \rangle.
\]
We see that $D^{b}(X_{H})$ decomposes into the parts $\cA_{i}(i)$ inherited from $D^{b}(X)$, and the one new part $\cC_{H}$.
This motivates the term ``Lefschetz decomposition'', cf. the Lefschetz hyperplane theorem.

More generally, let $L \subset V^{\vee}$ be a linear subspace, let $L^{\perp} = \{v \in V\ |\ (v,L) = 0\} \subset V$, and let $X_{L^{\perp}} := f^{-1}(\PP(L^{\perp}))$, which is the base locus of the linear system $\PP(L)$.
We then get a semiorthogonal decomposition
\[
D^{b}(X_{L^\perp}) = \langle \cC_{L^\perp}, \cA_{l}(l), \ldots, \cA_{k}(k) \rangle.
\]
One way of summarising HP duality is that if we know the category $\cC_{H}$ for all hyperplanes $H$ in the system $\PP(L)$, then we get a description of the category $\cC_{L^\perp}$, in terms of the ``homological projective dual'' variety, which we now describe.

\subsection{The homological projective dual}
Let $Y$ be a variety equipped with a map $g : Y \to \PP(V^{\vee})$, and assume that for every point $H \in \PP(V^{\vee})$ the fibre $Y_{H}$ satisfies $D^{b}(Y_{H}) \cong \cC_{H} \subset D^{b}(X_{H})$.
If $Y$ satisfies a certain strengthening of this condition,\footnote{Let $\cH \subset X \times \PP(V^{\vee})$ be the incidence variety.
The Lefschetz decomposition on $D^{b}(X)$ induces a certain decomposition $D^{b}(\cH) = \langle \cC, \cA_{1}(1)\boxtimes D^{b}(\PP(V^{\vee})), \ldots, \cA_{k}(k)\boxtimes D^{b}(\PP(V^{\vee})) \rangle$, and we require that there is an equivalence $D^{b}(Y) \cong \cC$, satisfying some further formal properties, see \cite[6.1]{kuznetsov_homological_2007}, \cite[2.3.9]{ballard_homological_2013}.} then we say that $Y$ is a homological projective dual for $X$.
Note that the existence of such a $Y$ is not automatic.

Here $Y$ is analogous to the incidence variety $\cH_{V^{\vee}} \subset X \times \PP(V^{\vee})$, with the difference that the ``categorical fibre'' $D^{b}(Y_{H})$ at each $H \in \PP(V^{\vee})$ is now the interesting part $\cC_{H} \subset D^{b}(X_{H})$ rather than the whole of $D^{b}(X_{H})$.

For any $L \subset V^{\vee}$, let $Y_{L} = g^{-1}(\PP(L))$.
Just as we saw above that $D^{b}(X_{L^{\perp}})$ includes into $D^{b}(\cH_{L})$, we can now include $\cC_{L^{\perp}}$ into $D^{b}(Y_{L})$.

To state the precise result, we will need some notation. 
Let $\cO_{Y}(1) = g^{*}\cO_{\PP(V^{\vee})}(1)$. 
Kuznetsov shows that $D^{b}(Y)$ admits a ``dual'' Lefschetz decomposition
\[
D^{b}(Y) = \langle \cB_{m}(-m), \cB_{m-1}(-m-1), \ldots, \cB_{0}(0) \rangle,
\]
where $\cB_{i} \subseteq \cB_{i+1}$ for $i \ge 1$.
Let $l$ and $c$ be the dimension and codimension of $L$, respectively.
\begin{nthm}[\cite{kuznetsov_homological_2007}, Thm.\ 1.1]
\label{thm:HPD}
If $X_{L^{\perp}}$ and $Y_{L}$ have the expected dimensions, then we have semiorthogonal decompositions
\[
D^{b}(X_{L^{\perp}}) = \langle \cC_{L^{\perp}}, \cA_{l}(l), \cdots, \cA_{k}(k) \rangle,
\]
\[
D^{b}(Y_{L}) = \langle \cB_{-m}(-m), \cB_{-m-1}(-m-1), \ldots, \cB_{-c}(-c), \cC_{L} \rangle,
\]
and $\cC_{L} \cong \cC_{L^{\perp}}$.
\end{nthm}
The most striking consequence of this theorem is that $D^{b}(X_{L^{\perp}})$ and $D^{b}(Y_{L})$ have the semiorthogonal piece $\cC_{L} \cong \cC_{L^{\perp}}$ in common.
The functor $\cC_{L^{\perp}} \stackrel{\cong}{\to} \cC_{L} \into D^{b}(Y_{L})$ is obtained by composing the functor $D^{b}(X_{L^{\perp}}) \into D^{b}(\cH_{L})$ with a certain projection $D^{b}(\cH_{L}) \to D^{b}(Y_{L})$.

If the dimension of $L$ is sufficiently low (resp.\ high) we get a fully faithful inclusion $D^{b}(Y_{L}) \into D^{b}(X_{L^{\perp}})$ (resp.\ $D^{b}(X_{L^{\perp}}) \into D^{b}(Y_{L})$).

As an aside, we note that the notion of HP dual makes sense more generally than in the setting described here.
In particular, one can drop the restriction of considering derived categories of varieties, and instead consider more general triangulated categories, linear over $D^{b}(\PP(V))$ and $D^{b}(\PP(V^{\vee}))$.
For some nice such categories the same results can be shown.
The main results of this paper deal with HP duality in this extended sense; see also \cite{ballard_homological_2013} and the next section.

\subsection{HP duality for quadrics}
\label{sec:Quadrics}
We will now explain the results of HP duality for the case of quadric hypersurfaces, worked out by Kuznetsov in \cite{kuznetsov_derived_2008}.
This is both an instructive example of HP duality in general and formally quite similar to the case we treat in this paper.

In the terminology used above, we take $X = \PP(V)$, and let the map $f$ be the Veronese embedding $\PP(V) \into \PP(\Sym^{2}V)$ with associated line bundle $\cL = \cO_{\PP(V)}(2)$.
Let $n = \dim \PP(V)$.
The semiorthogonal decomposition
\[
D^{b}(\PP(V)) = \langle \cO, \cO(1), \ldots, \cO(n) \rangle
\]
gives rise to a Lefschetz decomposition with $\cA_{0} = \cdots = \cA_{(n-1)/2} = \langle \cO, \cO(1) \rangle$ when $n$ is odd, and a Lefschetz decomposition $\cA_{0} = \cdots = \cA_{n/2-1} = \langle \cO, \cO(1) \rangle$, $\cA_{n/2} = \langle \cO \rangle$ when $n$ is even.

Let us focus on the case where $n$ is odd; similar results hold for even $n$.
Consider a hyperplane $H \subset \PP(\Sym^{2}V)$, such that the associated quadric $Q = f^{-1}(H) \subset \PP(V)$ is nonsingular.
One can then show that there is a semiorthogonal decomposition
\begin{equation}
\label{eqn:quadric}
D^{b}(Q) = \langle S_{+}, S_{-}, \cO(2), \cO(3), \ldots, \cO(n)\rangle,
\end{equation}
where $S_{+},S_{-}$ denotes (some twist of) the so-called spinor bundles.

The spinor bundles are natural bundles defined on all nonsingular quadrics; there are 2 spinor bundles on even-dimensional quadrics and 1 on odd-dimensional quadrics (see e.g.\ \cite{addington_spinor_2009}).
In low dimensions the spinor bundles are easily described: For a 2-dimensional quadric $\PP^{1}\times \PP^{1}$, they are $\cO(1,0)$ and $\cO(0,1)$, and for the 4-dimensional quadric $\Gr(2,4)$, they are the universal quotient bundle and the dual of the universal sub-bundle.

Let us write $\cC_{Q}$ for the component of $D^{b}(Q)$ denoted by $\cC_{H}$ in the previous section.
The decomposition \eqref{eqn:quadric} implies that $\cC_{Q} = \langle S_{+},S_{-}\rangle$.
The spinor bundles satisfy $\RHom(S_{\pm},S_{\pm}) = \CC$, and $\RHom(S_{\pm}, S_{\mp}) = 0$, so we have $\cC_{Q} = \langle S_{+}, S_{-} \rangle = D^{b}(\pt \sqcup \pt)$.
If we now deform the nonsingular quadric to a singular quadric $Q$ of corank 1, the bundles $S_{+}$ and $S_{-}$ become isomorphic, and we can furthermore show that in this case $\cC_{Q} = D^{b}(\Spec k[\varepsilon])$.

Ignoring a technical issue which we will discuss shortly, this tells us exactly what the HP dual variety $Y$ is over the locus in $\PP(V^{\vee})$ corresponding to nonsingular and corank 1 quadrics. 
Namely, we see that $Y$ is a double cover of the locus of nonsingular quadrics, ramified in the locus of corank 1 quadrics.

The simplest application of Theorem \ref{thm:HPD} is now to the case of a general pencil $\PP^{1} = \PP(L) \subset \PP(\Sym^{2}V^{\vee})$ generated by two quadrics $Q_{1}, Q_{2}$.
In this case the base locus $X_{L^{\perp}} = Q_{1} \cap Q_{2}$, and $Y_{L}$ is a double cover of $\PP(L) = \PP^{1}$, ramified in the $n+1$ points corresponding to singular quadrics in the pencil.
Theorem \ref{thm:HPD} then gives an old result of Bondal and Orlov \cite{bondal_semiorthogonal_1995}:
\[
D^{b}(Q_{1} \cap Q_{2}) = \langle D^{b}(Y_{L}), \cO(4), \ldots, \cO(n) \rangle,
\]

The technical issue ignored above is the fact that our description of the HP dual category was only true point-wise and may fail in a global setting.
To explain this complication, let us first give Kuznetsov's general description of the HP dual in terms of Clifford algebras.

Let $V$ be a vector space with a quadratic form $q$.
The Clifford algebra $C_{q}$ is defined to be $T(V)/I$, where $T(V)$ is the tensor algebra, and $I$ is the 2-sided ideal generated by $v \otimes v - q(v)$.
Taking $q = 0$ gives the exterior algebra $\wedge^{*}V$, and the Clifford algebras are in this sense deformations of $\wedge^{*}V$.
The natural grading on $T(V)$ descends to a $\ZZ_{2}$-grading on $C_{q}$, and taking the degree 0 part we obtain the ``even Clifford algebra'' $C_{q}^{0} \subset C_{q}$.

Now letting $q \in \PP(\Sym^{2}V^{\vee})$ vary, one can fit these even Clifford algebras into a global family, i.e.\ there is a sheaf of algebras $C$ on $\PP(\Sym^{2}V^{\vee})$ such that the restriction to each $q \in \PP(\Sym^{2}V^{\vee})$ is isomorphic to $C^{0}_{q}$.
Kuznetsov shows that the HP dual of $\PP(V)$ is the category $D^{b}(\PP(\Sym^{2} (V^{\vee})),C)$, i.e.\ the derived category of coherent $C$-modules on $\PP(\Sym^{2}(V^{\vee}))$.
This means in particular that Theorem \ref{thm:HPD} holds when we interpret $D^{b}(Y_{L})$ as $D^{b}(\PP(L),C|_{\PP(L)})$.

Let us consider what this means for a single quadric.
For any $q \in \PP(\Sym^{2}V^{\vee})$, if $Q \subset \PP(V)$ is the associated quadric, we find $\cC_{Q} \cong D^{b}(q, C|_{q}) = D^{b}(C^{0}_{q})$.
If we assume that $q$ and hence $Q$ is nonsingular, then it is a classical fact that $C^{0}_{q} \cong \End(\CC^{N}) \oplus \End(\CC^{N})$ for some $N$.
By Morita equivalence we then get $D^{b}(C^{0}_{q}) = D^{b}(\End(\CC^{N})) \oplus D^{b}(\End(\CC^{N})) \cong D^{b}(\pt) \oplus D^{b}(\pt)$.
Thus we recover the statement that the fibre of the HP dual at $q$ is 2 points.

We can now explain the complication in the global description of the HP dual.
Keeping to the locus of nonsingular $q$, the above discussion shows that the centre of the algebra $C$ is a commutative algebra on $\PP(\Sym^{2}V^{\vee})$, whose spectrum is a double cover $Z$.
The algebra $C$ is then equivalent to an Azumaya algebra $A$ on $Z$, i.e.\ an algebra which is étale locally isomorphic to $\sEnd(\cO_{Z}^{N})$.
The above results can be rephrased as saying that the HP dual is given by $D^{b}(Z,A)$.

If there exists a locally free sheaf $\cE$ on $Z$ such that $A \cong \sEnd(\cE)$, we can define an equivalence $D^{b}(Z) \cong D^{b}(Z,A)$ by the inverse functors $-\otimes_{\cO_{Z}} \cE$ and $\Rhom_{A}(\cE, -)$.
This can always be done locally, but there is a global obstruction to the existence of such an $\cE$, known as the Brauer class of $A$, which lives in $H^{2}_{\mathrm{an}}(Z,\cO_{Z}^{*})$.
In this example, the Brauer class does not always vanish, and in fact $D^{b}(Z,A)$ is not in general equivalent to $D^{b}(Z)$.

\subsection{HP duality for $\Sym^2\PP(V)$}
\label{sec:HPDOurCase}
The motivating problem for this paper is to construct the HP dual of $\Sym^{2}\PP(V)$, with respect to the natural map $f : \PP(\Sym^{2}\PP(V)) \to \PP(\Sym^{2}V)$ and a Lefschetz decomposition of $D^{b}(\Sym^{2}\PP(V))$ which we describe as follows.

We think of sheaves on $\Sym^2\PP(V)$ as $\ZZ_{2}$-equivariant sheaves on $\PP(V)^{2}$.
For any distinct $i, j \in \ZZ$, there is a unique $\ZZ_{2}$-equivariant sheaf whose underlying sheaf on $\PP(V)^{2}$ is $\cO(i,j) \oplus \cO(j,i)$.
For any $i$, there are two $\ZZ_{2}$-equivariant structures on $\cO(i,i)$.
We let $\cO(i,i)_{+}$ be the $\ZZ_{2}$-structure such that the $\ZZ_{2}$-action is trivial along the diagonal in $\PP(V)^{2}$, and let $\cO(i,i)_{-}$ be the other one.
Note that then $\cL = f^{*}(\cO_{\PP(\Sym^{2}V)}(1)) = \cO_{\Sym^{2}\PP(V)}(1,1)_{+}$.

We take the initial piece in our Lefschetz decomposition of $D^{b}(\Sym^{2}\PP(V))$ to be
\[
\cA_{0} = \langle \cO(0,0)_{+}, \cO(0,0)_{-}, \{\cO(i,j) \oplus \cO(j,i)\}_{(i,j) \in S} \rangle,
\]
where $S = \{(i,j)\ |\ i + j \in [0,1], i > j, i - j \le \lfloor \frac{n}{2} \rfloor \}$.\footnote{The choice of $S$ is somewhat arbitrary. We could equally well have chosen $S$ to be any set $(i_{1}, j_{1}), \ldots, (i_{\lfloor n/2 \rfloor}, j_{\lfloor n/2 \rfloor})$ satisfying $(i_{1},j_{1}) = (-1,0)$ or $(0,1)$, and for each $k$ either $(i_{k},j_{k}) = (i_{k-1} - 1, j_{k-1})$ or $(i_{k},j_{k}) = (i_{k-1}, j_{k-1}+1)$. The same results would hold, and we choose this particular $S$ because it simplifies the combinatorics of some of the arguments in Section \ref{sec:WindowCategories}.}
If $n$ is odd, we take $\cA_{i} = \cA_{0}$ for all $i \in [0, n-1]$.

If $n$ is even, we let $\cA_{0} = \cA_{1} = \cdots = \cA_{n/2-1}$.
We remove 1 element from $S$ to get $S^{\pr} = \{(i,j)\ |\ i + j \in [0,1], i > j, i - j \le \frac{n}{2} - 1\}$.
We let
\[
\cA^{\pr} = \langle \cO(0,0)_{+}, \cO(0,0)_{-}, \{\cO(i,j) \oplus \cO(j,i)\}_{(i,j) \in S^{\pr}} \rangle,
\]
and then let $\cA_{n/2} = \cdots = \cA_{n-1} = \cA^{\pr}$.

By Proposition \ref{thm:HPDualityDecomposition}, this gives a Lefschetz decomposition
\[
D^{b}(\Sym^{2} \PP(V)) = \langle \cA_{0}, \cA_{1}(1), \ldots, \cA_{n-1}(n-1) \rangle
\]
in both the even and the odd case.
Let $X_{L^{\perp}} = f^{-1}(\PP(L^{\perp}))$, which is denoted by $X$ in Theorem \ref{thm:MainTheorem}.
Theorem \ref{thm:MainTheorem} and the computation of the orthogonal complements in Proposition \ref{thm:HPDualityDecomposition} then shows that we have
\[
D^{b}(X_{L^{\perp}}) = \langle \cC_{L^{\perp}}, \cA_{l}(l), \cdots, \cA_{n-1}(n-1) \rangle,
\]
with $\cC_{L^{\perp}} = \cW_{+} \cap \cW_{-,\res}$, which is a fully faithful subcategory of $D^{b}(\cY_{L},C)_{\res}$.
If $D^{b}(\cY,C)_{\res}$ were the HP dual of $\Sym^{2}\PP(V)$, this is in accordance with what Theorem \ref{thm:HPD} would give.
In view of this and the similar results obtained in \cite{ballard_homological_2013}, it seems very likely that $D^{b}(\cY,C)_{\res}$ is the correct HP dual, though strictly speaking we do not prove this here.

\subsubsection{Geometric interpretation of the HP dual}
Our Proposition \ref{thm:InterpretationOfCliffordNOdd} can be rephrased as saying that when $n = \dim V$ is odd, then away from the corank $\ge 3$ locus the HP dual is a double cover of the corank 1 locus in $\PP(\Sym^{2}V^{\vee})$, ramified in the corank 2 locus.
Similarly, Proposition \ref{thm:InterpretationOfCliffordNEven} says that when $n$ is even, the HP dual is a double cover ramified in the corank 1 locus.

Let us show concretely what this means in the case where $n$ is odd.
First of all, if $H \subset \Sym^{2}\PP(V)$ is such that $X_{H}$ is a nonsingular bilinear divisor, then we have
\[
D^{b}(X_{H}) = \langle \cA_{1}(1), \ldots, \cA_{n}(n) \rangle,
\]
i.e.\ the interesting part $\cC_{H}$ is trivial.

Correspondingly, if $X_{H}$ is of corank 1, we would like to say that $\cC_{H}$ corresponds to the derived category of 2 points.
This is almost correct, but must be modified slightly because the fibre $Y_{H}$ of the double cover $Y \to \PP(\Sym^{2}V^{\vee})$ has higher dimension than expected.
The correct statement is that $\cC_{H}$ is the derived category of the derived fibre product of $0 \into \AA^{1}$ and $(0 \sqcup 0) \to \AA^{1}$.

A somewhat surprising aspect of our description is that our HP dual is globally a variety, and that there is no need for an Azumaya algebra or Brauer class as in the case of quadrics.
One way of thinking about this is that in the quadric case the spinor bundles, which are point-wise generators for the category of the HP dual, do not extend to globally defined bundles, and this can be explained by the presence of a Brauer twist.
In our case, it turns out that we can write down an explicit global object which locally generates the HP dual category; this is the object called $K$ in Section \ref{sec:GeometricInterpretation}.

\section{Factorisation categories}
\label{sec:Factorisations}
We review some background material on derived categories of factorisations -- further details can be found in \cite{addington_pfaffian-grassmannian_2014, ballard_variation_2012, shipman_geometric_2012}.
We first fix a definition of a gauged Landau--Ginzburg B-model (LG model for short).
\begin{ndefn}
A gauged LG model is the data of a smooth quasi-projective variety $X$, equipped with:
\begin{itemize}
\item An action of a reductive group $G$.
\item An action of a 1-dimensional torus $\CC_{R}^{*}$, commuting with the $G$-action.  
\item An element $g \in G$ such that $g^{2} = e$ and $(g,-1) \in G \times \CC^{*}_{R}$ fixes $X$.\footnote{Assuming the action of $G$ is faithful, which is the case in our examples, the choice of such a $g$ is unique, and we will not mention it further.}
\item A function $W$, which is $G$-invariant and has weight 2 with respect to the $\CC^{*}_{R}$-action, i.e.\ $W(t_{R}x) = t_{R}^{2}W(x)$ for $x \in X$ and $t_{R} \in \CC^{*}_{R}$.
\end{itemize}
\end{ndefn}
Let $\cX = X/(G \times \CC^{*}_{R})$.
The canonical character $t_{R}$ of $\CC^{*}_{R}$ induces a line bundle on $\cX$, which we denote $\cO_{\cX}[1]$.
For a sheaf $\cE$ on $\cX$ we write $\cE[l]$ for $\cE \otimes \cO_{\cX}[1]^{\otimes l}$.
Note that $W$ is a section of $\cO_{\cX}[2]$.

By work of Positselski and Orlov \cite{positselski_two_2011, efimov_coherent_2015, orlov_matrix_2012}, we can define a derived category of factorisations, $D(\cX,W)$, from the above data.
An object of this category is a quasi-coherent sheaf $\cE$ on $\cX$, equipped with a differential map $d : \cE \to \cE[1]$, satisfying $d^{2} = W$.
We call such an object a factorisation.
If we wish to emphasise the choice of differential, we denote this object by $(\cE,d)$, otherwise we will simply write $\cE$.

\begin{nex}
Consider the case where $X = \Spec A$ and $G$ is trivial.
Then the action of $\CC^{*}_{R}$ on $X$ induces a grading on $A$ and makes it a dg algebra with vanishing differential.
Since we require that $-1 \in \CC^{*}_{R}$ acts trivially on $X$, this grading will be even. 
Thus, $A$ is commutative as a dg algebra.
A factorisation on $\cX$ is in this case the same thing as a graded $A$-module $M$ with a differential $d : M \to M[1]$ squaring to $W$.
In particular, if $W = 0$, then a factorisation is the same thing as a dg module over $A$.
\end{nex}

If $\cE$ is a factorisation, we let $\cE[l]$ be the factorisation whose underlying sheaf is $\cE[l]$ and whose differential is $(-1)^{l}d[l]$.
Given two factorisations $\cE, \cF$ we have a graded vector space 
\[
\Hom(\cE,\cF) = \oplus_{i} \Hom_{\cX}(\cE, \cF[i]).
\]
The differentials $d_{\cE}$ and $d_{\cF}$ give a differential on $\Hom(\cE,\cF)$ by the usual Leibniz rule. 
This differential squares to 0, and so $\Hom(\cE,\cF)$ is a dg vector space.
We denote the homotopy category of the resulting dg category by $K(\cX,W)$.

The category $K(\cX,W)$ is triangulated, with the shift functor $[1]$ as described above.
The cone over $\cE \to \cF$ is $\cF[1] \oplus \cE$ with an induced differential, in the same way as for the usual homotopy category of complexes.

In analogy with the definition of the ordinary derived category, we should now take the Verdier quotient of $K(\cX,W)$ with respect to the subcategory of acyclic complexes.
Since the differentials of factorisations do not square to 0, they do not have a notion of cohomology, and so the usual definition of acyclic does not make sense.

The correct definition of acyclic in this setting is the following:
Consider a finite exact complex of factorisations
\[
\cE_{1} \to \cE_{2} \cdots \to \cE_{n}.
\]
Exactness is here defined by considering the underlying sheaves, and we require the maps to be closed with respect to the differentials on $\Hom(\cE_{i}, \cE_{i+1})$.
One can form the so-called totalisation $\mathrm{Tot}(\cE_{\bullet})$ of the above complex, which is a factorisation (see e.g. \cite[2.12]{shipman_geometric_2012}).
We declare $\mathrm{Tot}(\cE_{\bullet})$ to be acyclic, and let the category of acyclic objects be the thick triangulated subcategory of $K(\cX,W)$ generated by such totalisations.
Taking the Verdier quotient of $K(\cX,W)$ with respect to the subcategory of acyclic objects gives the derived category $D(\cX,W)$.

\subsection{Coherent and locally free factorisations}
We say a factorisation is coherent if the underlying sheaf is.
We define the category $D^{b}(\cX,W) \subset D(\cX,W)$ to be the full subcategory of objects isomorphic to coherent factorisations.
The category $D^{b}(\cX,W)$ is a generalisation of the usual bounded derived category, which is the special case where $W = 0$:
\begin{nprop}[\cite{ballard_homological_2013}, 2.1.6]
\label{thm:StandardDerivedCategory}
If $\CC^{*}_{R}$ acts trivially on $X$, then 
\[
D^{b}(\cX,0) \cong D^{b}(X/G).
\]
\end{nprop} 
We say a factorisation is locally free if the underlying sheaf is.
\begin{nprop}[\cite{ballard_category_2014}, 3.14]
Every factorisation on $\cX$ is isomorphic in $D(\cX,W)$ to a locally free factorisation.
Every coherent factorisation on $\cX$ is isomorphic in $D^{b}(\cX,W)$ to a finite rank locally free factorisation.
\end{nprop}

We record the following lemma, which gives a useful criterion for checking that a complex is acyclic:
\begin{nlemma}[\cite{shipman_geometric_2012}, 2.12]
\label{thm:BoundedsGenerate}
If $\cF \in D(\cX,W)$ and $\Hom(\cE, \cF) = 0$ for all $\cE \in D^{b}(\cX,W)$, then $\cF = 0$.
\end{nlemma}

\subsection{Functors}
Suppose we are given a map of LG models $f : (\cX,W_{\cX}) \to (\cY,W_{\cY})$, i.e.\ a morphism of stacks $f : \cX \to \cY$ such that $f^{*}\cO_{\cY}[1] \cong \cO_{\cX}[1]$, and such that $f^{*}W_{\cY} = W_{\cX}$.

For any factorisation $\cE$, the pushforward $f_{*}\cE$ of its underlying sheaf inherits a differential map which squares to $W_{\cY}$, and so becomes a factorisation on $(\cY,W_{\cY})$.
This defines a pushforward functor $f_{*} : K(\cX,W_{\cX}) \to K(\cY,W_{\cY})$, which admits a right derived functor
\[
Rf_{*} : D(\cX,W_{\cX}) \to D(\cY,W_{\cY}),
\]
see \cite[3.36]{ballard_category_2014}.
We can compute the pushforward functor by replacing a factorisation $\cE$ with a quasi-isomorphic injective factorisation $\cI$.
In general $Rf_{*}$ does not send coherent factorisations to coherent factorisations, but in the special case where $f$ is a closed immersion this is true, as the underived functor is then exact.

We similarly get a functor $f^{*}$, which we may left derive by taking locally free replacements to get a functor
\[
Lf^{*} : D(\cY,W_{\cY}) \to D(\cX,W_{\cX}).
\]
This functor clearly sends $D^{b}(\cY,W_{\cY})$ to $D^{b}(\cX,W_{\cX})$.

We also have a tensor product.
If $\cE$ is a factorisation on $(\cX,W_{1})$ and $\cF$ a factorisation on $(\cX,W_{2})$, we may equip the sheaf $\cE \otimes \cF$ with a differential which squares to $W_{1} + W_{2}$, and so becomes a factorisation on $(\cX,W_{1} + W_{2})$.

To be precise, the differential is the following:
We have assumed that there exists a 2-torsion element $g \in G$ such that $(g,-1) \in G \times \CC^{*}_{R}$ acts trivially on $X$.
As a consequence, any sheaf $\cE$ on $\cX$ splits canonically into eigensheaves $\cE_{+} \oplus \cE_{-}$, where $\cE_{\pm}$ is the subsheaf on which $(g,-1) \in G \times \CC^{*}_{R}$ acts by $\pm 1$.
We can then equip $\cE \otimes \cF$ with the differential which is $d_{\cE} \otimes 1 + 1 \otimes d_{\cF}$ on $\cE_{+} \otimes \cF$ and which is $d_{\cE} \otimes 1 - 1 \otimes d_{\cF}$ on $\cE_{-} \otimes \cF$.
This is essentially the standard sign rule for the tensor product of dg objects.

Replacing either $\cE$ or $\cF$ by a locally free factorisation, we obtain a derived tensor product
\[
- \otimes^{L}_{\cX} - : D(\cX,W_{1}) \times D(\cX,W_{2}) \to D(\cX,W_{1} + W_{2}).
\]

Given factorisations $\cE, \cF$ on $(\cX,W)$ with $\cE$ coherent, we have a sheaf hom $\sHom(\cE,\cF)$.
This is the usual sheaf hom with the induced differential, where we use the standard sign rule together with the splitting of $\cE$ and $\cF$ into even and odd graded parts.
The differential on $\sHom(\cE,\cF)$ satisfies $d^{2} = 0$.
We may derive this to get $\Rhom(\cE,\cF) \in D(\cX,0)$.
The derived sheaf hom can be computed either by taking an injective replacement of $\cF$ or a locally free replacement of $\cE$.
We have an isomorphism $\Rhom(\cE,\cF) \cong \cF \otimes^{L} \cE^{\vee}$, where $\cE = \Rhom(\cE, \cO_{\cX})$.

\subsection{Resolutions}
\label{sec:Resolutions}
A sheaf $\cF$ on $\cX$ supported on $\{W = 0\}$, equipped with the trivial differential, is a factorisation on $\cX$.
This provides a useful supply of objects in $D^{b}(\cX,W)$.
Lemma \ref{thm:Resolutions} gives a way of constructing a locally free representative of such an $\cF$.

As a matter of notation, we write
\begin{equation}
\label{eqn:ResolutionNotation}
\cE_{n} \mfarrows{d_{r}}{d_{l}} \cdots \mfarrows{d_{r}}{d_{l}} \cE_{1}  \mfarrows{d_{r}}{d_{l}} \cE_{0}
\end{equation}
to mean the factorisation $(\cE,d)$, where $\cE = \oplus \cE_{i}$ and $d = d_{r} + d_{l}$.
Note that the arrows $d_{l}, d_{r}$ in \eqref{eqn:ResolutionNotation} are then maps of degree $1$ with respect to the $\CC^{*}_{R}$-action.

\begin{nlemma}
\label{thm:Resolutions}
Let
\[
\cE = \cE_{n} \mfarrows{d_{r}}{d_{l}} \cdots \mfarrows{d_{r}}{d_{l}} \cE_{0}
\]
be a factorisation, and suppose that there is a map of sheaves $\cE_{0} \stackrel{f}{\to} \cF$, such that the sequence
\[
0 \to \cE_{n} \stackrel{d_{r}}{\to} \cdots \stackrel{d_{r}}{\to} \cE_{0} \stackrel{f}{\to} \cF \to 0
\]
is exact.
Then $\cF$ is scheme-theoretically supported on $\{W = 0\}$.

Thinking of $\cF$ as a factorisation on $(\cX,W)$ with trivial differential, the induced map of factorisations $\cE \to \cE_{0} \stackrel{f}{\to} \cF$ is a quasi-isomorphism.
\end{nlemma}
\begin{proof}
For the first claim, note that as $d_{r}d_{l} = W\id : \cE_{0} \to \cE_{0}$, we have $W\cE_{0} \subseteq \im d_{r}$.
For the second claim, apply \cite[3.4]{ballard_resolutions_2012} to the complex $\cE_{n} \stackrel{d_{r}}{\to} \cdots \stackrel{d_{r}}{\to} \cE_{0} \stackrel{f}{\to} \cF$.
\end{proof}

\subsection{Change of $R$-grading}
\label{sec:ChangeOfRGrading}
For the purpose of constructing $D(\cX,W)$ and $D^{b}(\cX,W)$, the definition of gauged LG model that we use contains some superfluous information, because the splitting of $\overline{G} = G \times \CC^{*}_{R}$ can be replaced with the choice of a surjection $\overline{G} \to \CC^{*}_{R}$.
Indeed, the only information used in the definition of $D(\cX,W)$ is the structure $\cX = X/\ol{G}$ as a stack and the line bundle $\cO_{\cX}[1]$.

We draw the following consequence.
Suppose that there is an automorphism $\sigma : \ol{G} \to \ol{G}$ commuting with the projection $G \times \CC_{R}^{*} \to \CC_{R}^{*}$.
Then we may replace the action $\rho : \ol{G} \to \Aut(X)$ by $\rho\sigma$, without modifying the category $D^{b}(\cX,W)$.
In particular, fixing the action of $G$ on $X$, different choices of $\CC^{*}_{R}$-action may give the same derived category $D^{b}(\cX,W)$.

Consider, for instance, the LG model where $X = \CC^{n}\setminus 0$, $G = \CC^{*}$, $W = 0$, and both $G$ and $\CC^{*}_{R}$ act by the usual multiplication.
Then by the above remarks we may instead take the same model with trivial $\CC^{*}_{R}$-action, without changing the category of factorisations.
Hence by Proposition \ref{thm:StandardDerivedCategory} we have $D^{b}(\cX,0) \cong D^{b}(\PP^{n-1})$.

We will use this flexibility to choose different $\CC^{*}_{R}$-actions at various points throughout the proof.
The choice of $\sigma \in \Aut(\ol{G})$ will in our case be unique, so that the categories corresponding to different choices are canonically equivalent.

\subsection{Notational abuse}
From the next section onwards we will drop the $\CC^{*}_{R}$-action from the notation and denote the LG model simply by $(X/G,W)$, and the category of factorisations by either $D^{b}(X/G,W)$ or $D_{G}^{b}(X,W)$.
As justification for this abuse, we offer Proposition \ref{thm:StandardDerivedCategory}, whose conclusion then has the natural form $D^{b}(X/G,0) = D^{b}(X/G)$.

\section{The GIT quotients}
\label{sec:GIT}
We now turn to the geometry of our examples.
Fix a vector space $V$ of dimension $n$ and a vector subspace $L \subset \Sym^{2}(V^{\vee})$ of dimension $l$.
We let $Z = V \times V \times L$.

Let $T$ be the group $(\CC^{*})^{2}$ with coordinates $t_{1}^{\pm 1}, t_{2}^{\pm 1}$, and let $G = T \rtimes \ZZ_{2}$, where the semi-direct product is given by the involution of $T$ which permutes the $t_{i}$.
We let $G$ act on $Z$ in such a way that $T$ acts via characters $t_{1}, t_{2}, t_{1}^{-1}t_{2}^{-1}$ on $V$, $V$, and $L$, respectively, and such that the $\ZZ_{2}$-factor of $G$ permutes the $V$ factors and fixes the $L$ factor.

Let $\CC^{*}_{R}$ act on $Z$ by scaling the $L$ factor by $t_{R}^{2}$.
There is a natural superpotential $W : Z \to \CC$, which at a point $(v_{1},v_{2},l)$ is the evaluation of $l \in \Sym^{2}(V^{\vee})$ on $(v_{1},v_{2})$.
We let $\cZ = Z/G$.
The above data defines for us a gauged LG model $(\cZ,W)$.

Consider the GIT problem posed by $Z/G$.
Let $\chi$ be the character of $G$ which restricts to $t_{1}t_{2}$ on $T$, and which is trivial on the $\ZZ_{2}$-factor.
By choosing either a positive or negative multiple of $\chi$ for our GIT linearisation we obtain two GIT quotients $\cZ_{+}$ and $\cZ_{-}$.

Note that for us a GIT quotient is the quotient stack $Z^{\sstable}/G$, as distinguished from the GIT quotient in the classical sense, which is the coarse moduli scheme of this stack.

\subsection{The positive GIT quotient}
Choosing $\chi$ as our linearisation, the unstable locus is 
\[
Z_{+}^{\us} = (0 \times V \times L) \cup (V \times 0 \times L).
\]
Recall that by $\cO_{\Sym^{2}\PP(V)}(-1,-1)_{+}$ we mean the $\ZZ_{2}$-equivariant line bundle $\cO_{\PP(V)^{2}}(-1,-1)$ on $\PP(V)^{2}$, equipped with the unique lifting of the $\ZZ_{2}$-action which leaves the restriction of $\cO_{\PP(V)^{2}}(-1,-1)$ to the diagonal of $\PP(V)^{2}$ fixed.
The GIT quotient $\cZ_{+} = (Z \setminus Z_{+}^{\us})/G$ is then the stacky vector bundle 
\[
p : \cO_{\Sym^{2}\PP(V)}(-1,-1)_{+}^{\oplus l} \to \Sym^2\PP(V).
\]

We get a superpotential $W$ on $\cZ_{+}$ by restriction from $\cZ$.
In each fibre of $p$, the function $W$ is linear, and so dually gives a section 
\[
s \in \Gamma(\Sym^2\PP(V), \cO(1,1)_{+}^{\oplus l}).
\]
We let $X \subset \Sym^{2}\PP(V)$ be defined by the vanishing of this $s$.
Equivalently, $X$ is defined by the Cartesian diagram
\[
\begin{tikzcd}
X \arrow{d} \arrow{r} & \PP(L) \arrow{d} \\
\Sym^{2}\PP(V) \arrow{r} & \PP(\Sym^{2} V),
\end{tikzcd}
\]
and therefore corresponds to the $X$ in Theorem \ref{thm:MainTheorem}.

Let $i$ denote the inclusion $p^{-1}(X) \into \cZ_{+}$.
\begin{nprop}
\label{thm:KnorrerPeriodicity}
If $X$ has the expected dimension, then the functor
\[
i_{*} \circ p^{*} : D^{b}(X) \to D^{b}(\cZ_{+}, W)
\]
is an equivalence.
\end{nprop}
\begin{proof}
This result is known as \Knorrer{} periodicity, and has been proved independently by Isik \cite[4.6]{isik_equivalence_2013} and Shipman \cite[3.4]{shipman_geometric_2012}, see also \cite[2.3.11]{ballard_variation_2012}.
\end{proof}

\subsubsection{An alternative \Knorrer{} functor}
Let $K$ denote the functor $i_{*}\circ p^{*}$ from Proposition \ref{thm:KnorrerPeriodicity}.
If $\cE \in D^{b}(X)$ is restricted from $\Sym^{2}\PP(V)$, then we may describe $K(\cE)$ differently, a result which will be needed in Section \ref{sec:OrthogonalComplement}.

Let $j_{1} : X \to \Sym^{2}\PP(V)$ be the inclusion, and let $j_{2} : \Sym^{2}\PP(V) \to \cZ_{+}$ be the inclusion along the 0-section of $p$.
Let the functors $\Phi_{1}, \Phi_{2} : D^{b}(\Sym^{2}\PP(V)) \to D^{b}(\cZ_{+},W)$ be given by 
\[
\Phi_{1} = K \circ j_{1}^{*}
\] 
and 
\[
\Phi_{2} = (j_{2})_{*}(- \otimes \cO(-1,-1)_{+}^{\otimes l}[l]).
\]
\begin{nlemma}
\label{thm:FunctorsEquivalent}
Assume that $X$ has the expected codimension. 
Then $\Phi_{1}$ and $\Phi_{2}$ are equivalent.
\end{nlemma}
\begin{proof}
We ignore the cohomological shifts for notational ease. We have 
\[
\Phi_{1}(\cE) = i_{*}p^{*}j_{1}^{*}(\cE) = p^{*}(\cE) \otimes \cO_{p^{-1}(X)},
\]
where $p^{*}(\cE) \in D^{b}(\cZ,0)$ and $\cO_{p^{-1}(X)} \in D^{b}(\cZ,W)$.
Next we have
\[
\Phi_{2}(\cE) = (j_{2})_{*}(\cE \otimes \cO(-1,-1)_{+}^{\otimes l})) = p^{*}(\cE) \otimes \cO_{\Sym^{2}\PP(V)} \otimes p^{*}(\cO(-1,-1)_{+}^{\otimes l}),
\]
with $p^{*}(\cE), p^{*}(\cO(-1,-1)_{+}^{\otimes l}) \in D^{b}(\cZ,0)$ and $\cO_{\Sym^{2}\PP(V)} \in D^{b}(\cZ_{+},0)$.

We will show that $\cO_{p^{-1}(X)} \cong \cO_{\Sym^{2}\PP(V)} \otimes p^{*}(\cO(-1,-1)_{+}^{\otimes l})$.
Let us first work on $\cZ$.
Choose coordinates $p_{1}, \ldots, p_{l}$ on $L$.
Then the potential $W$ decomposes as $\sum p_{i}f_{i}$, where the $f_{i}$ are symmetric bilinear forms on $V \times V$.
There is a factorisation
\begin{align*}
\cE &= p^{*}(\cO(-1,-1)_{+})^{\otimes l} \mfarrows{d_{r}}{d_{l}} \wedge^{l-1}\left(p^{*}(\cO(-1,-1)_{+})^{\oplus l}\right) \mfarrows{d_{r}}{d_{l}} \cdots \\
& \mfarrows{d_{r}}{d_{l}} \wedge^{2}\left(p^{*}(\cO(-1,-1)_{+})^{\oplus l}\right) \mfarrows{d_{r}}{d_{l}} p^{*}(\cO(-1,-1)_{+})^{\oplus l} \mfarrows{d_{r}}{d_{l}} \cO_{\cZ},
\end{align*}
where the $d_{r}$ are contractions with $(p_{i}) \in \Gamma(p^{*}(\cO(-1,-1)_{+})^{\oplus l})$ and the $d_{l}$ are exterior multiplications with $(f_{i}) \in \Gamma(p^{*}(\cO(-1,-1)_{+})^{\oplus l})$ (see \cite[Sec\ 2.3]{dyckerhoff_compact_2011}).
Restricting $\cE$ to $\cZ_{+}$, we see that since $\Sym^{2}\PP(V)$ is cut out by the $p_{i}$, since $p^{-1}(X)$ is cut out by the $f_{i}$, and since both of these have codimension $l$, it follows from Lemma \ref{thm:Resolutions} that
\[
\cO_{p^{-1}(X)} \cong \cE|_{\cZ_{+}} \cong \cO_{\Sym^{2}\PP(V)} \otimes p^{*}(\cO(-1,-1)^{\otimes{l}}_{+}).
\]
\end{proof}

\subsection{The negative GIT quotient}
\label{sec:NegativeGITQuotient}
If we take the character $\chi^{-1}$ as our linearisation, then the unstable locus is 
\[
Z_{-}^{\us} = V \times V \times 0.
\]
We let $Z_{-}^{\mathrm{ss}} = Z \setminus Z_{-}^{\us}$ and let $\cZ_{-} = (Z_{-}^{\mathrm{ss}}) / G$.

When dealing with $\cZ_{-}$ and its substacks, we will assume the $\CC^{*}_{R}$-action on $Z$ is the one which scales the $V$-factors by $t_{R}$ and leaves $L$ fixed.
The remarks in Section \ref{sec:ChangeOfRGrading} show that the category $D^{b}(\cZ_{-},W)$ is the same for this $\CC^{*}_{R}$-action as for the one described before.

The projection $Z_{-}^{\mathrm{ss}} \to L \setminus 0$ and the character $\chi: G \to \CC^{*}$ together give a map of quotient stacks
\[
f : \cZ_{-} = Z_{-}^{\mathrm{ss}}/G \to (L \setminus 0)/\CC^{*} = \PP(L).
\]
The kernel of $\chi$ is isomorphic to $O(2)$, and each fibre of $f$ is isomorphic to $V \times V/O(2)$.
The action of $O(2)$ on $V \times V$ is the one given by identifying $V \times V$ with $V \otimes \CC^{2}$, where $\CC^{2}$ is the standard representation of $O(2)$.
In particular each fibre of $f$ contains a point with positive-dimensional stabiliser group, so that $\cZ_{-}$ is not a Deligne--Mumford stack.

\subsubsection{Grade restricted objects}
\label{sec:GradeRestrictedObjects}
Because $\cZ_{-}$ contains points with positive-dimensional isotropy groups, the category $D^{b}(\cZ_{-},W)$ is too big to be directly compared with $D^{b}(\cZ_{+},W)$ in the way we would like.
For example, for each stacky point $p$ and each representation $\rho$ of $O(2)$, we have an object $\cO_{p}(\rho) \in D^{b}(\cZ_{-},W)$.
It will therefore be useful to consider a full subcategory of $D^{b}(\cZ_{-},W)$, defined as follows.

Choose a point $p \in \PP(L)$, and let $\CC^{*}$ be the connected component of its isotropy group $O(2)$ in $\cZ_{-}$.
The category $D^{b}_{\CC^{*}}(p)$ splits as $\oplus_{k} \langle \cO_{p}(k) \rangle$.
If $\cE \in D^{b}_{\CC^{*}}(p)$, then we let $\cE_{k}$ denote the projection of $\cE$ to $\langle \cO(k) \rangle$.

The inclusion $i_{p}: p/\CC^{*} \to \cZ_{-}$ induces a restriction functor
\[
i_{p}^{*} : D^{b}(\cZ_{-},W) \to D^{b}_{\CC^{*}}(p).
\]
For $\cE \in D^{b}(\cZ_{-},W)$, we let the ``weights of $\cE$ at $p$'' be the set $\{k\ |\ (i_{p}^{*}\cE)_{k} \not= 0\}$.
We say $\cE$ is \emph{grade restricted} if at every $p$ the weights of $\cE$ are contained in the interval $[-\lfloor n/2\rfloor, \lfloor n/2 \rfloor]$, where $n = \dim V$.
We denote the full subcategory of grade restricted objects in $D^{b}(\cZ_{-},W)$ by $D^{b}(\cZ_{-},W)_{\res}$.

\section{Window categories}
\label{sec:WindowCategories}
Consider the origin $0 \in Z$, which has isotropy group $G$, and recall that $T = (\CC^{*})^{2}$ is the identity component of $G$.
The category $D^{b}_{T}(0)$ splits as
\[
D^{b}_{T}(0) = \oplus_{(i,j) \in \ZZ^{2}} \langle \cO(i,j) \rangle.
\]
For any $\cE \in D^{b}_{T}(0)$, we let $\cE_{(i,j)}$ be the projection of $\cE$ to $\langle \cO(i,j) \rangle$.

\begin{ndefn}
\label{defi:NewWeightsAndWidth}
Let $\cE \in D^{b}_{T}(Z,W)$. The weights of $\cE$ are defined as
\[
\wt (\cE) = \{(i,j)\ |\ (\cE|_{0})_{i,j} \not \cong 0 \}.
\]
If $\cE \in D^{b}_{G}(Z,W)$, we let $\wt(\cE)$ be the weights of $\cE$ considered as a $T$-equivariant object.
\end{ndefn}

Recall that $n = \dim V$ and $l = \dim L$. 
We define subsets $S_{+}, S_{-}$ and $S_{-, \text{res}}$ of $\ZZ^{2} = \chi(T)$ as follows:
\begin{itemize}
\item If $n$ is odd, then $S_{+}$ is the set of pairs $(i,j)$ such that $0 \le i + j \le 2n-1$ and $|i-j| \le (n-1)/2$.
\item If $n$ is even, then $S_{+}$ is the set of pairs $(i,j)$ such that either $0 \le i + j < n$ and $|i-j| \le n/2$, or $n \le i + j \le 2n - 1$ and $|i - j| \le n/2-1$.
\item $S_{-}$ is the set of pairs $(i,j)$ such that $0 \le i + j \le 2l-1$.
\item $S_{-, \text{res}}$ is the set of pairs $(i,j)$ such that $0 \le i + j \le 2l - 1$ and $|i - j| \le \lfloor n/2 \rfloor$.
\end{itemize}
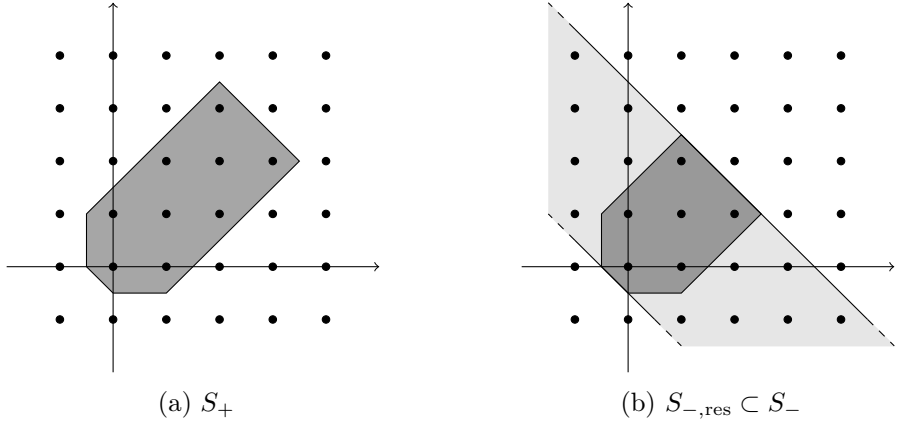
\begin{figure}[h]
\centering
\begin{subfigure}{0.4\textwidth}
\centering
\begin{tikzpicture}[scale = 0.7]
\draw[fill = gray!70] (-0.5, 1) -- (-0.5,0) -- (0, -0.5) -- (1, -0.5) -- ++(2.5, 2.5) -- ++(-1.5, 1.5) -- cycle;
\foreach \x in {-1,..., 4} {\foreach \y in {-1,..., 4} {\filldraw (\x,\y) circle[radius = 2pt];}};
\draw[->] (0,-2) -- (0,5);
\draw[->] (-2,0) -- (5,0);
\end{tikzpicture}
\caption{$S_{+}$}
\end{subfigure}
\qquad
\begin{subfigure}{0.4\textwidth}
\centering
\begin{tikzpicture}[scale = 0.7]
\filldraw[gray!20] (-1.5,5) -- (5, -1.5) -- (1, -1.5) -- (-1.5, 1);
\draw[fill = gray!80] (-0.5, 1) -- (-0.5,0) -- (0, -0.5) -- (1, -0.5) -- ++(1.5, 1.5) -- ++(-1.5, 1.5) -- cycle;
\draw[->] (0,-2) -- (0,5);
\draw[->] (-2,0) -- (5,0);
\draw (-1, 0.5) -- (0.5, -1);
\draw[dashed] (-1, 0.5) -- +(-0.5, 0.5);
\draw[dashed] (0.5, -1) -- +(0.5, -0.5);
\draw (-1, 4.5) -- (4.5, -1);
\draw[dashed] (-1, 4.5) -- +(-0.5, 0.5);
\draw[dashed] (4.5, -1) -- +(0.5, -0.5);

\foreach \x in {-1,..., 4} {\foreach \y in {-1,..., 4} {\filldraw (\x,\y) circle[radius = 2pt];}};
\end{tikzpicture}
\caption{$S_{-,\res} \subset S_{-}$}
\end{subfigure}
\caption{$S_{+}, S_{-}$ and $S_{-,\res}$ for $n = 3, l = 2$}
\label{fig:SPlusMinus}
\end{figure}

We define ``window categories'' inside $D^{b}(\cZ,W)$ as follows:
Let $\cW_{+}$ (resp.\ $\cW_{-}$, $\cW_{-, \res}$) be the full subcategory of $D^{b}(\cZ,W)$ consisting of all objects such that $\wt(\cE) \subseteq S_{+}$ (resp.\ $S_{-}$, $S_{-, \res}$).

We let $j_{\pm} : \cZ_{\pm} \to \cZ$ be the inclusions.
The restriction functors $j_{\pm}^{*}:D^{b}(\cZ,W) \to D^{b}(\cZ_{\pm},W)$ give functors 
\[
\Phi_{\pm}:\cW_{\pm} \into D^{b}(\cZ,W) \stackrel{j_{\pm}^{*}}{\to} D^{b}(\cZ_{\pm},W).
\]
The main result of this section is:
\begin{nprop}
\label{thm:windowsAsGIT}
The functors $\Phi_{\pm} : \cW_{\pm} \to D^{b}(\cZ_{\pm},W)$ are equivalences.
\end{nprop}
\begin{proof}
The case of $\Phi_{-}$ is an application of \cite[3.3.2]{ballard_variation_2012} and \cite[3.29]{halpern-leistner_derived_2015}.
The proof of the claim for $\Phi_{+}$ occupies the rest of this section, and is split into lemmas as follows.

Let $\cE \in D^{b}(\cZ_{+}, W)$.
By Lemma \ref{thm:restrictionIsSurjective} there is an $\widehat{\cE} \in D^{b}(\cZ,W)$ such that $j_{+}^{*}(\widehat{\cE}) = \cE$.
By Proposition \ref{thm:NewWindowsAndKernelsGenerate}, we may, after modifying $\widehat{\cE}$ by taking cones over maps to or from objects in $\ker j^{*}_{+}$, assume that $\widehat{\cE} \in \cW_{+}$, which means that $\Phi_{+}$ is essentially surjective. 
If $\cE_{1}, \cE_{2} \in \cW_{+}$, then they satisfy the assumptions of Lemma \ref{thm:restrictionIsIsomorphism}, and hence
\[
\RHom_{\cZ}(\cE_{1},\cE_{2}) = \RHom_{\cZ_{+}}(\cE_{1}|_{\cZ_{+}}, \cE_{2}|_{\cZ_{+}}).
\] 
It follows that $\Phi_{+}$ is fully faithful.
\end{proof}

\begin{ncor}
\label{thm:GITCorollary}
The functor $\Phi_{-}$ restricts to give an equivalence 
\[
\cW_{-,\res} \cong D^{b}(\cZ_{-},W)_{\res}.
\]
\end{ncor}
\begin{proof}
Let $\cE \in D^{b}(\cZ_{-},W)$.
By Proposition \ref{thm:windowsAsGIT}, there is a unique $\wh{\cE} \in \cW_{-}$ such that $j_{-}^{*}(\wh{\cE}) = \cE$.

Assume first that $\wh{\cE} \in \cW_{-, \res}$, and choose a locally free representative of $\wh{\cE}$.
Remembering only the $T$-equivariant structure, it follows from \cite{masuda_equivariant_1996} that the underlying sheaf of $\wh{\cE}$ will have the form $\oplus_{i,j}\cO_{\cZ}(i,j)^{m_{ij}}$ (ignoring cohomological shifts).

On $L/T$, there is a vanishing of Hom spaces
\[
\Hom(\cO_{L}(i,j), \cO_{L}(i^{\pr},j^{\pr})) = 0 \ \ \ \ \text{if } i-j \not= i^{\pr} - j^{\pr},
\]
so we get a splitting $\wh{\cE}|_{L} = \oplus_{k}(\wh{\cE}|_{L})_{k}$, where
\[
(\wh{\cE}|_{L})_{k} = \bigoplus_{i-j = k} \cO_{L}(i,j)^{m_{ij}}
\]
with the induced differential.

Now if $|k| > \lfloor n/2 \rfloor$, then $(\wh{\cE}|_{L})_{k}|_{0}$ vanishes, since $\wh{\cE} \in \cW_{-, \res}$.
The support of the complex $(\wh{\cE}|_{L})_{k}$ is a closed, $T$-invariant subset of $L$.
Since it does not contain $0$ it is empty, and therefore $(\wh{\cE}|_{L})_{k} \cong 0$.

Let $p \in \PP(L)$, and recall that for $\cF \in D^{b}_{\CC^{*}}(p)$, we write $\cF_{k}$ for the projection of $\cF$ to $\langle\cO_{p}(k) \rangle$.
For $p \in \PP(L)$ we have in $D_{\CC^{*}}^{b}(p)$
\[
(\cE|_{p})_{k} \cong ((\wh{\cE}|_{L})_{k})|_{p} \cong 0,
\]
and therefore $\cE \in D(\cZ_{-},W)_{\res}$.

Assume now that $\wh{\cE} \not \in \cW_{-,\res}$.
Then there is a $k$ with $|k| > \lfloor n/2 \rfloor$ such that $(\wh{\cE}|_{L})_{k}|_{0} \not \cong 0$.
Let $T_{d}$ be the diagonal subtorus of $T$.
Then $T_{d}$ acts with weight 2 on $L$.
As a $T_{d}$-equivariant complex, the object $(\cE|_{L})_{k}$ is such that after restriction to $0 \in L$, it has weights in $[0, 2l-1]$.

Applying \cite[3.29]{halpern-leistner_derived_2015}, we see that such a complex is acyclic on $L \setminus 0$ if and only if it is acyclic on $L$, and hence $(\wh{\cE}|_{k})_{L\setminus 0} \not \cong 0$.
This means that there is a $p \in \PP(L)$ and a $k$ with $|k| > \lfloor n/2\rfloor$ such that $(\cE|_{p})_{k} \cong (\wh{\cE}|_{L})_{k}|_{p} \not \cong 0$, and so $\cE \not \in D^{b}(\cZ_{-},W)_{\res}$.
\end{proof}

\subsection{Essential surjectivity}
\label{sec:EssentialSurjectivity}
The goal of this section is to prove that $\Phi_{+} : \cW_{+} \to D^{b}(\cX_{+},W)$ is essentially surjective, which follows from Lemma \ref{thm:restrictionIsSurjective} and Proposition \ref{thm:NewWindowsAndKernelsGenerate}.

\begin{nlemma}
\label{thm:restrictionIsSurjective}
Let $(X/G, W)$ be a gauged LG model, and let $U \subset X$ be a $G \times \CC^{*}_{R}$-invariant open subvariety.
The restriction map $D^{b}(X/G,W) \to D^{b}(U/G,W)$ is essentially surjective.
\end{nlemma}
\begin{proof}
Let $\ol{G} = G \times \CC^{*}_{R}$, let $j$ be the inclusion $U/\ol{G} \into X/\ol{G}$, and let $\cE$ be a coherent factorisation on $U/\ol{G}$.
The pushforward $j_{*}(\cE)$ is a quasi-coherent factorisation on $X/{G}$.
The underlying ($\ol{G}$-equivariant) sheaf of $j_{*}(\cE)$ is the sum of its coherent subsheaves \cite[15.4]{laumon_champs_2000}.
Since $j^{*}j_{*}(\cE) = \cE$, there is therefore a coherent subsheaf $\cF$ of $j_{*}(\cE)$ such that $j^{*}(\cF) \to j^{*}j_{*}\cE \to \cE$ is an isomorphism.
Let $f$ be the composition $\cF[-1] \into j_{*}\cE[-1] \stackrel{d}{\to} j_{*}\cE$, and let $\ol{\cF} = \cF + \im f \subset j_{*}\cE$.
Then the differential on $\cE$ restricts to a differential on $\ol{\cF}$, and the map $j^{*}\ol{\cF} \to j^{*}j_{*}\cE \to \cE$ is an isomorphism of factorisations.
\end{proof}

Recall that $j_{+}$ denotes the inclusion $\cZ_{+} \into \cZ$.
We write $\ker j_{+}^{*}$ for the full subcategory of $D^{b}(\cZ,W)$ consisting of objects $\cE$ such that $j_{+}^{*}(\cE) \cong 0$.
\newcommand{\NewWindowsAndKernelsGenerate}{The subcategories $\cW_{+}$ and $\ker j_{+}^{*}$ generate $D^{b}(\cZ,W)$ as a triangulated category.}
\begin{nprop}
\label{thm:NewWindowsAndKernelsGenerate}
\NewWindowsAndKernelsGenerate{}
\end{nprop}
The proof of this proposition occupies the rest of Section \ref{sec:EssentialSurjectivity}; the main idea is the following.
If $\cE \in D^{b}(\cZ,W)$ and the weights of $\cE$ are not contained in $S_{+}$, then we can construct a factorisation $\cG$ supported on $\cZ \setminus \cZ_{+}$ (so that $\cG \in \ker j_{+}^{*}$), which admits a map to or from $\cE$ with cone $\cF$.
Setting up this correctly, we can ensure that the weights of $\cF$ are closer to being contained in $S_{+}$ than those of $\cE$, and by repeatedly replacing $\cE$ with $\cF$ we eventually arrive at an $\cE$ whose weights are in $S_{+}$.

\subsubsection{Affine spaces with torus actions}
Let $X = \Spec \CC[x_{1}, \ldots, x_{n}]$, and let $T$ be a $k$-dimensional torus acting linearly on $X$.
Assume that the action is such that $\Gamma(X,\cO_{X})^{T} = \CC$.
We will need some simple lemmas about $D_{T}^{b}(X)$.

Let $\rho_{i} \in \chi(T)$ be the character such that $x_{i}$ is a $T$-invariant section of $\cO_{X}(\rho_{i})$.
Define a partial ordering on $\ZZ^{k} = \chi(T)$ by saying $\rho \le \rho^{\pr}$ if there exist $i_{1}, \ldots, i_{k} \ge 0$ such that $\rho^{\pr} = \rho + \sum i_{j}\rho_{k}$.
This is equivalent to saying that $\rho \le \rho^{\pr}$ if and only if $\Hom_{X}(\cO(\rho), \cO(\rho^{\pr}))^{T} \not= 0$.

We say a triangulated category $\cC$ is generated by a set $S \subset Ob(\cC)$ if $\cC$ is the smallest full triangulated subcategory containing all objects in $S$.
\begin{nlemma}
\label{thm:GeneratorsOfDY}
The category $D^{b}_{T}(X)$ is generated by $\{\cO_{X}(\rho)\}_{\rho \in \chi(T)}$, and for any $n \in \ZZ$ we have $\Hom(\cO_{X}(\rho), \cO_{X}(\rho^{\prime})[n]) = 0$ unless $\rho \le \rho^{\pr}$.
\end{nlemma}
\begin{proof}
By \cite[2.2.10]{ballard_variation_2012}, any $T$-equivariant complex of sheaves on $X$ is isomorphic to a complex of locally free $T$-equivariant sheaves.
By \cite{masuda_equivariant_1996}, any $T$-equivariant locally free sheaf on $X$ decomposes as a direct sum of copies of $\cO_{X}(\rho)$.
This proves generation.

For the second statement, note that as $T$ is reductive, we have
\[
\Hom_{X/T}(\cO(\rho), \cO(\rho^{\pr})[n]) = \Ext^{n}_{X}(\cO(\rho), \cO(\rho^{\pr}))^{T}
\] 
As $X$ is affine and $\cO(\rho)$ locally free, this vanishes if $n \not= 0$, and the claim follows.
\end{proof}

Let $\cE \in D_{T}^{b}(X)$.
We let the weights of $\cE$, denoted $\wt(\cE)$, be the set of $\rho \in \chi(T)$ such that $(\cE|_{0})_{\rho} \not= 0$, where the subscript $\rho$ denotes the projection to the subcategory $\langle \cO(\rho) \rangle \subset D^{b}_{T}(0)$.

We say $S \subset \wt(\cE)$ is a maximal set if no element of $S$ is smaller than an element of $\wt (\cE) \setminus S$, and we say $S \subset \wt(\cE)$ is a minimal set if no element of $S$ is bigger than an element of $\wt (\cE) \setminus S$. 

For any $S \subset \chi(T)$, we let $D_{T}^{b}(X)_{S} \subset D_{T}^{b}(X)$ be the full subcategory of objects $\cE$ with weights in $S$.

\begin{nlemma}
\label{thm:VanishingAtOrigin}
If $\cE \in D^{b}_{T}(X)$ is such that $\cE|_{0} = 0 \in D^{b}_{T}(0)$, then $\cE = 0$.
\end{nlemma}
\begin{proof}
The support of $\cE$ is a closed $T$-invariant subset of $Y$.
Assume that $\cE|_{0} = 0$.
Then the support of $\cE$ does not intersect $0$.
The assumption $\Gamma(X,\cO_{X})^{T} = \CC$ implies that the closure of every $T$-orbit intersects $0$, hence the support of $\cE$ is empty and so $\cE = 0$.
\end{proof}

For any $S \subset \chi(T)$, we let $\overline{S}$ be the set of all $\rho$ such that $\rho \ge s$ for some $s \in S$, and let $\underline{S}$ be the set of $\rho$ such that $\rho \le s$ for some $s \in S$.
\begin{nlemma}
\label{thm:SemiOrthogonalDecompositionZ1}
For any $S \subset \chi(T)$, there are semiorthogonal decompositions
\[
D^{b}_{T}(X) = \langle D^{b}_T(X)_{\chi(T) \setminus \ol{S}}, D^{b}_{T}(X)_{\ol{S}} \rangle.
\]
and
\[
D^{b}_{T}(X) = \langle D^{b}_T(X)_{\underline{S}}, D^{b}_{T}(X)_{\chi(T) \setminus \underline{S}} \rangle.
\]
\end{nlemma}
\begin{proof}
We prove the first claim; the second is proved in the same way.
We work in the homotopy category $\cC$ of bounded locally free complexes on $X/T$, which is equivalent to $D^{b}_{T}(X)$.
By \cite{masuda_equivariant_1996}, any locally free sheaf is a direct sum of sheaves of the form $\cO(\rho)$.

For $\cE \in \cC$, let $\cE_{\ol{S}}$ be the subcomplex consisting of those $\cO(\rho)$ with $\rho \in \ol{S}$.
The fact that this is indeed a subcomplex, i.e.\ that $d(\cE_{\ol{S}}) \subset \cE_{\ol{S}}$, follows from the fact that any map
\[
\cO_{X}(\rho) \to \cO_{X}(\rho^{\pr})
\]
with $\rho \in \ol{S}$ and $\rho^{\pr} \not \in \ol{S}$ must be trivial.
The operation $\cE \mapsto \cE_{\ol{S}}$ is functorial.

Now for any $\cE \in \cC$ we have a functorial short exact sequence
\[
0 \to \cE_{\ol{S}} \to \cE \to \cE_{\chi(T) \setminus \ol{S}} \to 0,
\]
where $\cE_{\ol{S}} \in D^{b}_{T}(X)_{\ol{S}}$ and $\cE_{\chi(T) \setminus \ol{S}} \in D^{b}_{T}(X)_{\chi(T) \setminus \ol{S}}$.

If $\cE \in D^{b}_{T}(X)_{\ol{S}}$, then $\cE_{\chi(T) \setminus \ol{S}}|_{0} \cong 0$, and therefore $\cE_{\chi(T) \setminus \ol{S}} \cong 0$ by Lemma \ref{thm:VanishingAtOrigin}.
It follows that $\cE \cong \cE_{\ol{S}}$.
Similarly, if $\cF \in D^{b}_{T}(X)_{\chi(T) \setminus \ol{S}}$, then $\cF \cong \cF_{\chi(T) \setminus \ol{S}}$.
It now follows from Lemma \ref{thm:GeneratorsOfDY} that $\Hom(\cE, \cF) = \Hom(\cE_{\ol{S}}, \cF_{\chi(T) \setminus \ol{S}}) = 0$, which is what we needed.
\end{proof}
\begin{nlemma}
\label{thm:MaximalsGiveMaps}
Let $\cE \in D_{T}^{b}(X)$. If $\rho \in \wt (\cE)$ is maximal, then 
\[
\RHom(\cO_{X}(\rho), \cE) \not= 0.
\]
\end{nlemma}
\begin{proof}
By the semiorthogonal decomposition of Lemma \ref{thm:SemiOrthogonalDecompositionZ1}, we have a map $\cE_{\ol{\rho}} \to \cE$, where $\cE_{\ol{\rho}}$ is locally free and the underlying sheaf is a direct sum of $\cO_{X}(\rho^{\pr})$ with $\rho^{\pr} \ge \rho$.
Next we claim that the map
\[
\cE_{\ol{\rho}} \to (\cE_{\ol{\rho}})_{\underline{\rho}}
\]
is an isomorphism.
This holds because, by construction, the map is surjective and its kernel has weights in $(\ol{\rho} \setminus \rho) \cap \wt(\cE) = \varnothing$.

Now by construction the underlying sheaf of $(\cE_{\ol{\rho}})_{\underline{\rho}}$ is a direct sum of copies of $\cO(\rho)$ (with cohomological shifts).
Since we assume that $\Gamma(X,\cO_{X})^{T} = \CC$, every differential in the complex $(\cE_{\ol{\rho}})_{\underline{\rho}}$ must be constant on $X$.
As the restriction of the complex to $0$ does not vanish, it follows that there is a non-trivial map $\cO(\rho) \to (\cE_{\ol{\rho}})_{\underline{\rho}} \cong \cE_{\ol{\rho}}$, and composing with $\cE_{\ol{\rho}} \to \cE$ induces a non-trivial map from $\cO(\rho)$ to $\cE$.
\end{proof}

\subsubsection{Constructing objects supported on $Z_{+}^{\us}$}
We now return to our example and apply the above results to $X = V \times 0 \times L$, equipped with the previously described action of $T = (\CC^{*})^{2} \subset G$.
We have $\chi(T) = \ZZ^{2}$, and the partial ordering is given by letting $(i,j) \le (i^{\pr},j^{\pr})$ if there exist $l,m \ge 0$ such that $(i^{\pr},j^{\pr}) = (i + l - m, j -m)$.

We let $i : X \into Z$ be the inclusion, let $\cF \in D^{b}_{T}(X,0)$, and let $S = \wt(\cF)$.
\begin{nlemma}
\label{thm:WeightsOfRestrictions}
The weights of $i_{*}\cF$ are contained in 
\[
\bigcup_{i=0}^{n} S - (0,i).
\]
The weights of the cone over $i^{*}i_{*}\cF \to \cF$ are contained in
\[
\bigcup_{i=1}^{n} S - (0,i).
\]
\end{nlemma}
\begin{proof}
Let $p$ be the projection $Z \to X$.
Then $p\circ i = \id_{X}$, and so $\cF = i^{*}p^{*}(\cF)$.
It follows that the weights of $p^{*}(\cF)$ are equal to those of $\cF$, hence are also contained in $S$.
We have $i_{*}(\cF) = i_{*}i^{*}p^{*}(\cF) = p^{*}(\cF) \otimes i_{*}(\cO_{X})$, where $p^{*}(\cF) \in D^{b}_{T}(Z,0)$ and $i_{*}(\cO_{X}) \in D^{b}_{T}(Z,W)$.
The weights of $i_{*}(\cO_{X})$ are easy to compute using a Koszul resolution like the one in the proof of Lemma \ref{thm:FunctorsEquivalent}; they are $(0,-i), 0 \le i \le n$.
The first claim follows.

For the second claim, let $\cG = p^{*}\cF \in D^{b}_{T}(Z,0)$, so that we have $i^{*}\cG = \cF$.
The cone is taken over the map $i^{*}i_{*}i^{*}(\cG) \to i^{*}(\cG)$.
Since the composition
\[
i^{*}(\cG) \to i^{*}i_{*}i^{*}(\cG) \to i^{*}(\cG)
\]
is the identity (\cite[p.\ 85]{mac_lane_categories_1998}), the cone is isomorphic up to shift to the cone over $i^{*}(\cG) \to i^{*}i_{*}i^{*}(\cG) = i^{*}(\cG \otimes \cO_{X})$, i.e.\ to $i^{*}(\cG \otimes \cI_{X})[1]$.
The weights of $\cI_{X}$ are easily computed; they are $(0,-i), 1 \le i \le n$.
The second claim follows.
\end{proof}

Let now $\cE \in D^{b}_{T}(Z,W)$, and let $S$ be a minimal subset of $\wt(\cE)$.
\begin{nlemma}
\label{thm:WeightOfConeTCase}
The weights of the cone over the natural map $\cE \to i_{*}((\cE|_{X})_{\underline{S}})$ are contained in
\[
(\wt (\cE) \setminus S) \cup \left (\bigcup_{i = 1}^{n} S - (0,i) \right)
\]
\end{nlemma}
\begin{proof}
Let $\cF$ be the cone.
Pulling back the triangle defining $\cF$ along $i$, we get a triangle
\[
i^{*}(\cE) \stackrel{g}{\to} i^{*}i_{*}i^{*}(\cE_{\underline{S}}) \to C(g) = i^{*}\cF
\]
The weights of $\cF$ equal those of $C(g)$, so it is enough to prove the claim for the weights of $C(g)$.

We claim that the map
\[
f: i^{*}(\cE) {\to} (i^{*}(\cE))_{\underline{S}}
\]
induced by the semiorthogonal decomposition of Lemma \ref{thm:SemiOrthogonalDecompositionZ1} is equal to the composition
\[
i^{*}(\cE) \stackrel{g}{\to} i^{*}(i_{*}(i^{*}\cE)_{\underline{S}})) \stackrel{h}{\to} (i^{*}(\cE))_{\underline{S}},
\]
where $h$ is induced by the counit $i^{*}i_{*} \to \id$. 

To see this, let $P$ the functor $D^{b}_{T}(X,W) \to D^{b}_{T}(X,W)$ given by $\cE \mapsto \cE_{\underline{S}}$.
The map $f$ is induced by the natural transformation $i^{*} \to Pi^{*}$, which equals the composition of unit and counit transformations $i^{*} \to i^{*}i_{*}i^{*} \to i^{*} \to Pi^{*}$, by the triangle equalities of \cite[p.\ 85]{mac_lane_categories_1998}.
The composition $h \circ g$ is induced by the natural transformation $i^{*} \to i^{*}i_{*}i^{*} \to i^{*}i_{*}Pi^{*} \to Pi^{*}$.
We see that the two natural transformations are the same, hence $f = h \circ g$.

By the octahedral axiom, there is a distinguished triangle whose vertices are $C(f)$, $C(g)$ and $C(h)$.
The weights of $C(f)$ are contained in $\wt(\cE) \setminus S$, while the weights of $C(h)$ are contained in 
\[
S - (0,1), \ldots, S - (0,n)
\]
by Lemma \ref{thm:WeightsOfRestrictions}.
The claim about the weights of $C(g)$ follows.
\end{proof}

Now let $\cE \in D^{b}_{G}(Z,W)$, and let $\cE^{\pr} \in D^{b}_{T}(Z,W)$ be the underlying $T$-equivariant object.
Let $\sigma \in G$ be the order 2 element which permutes the factors of $V$ in $Z$.
The $G$-equivariant object induced from $(\cE^{\pr}|_{X})_{\underline{S}}$ is
\[
i_{*}((\cE^{\pr}|_{X})_{\underline{S}}) \oplus \sigma_{*} i_{*}((\cE^{\pr}|_{X})_{\underline{S}})
\]
and there is a canonical $G$-equivariant map
\[
\phi : \cE \to i_{*}((\cE^{\pr}|_{X})_{\underline{S}}) \oplus \sigma_{*} i_{*}((\cE^{\pr}|_{X})_{\underline{S}}).
\]

We also denote by $\sigma$ the involution of $\ZZ^{2}$ which permutes the $\ZZ$-factors.
We say $S \subset \ZZ^{2}$ is \emph{good} if
\[
\left (\bigcup_{i=0}^{n} \sigma S - (i,0) \right ) \cap S = \varnothing.
\]
\begin{nlemma}
\label{thm:WeightsOfG}
If $S \subset \wt(\cE)$ is minimal and good, then the weights of $C(\phi)$ are contained in
\[
\left(\bigcup_{i=1}^{n}(S - (0,i)) \cup (\sigma S - (i,0)) \right) \cup \left(\wt(\cE) \setminus (S \cup \sigma S)\right).
\]
\end{nlemma}
\begin{proof}
Let $\phi = (\phi_{1}, \phi_{2})$ be the decomposition of $\phi$ corresponding to the splitting of its codomain
\[
i_{*}((\cE^{\pr}|_{X})_{\underline{S}}) \oplus \sigma_{*} i_{*}((\cE^{\pr}|_{X})_{\underline{S}}),
\]
and let $p_{1}$ and $p_{2}$ be the projections of this object onto each factor.
Then the identity $\phi_{i} = p_{i}\phi$ and the octahedral axiom imply that the cones $C(\phi_{i})$, $C(\phi)$ and $C(p_{i})$ fit into a distinguished triangle for $i = 1, 2$.
We thus have 
\[
\wt(C(\phi)) \subset \bigcap_{i = 1, 2}\left(\wt(C(\phi_{i})) \cup \wt(C(p_{i}))\right).
\]
We have $\sigma\wt( C(\phi_{1})) = \wt(C(\phi_{2}))$ and $\sigma \wt(C(p_{1})) = \wt(C(p_{2}))$.
By Lemmas \ref{thm:WeightsOfRestrictions} and \ref{thm:WeightOfConeTCase}, we get
\[
\wt (C(p_{1})) \subset \bigcup_{i=0}^{n} \left(\sigma S - (i,0)\right)
\]
and 
\[
\wt (C(\phi_{1})) \subset (\wt (\cE) \setminus S) \cup \left (\bigcup_{i = 1}^{n} S - (0,i) \right).
\]
The claim now follows using the goodness of $S$.
\end{proof}

\begin{nprop}
\label{thm:ModifyingWeightsOurWay}
Let $\cE \in D^{b}_{G}(\cZ,W)$.
If $S \subset \wt(\cE)$ is minimal and good (resp. maximal and $-S$ is good), then there is a distinguished triangle in $D^{b}_{G}(Z,W)$
\[
\cF \to \cE \to \cG \ \ \ \ (\text{resp. } \cG \to \cE \to \cF)
\]
such that $\cG$ is supported on $Z \setminus Z_{+} = X \cup \sigma(X)$ and such that
\begin{align*}
\wt(\cF) &\subset \left(\bigcup_{i = 1}^{n} (S - (0,i)) \cup \sigma (S - (i,0))\right) \cup \left(\wt(\cE) \setminus (S \cup \sigma S)\right) \\
(\text{resp. } \wt(\cF) &\subset \left(\bigcup_{i = 1}^{n} (S + (0,i)) \cup \sigma (S + (i,0))\right) \cup \left(\wt(\cE) \setminus (S \cup \sigma S)\right)).
\end{align*}
\end{nprop}
\begin{proof}
The statements for maximal $S$ follow from those of minimal $S$ by dualising.
So assume that $S$ is minimal, in which case we may take 
\[
\cG = i_{*}((\cE^{\pr}|_{X})_{\underline{S}}) \oplus \sigma_{*} i_{*}((\cE^{\pr}|_{X})_{\underline{S}})
\]
by Lemma \ref{thm:WeightsOfG}.
\end{proof}

We can now give the proof of Proposition \ref{thm:NewWindowsAndKernelsGenerate}.
\theoremstyle{plain}
\newtheorem*{thm:NewWindowsAndKernelsGenerate}{Proposition \ref*{thm:NewWindowsAndKernelsGenerate}}
\begin{thm:NewWindowsAndKernelsGenerate}
\NewWindowsAndKernelsGenerate
\end{thm:NewWindowsAndKernelsGenerate}
\begin{proof}
Let $\cE \in D^{b}(\cZ,W)$, and suppose that $\cE^\pr$ is a cone over a map between $\cE$ and an object in $\ker j_{+}^{*}$.
If $\cE^{\prime} \in \langle \cW_{+}, \ker j_{+}^{*}\rangle$, then $\cE \in \langle \cW_{+}, \ker j_{+}^{*}\rangle$.
This means that if we can always after replacing $\cE$ with such a $\cE^{\pr}$ a finite number of times get $\cE \in \cW_{+}$, then the claim of the proposition follows.

We first see that we can get $\wt(\cE) \subset \{(i,j)\ |\ i + j \in [0, 2n-1]\}$.
This is an application of \cite[3.3.2]{ballard_variation_2012}, where we apply the result to the unstable locus $0 \times 0 \times L \subset Z$ and the 1-parameter diagonal subgroup $\CC^{*} \subset G$.

Define the width of $\cE$ as the maximal value of $|i - j|$ for $(i,j) \in \wt(\cE)$.
Let $k$ be the width of $\cE$, assume that $k > n/2$, and let $\wt(\cE)_{k}$ be the set of $(i,j) \in \wt(\cE)$ such that $|i-j| = k$.
Then either $\max_{(i,j)\in \wt(\cE)_{k}}(i + j) \ge n$ or $\min_{(i,j) \in \wt(\cE)_{k}}(i + j) < n$.

In the first case let $S \subset \wt(\cE)_{k}$ be the set of $(i,j)$ such that $i + j \ge n$ and $i < j$.
The set $S$ is minimal and good, and we modify $\cE$ by taking a cone over $\cG$ as in Proposition \ref{thm:ModifyingWeightsOurWay}.
In the second case we let $S$ be the set of $(i,j) \in \wt(\cE)_{k}$ with $i + j < n$ and such that $i > j$.
This $S$ is maximal and $-S$ is good, and again we modify $\cE$ by taking a cone over $\cG$ as in Proposition \ref{thm:ModifyingWeightsOurWay}.

In either case, we see that the replacing $\cE$ by this cone will either decrease the width of $\cE$ or leave it unchanged.
If the width of $\cE$ is unchanged, the cardinality of $\wt(\cE)_{k}$ will decrease.
Furthermore, we still keep the property that $\wt(\cE) \subset \{(i,j)\ |\ i + j \in [0, 2n-1]\}$.
Thus by repeating this procedure a finite number of times we obtain an $\cE$ with width $\le n/2$.
If $n$ is odd, we then have $\wt (\cE) \subseteq S_{+}$, hence $\cE \in \cW_{+}$, and we are done.

If $n$ is even, let $S$ be the set of $(i,j) \in \wt(\cE)_{n/2}$ with $i + j > n$ and such that $i < j$.
This $S$ is minimal and good, and by taking the cone over $\cG$ as in Proposition \ref{thm:ModifyingWeightsOurWay} we get $\wt(\cE) \subseteq S_{+}$.
\end{proof}

\subsection{Fully faithfulness}
Lemma \ref{thm:restrictionIsIsomorphism} below implies that the functor $\cW_{+} \into D^{b}(\cZ,W) \stackrel{|_{\cZ_{+}}}{\to} D^{b}(\cZ_{+},W)$ is fully faithful.

Let $B \subset \ZZ^{2}$ be the union of the sets
\begin{align*}
B_{1} = &\{(j-i, -n - i)\ |\ i, j \ge 0\}\\
B_{2} = &\{(-n - i, j-i)\ |\ i, j \ge 0\}\\
B_{3} = &\{(-n - i,-n - j)\ |\ i, j \ge 0\}. 
\end{align*}
\begin{nlemma}
\label{thm:restrictionIsIsomorphism}
Let $\cE_{1}, \cE_{2}$ be such that if $(i_{1}, j_{1}) \in \wt(\cE_{1})$ and $(i_{2},j_{2}) \in \wt(\cE_{2})$, then $(i_{2}-i_{1}, j_{2}-j_{1}) \not \in B$.
Then the restriction map
\[
\RHom_{\cZ}(\cE_{1}, \cE_{2}) \to \RHom_{\cZ_{+}}(j_{+}^{*}\cE_{1}, j_{+}^{*}\cE_{2})
\]
is an isomorphism.
\end{nlemma}
\begin{proof}
Let $X = V \times 0 \times L$.
The complement of $\cZ_{+}$ in $\cZ$ is $(X \cup \sigma X)/G$.
We have $\RHom(\cE_{1}, \cE_{2}) = \RGamma(\cZ, \cE_{1}^{\vee} \otimes \cE_{2})$, and a distinguished triangle in $D^{b}(\cZ,0)$
\[
\RGamma_{X\cup \sigma X} (\cE_{1}^{\vee}\otimes \cE_{2}) \to \cE_{1}^{\vee} \otimes \cE_{2} \to (j_{+})_{*}j_{+}^{*}(\cE^{\vee}_{1}\otimes \cE_{2}),
\]
where $\RGamma_{X\cup \sigma X}$ is the derived sheafy sections with support functor (see \cite[2.3.9]{ballard_variation_2012}).

Letting $\cE = \cE_{1}^{\vee} \otimes \cE_{2}$, it suffices to show $\RGamma(\cZ,\RGamma_{X \cup \sigma X}(\cE)) = 0$.
It is enough to prove this after replacing $\cE$ by its underlying $T$-equivariant factorisation on $Z$, so let us from this point on work in the category $D^{b}_{T}(Z,0)$.
We have a distinguished triangle
\[
\RGamma_{X \cap \sigma X}(\cE) \to \RGamma_{X}(\cE) \oplus \RGamma_{\sigma{X}}(\cE) \to \RGamma_{X \cup \sigma{X}}(\cE),
\]
and thus it suffices to show the vanishing of $\RGamma(Z/T,\cF)$ for $\cF$ equal to $\RGamma_{X}(\cE), \RGamma_{\sigma{X}}(\cE)$ or $\RGamma_{X \cap \sigma{X}}(\cE)$.

Let $\cI_{X}$ be the ideal sheaf of $X \subset Z$.
We first claim that 
\[
\RHom(\cI_{X}^{n-1}/\cI_{X}^{n}, \cE) \cong 0
\]
for all $n$.
We have $\cI_{X}^{k}/\cI_{X}^{k+1} = \Sym^{k}(\cO_{Z}(0,-1)^{\oplus n}) \otimes \cO_{X}$.
To prove the claim it therefore suffices to show
\[
\RHom_{Z/T}(\cO_{X}(0, -i), \cE) = 0\ \ \ \ \text{for all } i \ge 0.
\]
We have
\begin{align*}
\RHom_{Z/T}(\cO_{X}(0, -i), \cE) &\cong \RHom_{X/T}(\cO_{X}(0, -i), i^{!}(\cE)) \\
&\cong \RHom_{X/T}(\cO_{X}(0, n-i), \cE|_{X})[n].
\end{align*}
Now since we know the weights of $\cE$ are not contained in $B_{1}$, the same is true for the weights of $\cE|_{X}$.
In the terminology of Lemma \ref{thm:SemiOrthogonalDecompositionZ1} we have $\wt (\cE|_{X}) \subset \ZZ^{2} \setminus \overline{\{(0,-n-i)\}}$ for all $i \ge 0$, hence by that lemma we get $\RHom(\cO_{X}(0,-i), \cE) = 0$.

Now $\RHom(\cI_{X}^{n-1}/\cI_{X}^{n}, \cE) \cong 0$ for all $n$ implies $\RHom(\cO_{Z}/\cI_{X}^{n}, \cE) \cong 0$.
For any sheaf $\cF$ we have $\varinjlim_{n} \Hom(\cO_{Z}/\cI_{X}^{n}, \cF) \cong \Gamma(Z/T,\Gamma_{X}(\cF))$ \cite{grothendieck_cohomologie_2005}.
Taking an injective replacement $\cI$ of $\cE$ and using the exactness of filtered colimits we get
\[
\RGamma(Z/T,\RGamma_{X}(\cE)) \cong \Gamma(Z/T,\Gamma_{X}(\cI)) \cong \varinjlim_{n}\Hom(\cO_{Z}/\cI_{X}^{n}, \cI) \cong 0,
\]
which is what we wanted.
The proof that $\RGamma(Z/T,\RGamma_{\sigma{X}}(\cE)) = 0$ is exactly the same.

Arguing in the same way for $X \cap \sigma{X}$, we reduce to showing that 
\[
\RHom_{X \cap \sigma{X}}(\cO_{X \cap \sigma{X}}(-n - i,-n - j), \cE|_{X\cap \sigma{X}}) = 0
\]
for all $i, j \ge 0$.
This claim follows by Lemma \ref{thm:SemiOrthogonalDecompositionZ1} and the fact that $\wt(\cE|_{X\cap \sigma{X}}) \cap B_{3} = \varnothing$.
\end{proof}

Let $C = \{(i,j)\ |\ i + j \ge 2l \}$, and recall that $j_{-} : \cZ_{-} \to \cZ$ is the inclusion map.
The following lemma can be proved as above, but is also a consequence of \cite[3.29]{halpern-leistner_derived_2015}.
\begin{nlemma}
\label{thm:FullyFaithfulMinusCase}
Let $\cE_{1}, \cE_{2}$ be such that if $(i_{1}, j_{1}) \in \wt(\cE_{1})$ and $(i_{2},j_{2}) \in \wt(\cE_{2})$, then $(i_{2}-i_{1}, j_{2}-j_{1}) \not \in C$.
Then the restriction map
\[
\Hom_{\cZ}(\cE_{1}, \cE_{2}) \to \Hom_{\cZ_{-}}(j_{-}^{*}\cE_{1}, j_{-}^{*}\cE_{2})
\]
is an isomorphism.
\end{nlemma}

\subsection{The orthogonal complement in $\cW_{+}$}
\label{sec:OrthogonalComplement}
Assume that $\dim L = l < n = \dim V$, which is equivalent to assuming that $S_{+} \not \subset S_{-,\res}$.
Let $\cC$ be the category $\cC = \cW_{+} \cap \cW_{-,\res} \subset \cW_{+}$.

We define a partial ordering on $\ZZ^{2} = \chi(T)$ by $(i,j) \le (i^{\pr},j^{\pr})$ if $i \le i^{\pr}$ and $j \le j^{\pr}$.
Given representations $\rho, \rho^{\pr}$ of $G$, we say $\rho \le \rho^{\pr}$ if there are $T$-weights $(i,j)$ of $\rho$ and $(i^{\pr},j^{\pr})$ of $\rho^{\pr}$ such that $i \le i^{\pr}$ and $j \le j^{\pr}$.

Let $K = i_{*} \circ p^{*} : D^{b}(X) \to D^{b}(\cZ_{+}, W)$ be the equivalence from Proposition \ref{thm:KnorrerPeriodicity}.
Recall that an object $\cE$ in a triangulated category is exceptional if $\Hom(\cE,\cE) = \CC$ and $\Hom(\cE, \cE[n]) = 0$ for $n \not= 0$.

\begin{nprop}
\label{thm:HPDualityDecomposition}
The subcategory $\tensor[^{\perp}]{\cC}{} \subseteq \cW_{+}$ is generated by a set of exceptional objects $\cE_{\rho}$ indexed by irreducible $G$-representations $\rho$ such that $\wt(\rho) \in S_{+} \setminus S_{-,\res}$.
We have $\RHom(\cE_{\rho},\cE_{\rho^{\pr}}) = 0$ unless $\rho \le \rho^{\pr}$.
Furthermore, under the equivalence $\cW_{+} \stackrel{\Phi_{+}}{\to} D^{b}(\cZ_{+},W) \stackrel{K^{-1}}{\to} D^{b}(X)$, the objects $\cE_{\rho}$ are sent to $\cO_{X}(\rho)$.
\end{nprop}
Note the special case where $l = 0$, in which case $\cC = 0$. 
Proposition \ref{thm:HPDualityDecomposition} then describes a full exceptional collection on $D^{b}(\Sym^{2}\PP(V))$ -- see Section \ref{sec:HPDOurCase}.
\begin{proof}
Let $\rho$ be such that $\wt (\rho) \in S_{+} \setminus S_{-,\res}$.
Using Lemma \ref{thm:FunctorsEquivalent}, we find that $K(\cO_{X}(\rho))$ equals $\cO_{\Sym^2\PP(V)}(\rho) \otimes \cO(-l,-l)_{+}[l] \in D^{b}(\cZ_{+},W)$.

We now claim that
\begin{equation}
\label{eqn:EquivalentObjects}
\Phi_{+}^{-1}(\cO_{\Sym^2\PP(V)}(\rho) \otimes \cO(-l,-l)_{+}[l]) \cong \cO_{V \times V \times 0}(\rho) \otimes \cO(-l,-l)[l]_{+}.
\end{equation}
Let $\cE_{\rho}$ be the object on the right hand side of \eqref{eqn:EquivalentObjects}.
Using a Koszul resolution of $V \times V \times 0$ as in the proof of Lemma \ref{thm:FunctorsEquivalent} we find that $\wt (\cE_{\rho})$ is
\[
\{\rho, \rho - (1,1), \ldots, \rho - (l,l) \} \subset S_{+},
\]
so that $\cE_{\rho}$ lies in $\cW_{+}$.
We can then show \eqref{eqn:EquivalentObjects} by applying $\Phi_{+}$ to both sides, and the last claim of the proposition follows.

We next show that $\cE_{\rho}$ lies in $^{\perp}\cC$. 
Applying Lemma \ref{thm:FullyFaithfulMinusCase}, we find that for any $\cE \in \cC$,
\[
\Hom_{\cZ}(\cE_{\rho}, \cE) = \Hom_{\cZ_{-}}(\cE_{\rho}|_{\cZ_{-}}, \cE|_{\cZ_{-}}).
\]
But as $\cE_{\rho}$ is supported on $V \times V \times 0$, we have $\cE_{\rho}|_{\cZ_{-}} = 0$, and thus $\cE_{\rho} \in{} ^\perp\cC$.

An easy computation shows that $\cO_{\Sym^{2}\PP(V)}(\rho) \in D^{b}(\Sym^{2}\PP(V))$ is exceptional.
It now follows by the arguments of \cite{kuznetsov_homological_2007} that $\cO_{X}(\rho)$ is exceptional.
\[
\RHom(\cO_{X}(\rho), \cO_{X}(\rho)) = \RHom(\cO_{\Sym^{2}\PP(V)}(\rho), i_{*}i^{*}\cO_{\Sym^{2}\PP(V)}(\rho)),
\]
where $i$ is $X \into \Sym^{2}\PP(V)$.
Using a Koszul resolution of $\cO_{X}$ we see that $i_{*}i^{*}\cO_{\Sym^{2}\PP(V)}(\rho)$ is contained in $\langle \cO(\rho - (l,l)), \cdots, \cO(\rho) \rangle$, and the vanishing of the right hand side now follows using our assumption $l < n$.
Therefore $\cE_{\rho} = \Phi_{+}(\cO_{X}(\rho))$ is exceptional.

The fact that $\RHom(\cE_{\rho}, \cE_{\rho^{\pr}}) = 0$ unless $\rho \le \rho^{\pr}$ is proved in the same way, using the assumption that the weights of $\rho$ and $\rho^{\pr}$ are in $S_{+} \setminus S_{-,\res}$.

Finally, we must show that the objects $\cE_{\rho}$ generate $^{\perp}\cC$.
By \cite[3.1, 3.2]{bondal_representations_1989}, it suffices to show that if $\cE \in \cW^{+}$ and $\RHom(\cE_{\rho},\cE) = 0$ for all $\rho$, then $\cE \in \cC$.

Let $\cE \in \cW^{+}$ and assume $\cE \not\in \cC$.
Now let $\rho$ be such that $\wt (\rho)$ is a maximal subset of $\wt(\cE)$ with respect to the partial ordering on $\ZZ^{2}$ defined above.
Since $\cE \not \in \cC$, we have $\wt (\rho) \not \in S_{-,\res} \cap S_{+}$.
We have
\begin{align*}
\RHom_{Z/T}(\cE_{\rho}, \cE) &= \RHom_{Z/T}(i_{*}(\cO_{V\times V \times 0}(\rho))(-l,-l)), \cE) \\
&= \RHom_{(V \times V)/T}(\cO(\rho), \cE|_{V\times V}),
\end{align*}
and the latter space is non-vanishing by Lemma \ref{thm:MaximalsGiveMaps}.

The space $\RHom_{Z/T}(\cE_{\rho},\cE)$ splits into 2 eigenspaces where $\sigma \in G$ acts by $\pm 1$, and the elements in the $+1$ eigenspace are the $G$-invariant maps.
If there are no $G$-invariant maps, we replace $\rho$ with $\rho \otimes \tau$, where $\tau$ is the natural character $G \to \ZZ_{2} \into \CC^{*}$.
This switches the 2 eigenspaces, and so after doing this we will have $\RHom_{Z/G}(\cE_{\rho}, \cE) \not= 0$.
\end{proof}

\subsection{Why the strange windows?}
\label{sec:StrangeWindows}
The papers \cite{ballard_variation_2012, halpern-leistner_derived_2015} provide window categories inside $D^{b}(\cZ,W)$ equivalent to the categories $D^{b}(\cZ_{\pm},W)$.
The reader familiar with these results may want to know why we do not use the general construction from these papers.
For the case of $\cZ_{-}$, things work as expected, and the category $\cW_{-}$ corresponds to the one given in \cite{ballard_variation_2012, halpern-leistner_derived_2015}.
The need to consider $\cW_{-,\res} \subset \cW_{-}$ is then explained by the fact that $\cZ_{-}$ is an Artin stack, cf.\ \cite{addington_pfaffian-grassmannian_2014}.

For the other window category $\cW_{+}$, our definition is dictated by the proof of Theorem \ref{thm:MainTheorem}, which requires $\cW_{+}$ to either contain or be contained in $\cW_{-,\res}$, depending on the values of $l$ and $n$.

Let us recall the definition of the corresponding category from \cite{ballard_variation_2012, halpern-leistner_derived_2015}, to see that it does not behave as well in this respect:
For $i = 1,2$, let $T_{i}$ be the $i$-th factor of $T = (\CC^{*})^{2}$.
Choose two integers $m_{1},m_{2}$.
The window category is then the full subcategory $\cW_{+,m_{1},m_{2}} \subset D^{b}(\cZ,W)$ such that for any object $\cE \in D^{b}(\cZ,W)$ we have $\cE \in \cW_{+,m_{1},m_{2}}$ if and only if
\begin{itemize}
\item For all $T$-weights $(i,j)$ of $\cE|_{0}$, we have $m_{1} \le i + j \le m_{1} + 2n$.
\item After restricting to $(V \setminus 0) \times 0 \times 0 \subset Z$ (resp.\ $0 \times (V \setminus 0) \times 0 \subset Z$), the weights of $\cE$ with respect to $T_{2}$ (resp.\ $T_{1}$) lie in $[m_{2}, m_{2} + n]$.
\end{itemize}
Taking for instance $m_{1} = m_{2} = 0$, we obtain the subcategory of factorisations whose underlying $T$-equivariant sheaf decomposes as a sum of $\cO(i,j)$ with $(i, j)$ contained in the square $[0,n] \times [0,n]$.
It is easy to check that $\cW_{+,0,0}$ in general neither contains nor is contained in $\cW_{-,\res}$, and with a bit more work one can show that this remains true when replacing $\cW_{+,0,0}$ with the general $\cW_{+,m_{1},m_{2}}$.

\section{Sheaves of dg algebras}
In this section, we fix some notation and recall some results about sheaves of dg algebras which will be used in the remaining sections.
See for instance \cite{isik_equivalence_2013, mirkovic_linear_2010, riche_koszul_2010} for further details.

Let $\cX$ be an algebraic stack.
A sheaf of dg algebras on $\cX$ is a graded quasi-coherent algebra $R = \oplus_{i \in \ZZ} R^{i}$ on $\cX$, equipped with a differential map $d : R \to R[1]$ satisfying the Leibniz rule and such that $d^{2} = 0$.
A dg module over $R$ is a quasi-coherent sheaf $M$ on $\cX$, equipped with an action of $R$ and a differential map $M \to M[1]$ which squares to 0, where these structures satisfy the usual compatibility relations.
We denote by $C(\cX,R)$ the dg category of right dg modules over $R$.
We let $K(X,R)$ be the homotopy category of $C(\cX,R)$; this is a triangulated category.

Given a dg module $M$, we get a graded cohomology module $H(M) = \ker d/\im d$.
There is a triangulated subcategory of dg modules such that $H(M) = 0$, and we take the Verdier quotient by this subcategory to obtain the derived category $D(\cX,R)$.
We let $D^{b}(\cX,R) \subset D(\cX,R)$ be the full subcategory consisting of those dg modules whose cohomology module is a coherent $\cO_{\cX}$-module.

Assume from this point on that $\cX = X/G$, where $X$ is a quasi-projective variety and $G$ is an algebraic group.
Then by the results of  \cite{thomason_equivariant_1987}, the stack $\cX$ has the resolution property, which means that any coherent sheaf admits a surjection from a finite rank locally free sheaf.
By the construction in \cite{thomason_equivariant_1987}, a generating set $\{L_{s}\}$ of locally free sheaves can be found such that there is a single atlas of $\cX$ on which they are all trivialised.
It follows that arbitrary direct sums of the $L_{s}$ are locally free, hence any quasi-coherent sheaf on $\cX$ admits a surjection from a locally free sheaf.

We say a dg module $M$ is $K$-flat if for any acyclic $R^{\opposite}$-module $N$, the $\cO_{\cX}$ dg-module $M \otimes_{R} N$ is acyclic.
We say the category $K(\cX,R)$ has enough $K$-flat objects if for every dg module $M$ there exists a quasi-isomorphism $P \to M$ where $P$ is $K$-flat.
\begin{nlemma}
The category $K(\cX,\cO_{\cX})$ has enough $K$-flat objects.
\end{nlemma}
\begin{proof}
For any bounded above complex of $\cO_{\cX}$-modules $\cE^{\bullet}$ there is a locally free bounded above complex $\cP^{\bullet}$ and a surjective quasi-isomorphism $\cP^{\bullet} \to \cE^{\bullet}$.
The complex $\cP^{\bullet}$ is $K$-flat, and arguing as in the proof of \cite[5.6]{spaltenstein_resolutions_1988} the claim follows.
\end{proof}

\begin{nlemma}
\label{thm:KFlatTensorAcyclic}
The category $K(\cX,R)$ has enough $K$-flat objects.
\end{nlemma}
\begin{proof}
The proof in \cite[1.3.3]{riche_koszul_2010} goes through in our setting.
\end{proof}

\begin{nlemma}[\cite{spaltenstein_resolutions_1988}, 5.7]
\label{thm:TensorKFlatAcyclic}
If $M \in K(\cX,R)$ is $K$-flat and acyclic, and $N \in K(\cX,R^{\opposite})$, then $M \otimes_{R} N$ is acyclic.
\end{nlemma}

Let $\phi : R \to S$ be a homomorphism of dg algebras on $\cX$.
We then get an induced functor $\phi_{*} = - \otimes_{R} S : K(\cX,R) \to K(\cX,S)$, which can be derived on the left, since we have $K$-flat resolutions, and a functor $\phi^{*} = (-)_{R} : K(\cX,S) \to K(\cX,R)$, which is exact.

We say $\phi$ is a quasi-isomorphism if induced map of of cohomology algebras $H(\phi) : H(R) \to H(S)$ is an isomorphism.
The derived category of a sheaf of dg algebras is invariant under quasi-isomorphisms of dg algebras:
\begin{nlemma}
\label{thm:QuasiIsoImpliesEquivalence}
If $\phi : R \to S$ is a quasi-isomorphism of sheaves of dg algebras, then the functors $L\phi_{*}$ and $R\phi^{*}$ are inverse equivalences giving $D(X,R) \cong D(X,S)$, and they restrict to give $D^{b}(X,R) \cong D^{b}(X,S)$.
\end{nlemma}
\begin{proof}
See \cite[2.6]{isik_equivalence_2013}.
\end{proof}

Let $\cX$ and $\cY$ be gauged LG models with vanishing superpotential, i.e. $\cX = X/(G_{1} \times \CC^{*}_{R})$ and $\cY = Y/(G_{2} \times \CC^{*}_{R})$.
Let $\pi: \cX \to \cY$ be a morphism such that $\pi^{*}\cO_{\cY}[1] = \cO_{\cX}[1]$.

Let $R$ be a dg algebra on $\cY$, i.e.\ an algebra $R$ with a differential $R \to R \otimes \cO_{\cY}[1]$ which squares to 0, satisfying the usual compatibility axioms.
We then get a dg algebra $\pi^{*}R$ on $\cX$.

We will need a projection formula.
We say $\pi$ is equivariantly affine if fppf locally on $\cY$ the morphism $\pi$ is of the form $\Spec A/G \to \Spec B$, where $G$ is an algebraic group.
The functors in the following lemma are underived.
\begin{nlemma}
\label{thm:ProjectionFormula}
Assume $\pi$ is equivariantly affine.
For any right dg $R$-module $M$ on $\cY$ and left dg $\pi^{*}R$-module $N$ on $\cX$, the natural map
\[
M \otimes_{R} \pi_{*}(N) \to \pi_{*}(\pi^{*}(M) \otimes_{\pi^{*}(R)} N)
\]
is an isomorphism.
\end{nlemma}
\begin{proof}
Restrict to an affine chart $\Spec B \to \cY$ such that $\cX|_{\Spec B} \to \Spec B$ has the form $\Spec A/G \to \Spec B$.
The claim is then that
\[
M \otimes_{R} N^{G} \to \left(M \otimes_{R} (R \otimes_{B} A) \otimes_{R \otimes_{B} A} N\right)^{G}
\]
is an isomorphism.
If we omit taking $G$-invariants, the map is an isomorphism of dg $R$-modules, and since the map is $G$-equivariant the same is true for the associated $G$-invariant submodules.
\end{proof}

\section{Equivalence with the category of Clifford modules}
\label{sec:CliffordAlgebras}
Let $\cY = (L \setminus 0)/G$, which we recall is an $O(2)$-gerbe over $\PP(L)$.
This is the space denoted by $\cY_{L}$ in Theorem \ref{thm:MainTheorem}; as $L$ is fixed, we omit it from the notation.
Recall that $\cZ_{-} = Z_{-}^{\mathrm{ss}}/G$ is the GIT quotient from Section \ref{sec:NegativeGITQuotient}, and let $\pi : \cZ_{-} \to \cY$ be the projection.

The morphism $\pi$ is a rank $2n$ vector bundle over $\cY$, and we let $E$ be the dual of its sheaf of sections.
The function $W$ then induces a section of $\Sym^{2}E$.
We define a sheaf of Clifford algebras on $\cY$ by
\[
C = \Sym^{\bullet}(E)/\cI,
\]
where $\cI$ is the two-sided ideal generated by $s \otimes s - W(s) \cdot 1$ for sections $s$ of $E$.
Considering $C$ as a coherent sheaf and ignoring the algebra structure, we have $C \cong \wedge^{\bullet}(E)$.

For every point $p \in \PP(L)$, there is a functor $D^{b}(\cY,C) \to D^{b}_{\CC^{*}}(p)$ given by forgetting the $C$-module structure and pulling back along $p/\CC^{*} \to \cY$.
The category $D^{b}_{\CC^{*}}(p)$ splits as $\oplus_{i\in \ZZ} \langle \cO(i) \rangle$.
We define the grade restricted subcategory $D^{b}(\cY,C)_{\res} \subset D^{b}(\cY,C)$ to be the full subcategory of those objects which, after restriction to $D^{b}_{\CC^{*}}(p)$, lie in
\[
\langle \cO(-\lfloor n/2 \rfloor), \ldots, \cO(\lfloor n/2 \rfloor) \rangle.
\]
for all $p \in \PP(L)$.

The purpose of this section is to prove the following proposition.
\begin{nprop}
\label{thm:MainPropositionClifford}
There is an equivalence $D^{b}(\cZ_{-},W) \cong D^{b}(\cY,C)$, which induces an equivalence $D^{b}(\cZ_{-},W)_{\res} \cong D^{b}(\cY,C)_{\res}$.
\end{nprop}
\begin{proof}
We define below a locally free object $\cK \cong \cO_{\cY} \in D^{b}(\cZ_{-},W)$, and let $R = \pi_{*}(\sHom(\cK,\cK))$.
In Lemma \ref{thm:CliffordIsCohomology}, we show that $H(R) = C$, and so $D^{b}(\cY,R) \cong D^{b}(\cY,C)$ by Lemma \ref{thm:QuasiIsoImpliesEquivalence}.
We let $F : D^{b}(\cZ_{-},W) \to D^{b}(\cY,R)$ be given by $F = \pi_{*}(\sHom(\cK,-))$, and in Lemma \ref{thm:CliffordLeftAdjoint} show that it has a left adjoint $G$ given by the left derived functor of $\pi^{*}(-)\otimes_{\pi^{*}R} \cK$.
We show that $G$ is fully faithful in Lemma \ref{thm:CliffordFullyFaithful}.

Applying \cite[Thm.\ 3.3]{kuznetsov_homological_2007}, there is then a semiorthogonal decomposition
\[
D^{b}(\cZ_{-},W) = \langle \im G, \ker F \rangle.
\]
By Lemma \ref{thm:CliffordGenerates}, $\ker F = 0$, and so we see that $\im G = D^{b}(\cZ_{-},W)$. 
So $G$ is essentially surjective, hence gives an equivalence $D^{b}(\cY,C) \cong D^{b}(\cY,R) \cong D^{b}(\cZ_{-},W)$.
The fact that $G$ restricts to an equivalence of the grade-restricted subcategories is Lemma \ref{thm:CliffordGradeRestriction}.
\end{proof}

Consider $\cY$ as a substack of $\cZ_{-}$ via the inclusion $i : L\setminus 0 \into V \times V \times (L \setminus 0)$.
Let $K = i_{*}(\cO_{\cY}) \in D^{b}(\cZ_{-},W)$.
\begin{nlemma}
\label{thm:CliffordResolutionOfK}
We have an isomorphism
\[
K^{\vee} \cong i_{*}(\wedge^{2n}(E^{\vee}))
\]
of objects in $D^{b}(\cZ_{-},-W)$.
\end{nlemma}
\begin{proof}
The stack $\cY$ is cut out of $\cZ_{-}$ by the canonical section of $\pi^{*}(E^{\vee})$, which has degree $1$ with respect to the $\CC^{*}_{R}$-action.
This induces the following Koszul resolution of $i_{*}\cO_{\cY}$ as a ($\CC^{*}_{R}$-equivariant) sheaf on $\cZ_{-}$:
\[
i_{*}\cO_{\cY} = \pi^{*}(\wedge^{2n} E)[-2n] \to \cdots \to \pi^{*}E[-1] \to \cO_{\cZ_{-}}.
\]
We may add leftwards arrows to this resolution to obtain an object $\cE$ in $D^{b}(\cZ_{-},W)$:
\begin{equation}
\label{eqn:DefOfK}
\cK = \pi^{*}(\wedge^{2n} E) \mfarrows{}{} \cdots \mfarrows{}{} \pi^{*}(E) \mfarrows{}{} \cO_{\cZ_{-}}.
\end{equation}
See \cite[p.\ 14]{ballard_derived_2014} for explicit formulas for the leftwards arrows.
The disappearance of the cohomological shifts is because of our conventions for writing factorisations, see Section \ref{sec:Resolutions}.
Dualising, we get the factorisation
\[
\cK^{\vee} = \pi^{*}(\wedge^{2n} E^{\vee}) \mfarrows{}{} \cdots \mfarrows{}{} \pi^{*}(E^{\vee}) \mfarrows{}{} \cO_{\cZ_{-}}.
\]

By Lemma \ref{thm:Resolutions}, $\cK \cong K$, and so $\cK^{\vee} \cong K^{\vee}$.
Considering only the leftwards arrows in the resolution $\cK^{\vee}$, these form a Koszul resolution of $i_{*}(\wedge^{2n}(E^{\vee}))$, so by Lemma \ref{thm:Resolutions} again we get $K^{\vee} \cong i_{*}(\wedge^{2n}(E^{\vee}))$.
\end{proof}

\begin{nlemma}
\label{thm:ShapeOfFClifford}
For any $\cE \in D^{b}(\cZ_{-},W)$, we have 
\[
\pi_{*}\Rhom_{\cZ_{-}}(K,\cE) \cong \wedge^{2n}(E^{\vee}) \otimes \cE|_{\cY}.
\]
\end{nlemma}
\begin{proof}
This follows from
\begin{align*}
\Rhom_{\cZ_{-}}(K, \cE) & \cong \Rhom_{\cZ_{-}}(\cE^{\vee}, K^{\vee}) \cong \Rhom_{\cZ_{-}}(\cE^{\vee},i_{*}(\wedge^{2n}(E^{\vee}))) \\
& \cong i_{*}\Rhom_{\cY}(\wedge^{2n}(E), \cE|_{\cY}) \cong i_{*}(\wedge^{2n}(E^{\vee}) \otimes \cE|_{\cY}).
\end{align*}
\end{proof}

Given an object $\cE \in D^{b}(\cZ_{-},W)$, we define the support of $\cE$ to be the support of the cohomology of $\cE \otimes \cE^{\vee}$, considered as a subset of $Z_{-}^{\mathrm{ss}}$.
The support of $\cE$ is closed and $G\times \CC^{*}_{R}$-invariant.
Furthermore, if $\cE \otimes \cE^{\vee} \cong 0$, then we must have $1 = 0 \in \Hom(\cE,\cE)$, so that $\cE \cong 0$.
Therefore the support of $\cE$ is empty if and only if $\cE \cong 0$.

\begin{nlemma}
\label{thm:CliffordGenerates}
If $\cE \in D^{b}(\cZ_{-},W)$ is such that $\pi_{*}\Rhom(K, \cE) = 0$, then $\cE = 0$.
\end{nlemma}
\begin{proof}
By Lemma \ref{thm:ShapeOfFClifford} we have $\cE|_{\cY} \cong 0$, hence $\cE \otimes \cE^{\vee}|_{\cY} \cong 0$.
It follows that the support of $\cE$ does not intersect $L \setminus 0$.
Since it is a closed $\CC^{*}_{R}$-invariant subset of $Z^{\mathrm{ss}}_{-}$, it must then be empty, which implies $\cE \cong 0$.
\end{proof}

Let $\cK$ be the locally free representative of $K$ given in \eqref{eqn:DefOfK} and define the dg algebra $R$ on $\cY$ by
\[
R = \pi_{*}(\cEnd(\cK)).
\]
We define the functor $F : D(\cZ_{-},W) \to D(\cY,R)$ by
\[
\cE \mapsto \pi_{*}(\sHom_{\cZ_{-}}(\cK,\cE)).
\]
Here $\sHom_{\cZ_{-}}(\cK, -)$ and $\pi_{*}$ are both exact, and so $F$ is exact.

\begin{nlemma}
\label{thm:CliffordFBounded}
The functor $F$ takes $D^{b}(\cZ_{-},W)$ to $D^{b}(\cY,R)$.
\end{nlemma}
\begin{proof}
Follows from Lemma \ref{thm:ShapeOfFClifford}.
\end{proof}

\begin{nlemma}
\label{thm:CliffordIsCohomology}
The cohomology algebra of $R$ is isomorphic to $C$, and there is a quasi-isomorphism $C \cong \ker d_{R} \cap R^{0} \subset R$.
\end{nlemma}
\begin{proof}
For the first claim, see \cite[5.7]{ballard_derived_2014}; the computation there goes through in our case.
The second claim is true because $R$ is concentrated in positive cohomological degrees.
\end{proof}
One can explain why the computation of Lemma \ref{thm:CliffordIsCohomology} works in the following way.
Suppose we turned off the superpotential, and computed the algebra $H(\pi_{*}\Rhom(\cO_{\cY},\cO_{\cY}))$ with $\cO_{\cY}$ in the category $D^{b}(\cZ_{-},0)$ instead of in $D^{b}(\cZ_{-},W)$.
In general, if $X \into Y$ is a closed immersion of nonsingular varieties, then we have $\sExt^{i}(\cO_{X},\cO_{X}) = \wedge^{i}N_{X/Y}$ with the natural algebra structure on $\sExt^{\bullet}(\cO_{X},\cO_{X})$.
In our case, the same computation, coupled with the observation that the $\CC^{*}_{R}$-action changes the cohomological degrees, gives an isomorphism of sheaves of algebras (concentrated in cohomological degree 0):
\[
H(\pi_{*}\Rhom(\cO_{\cY},\cO_{\cY})) = \wedge^{\bullet} E.
\]
Turning on the superpotential, we can deform a locally free resolution of $\cO_{\cY} \in D^{b}(\cZ_{-},0)$ to a locally free factorisation for $\cO_{\cY} \in D^{b}(\cZ_{-},W)$.
This gives a deformation of $H(\pi_{*}\Rhom(\cO_{\cY},\cO_{\cY}))$ from an exterior algebra to a Clifford algebra.

We now claim that $F$ has a left adjoint functor $G : D^{b}(\cY,R) \to D^{b}(\cZ_{-},W)$.
We begin by defining a functor $UG : K(\cY,R) \to K(\cZ_{-},W)$ by
\[
UG(M) = \pi^{*}(M)\otimes_{\pi^{*}(R)} \cK,
\]
for any $R$-module $M$.
Here $\cK$ is a $\pi^{*}(R)$-module through the canonical map $\pi^{*}(R) = \pi^{*}\pi_{*}(\cEnd(\cK)) \to \cEnd(\cK)$.

We claim that $UG$ has a left derived functor $G : D(\cY,R) \to D(\cZ_{-},W)$.
Since $K(\cY,R)$ has enough $K$-flat objects by Lemma \ref{thm:KFlatTensorAcyclic}, the following lemma proves that $G$ is defined on $D(\cY,R)$.
\begin{nlemma}
\label{thm:CliffordDerived}
If $P_{1} \to P_{2}$ is a quasi-isomorphism of $K$-flat objects of $K(\cY,R)$, then the map $UG(P_{1}) \to UG(P_{2})$ is an isomorphism in $D(\cZ_{-},W)$.
\end{nlemma}
\begin{proof}
We must check that for any acyclic $K$-flat $P$, we have $UG(P) \cong 0$ in $D(\cZ_{-},W)$.
Let $\cE \in D^{b}(\cZ_{-},W)$, and compute
\begin{align*}
\RHom(\cE,UG(P)) &\cong \RGamma(\cZ_{-}, \pi^{*}(P) \otimes_{\pi^{*}R} \cK \otimes \cE^{\vee}) \\
&\cong \RGamma(\cY,\pi_{*}(\pi^{*}(P) \otimes_{\pi^{*}R} \cK \otimes \cE^{\vee})).
\end{align*}
Now by the projection formula of Lemma \ref{thm:ProjectionFormula} we have
\[
\pi_{*}(\pi^{*}(P) \otimes_{\pi^{*}R} \cK \otimes \cE^{\vee}) \cong P \otimes_{R} \pi_{*}(\cK \otimes \cE^{\vee}).
\]
Since $P$ is acyclic and $K$-flat, the right hand side is acyclic by Lemma \ref{thm:TensorKFlatAcyclic}.
Hence $\RHom(\cE, UG(P)) = 0$ for all $\cE \in D^{b}(\cZ_{-},W)$.
By Lemma \ref{thm:BoundedsGenerate} it follows that $UG(P) \cong 0$.
\end{proof}

\begin{nlemma}
\label{thm:CliffordLeftAdjoint}
The functor $G$ is left adjoint to $F$.
\end{nlemma}
\begin{proof}
The underived versions of these functors, i.e.\ $F : K(\cZ_{-},W) \to K(B,R)$ and $UG : K(B,R) \to K(\cZ_{-},W)$ are clearly adjoint.
Then by \cite[Exp.\ 17, Thm.\ 2.3.7]{SGA43} their derived functors are adjoint as well.
\end{proof}

\begin{nlemma}
\label{thm:CliffordFullyFaithful}
The functor $G$ is fully faithful.
\end{nlemma}
\begin{proof}
We must show that the transformation of functors $\id \to FG$ is an equivalence.
That is to say, we must show that for a $K$-flat $M \in K(\cY,R)$, the natural map 
\[
M \to \pi_{*}\sHom(\cK,\pi^{*}(M) \otimes_{\pi^{*}(R)} \cK)
\]
is a quasi-isomorphism.
Applying the projection formula of Lemma \ref{thm:ProjectionFormula}, the right hand side may be rewritten as 
\[
\pi_{*}(\pi^{*}(M) \otimes_{\pi^{*}(R)} \cK \otimes_{\cZ_{-}} \cK^{\vee}) \cong M \otimes_{R} \pi_{*}(\cK \otimes_{\cZ_{-}} \cK^{\vee}) = M \otimes_{R} R = M.
\]
\end{proof}

The factorisation $\cK$ admits a left action of $\pi^{*}C$ through the inclusion $\pi^{*}C \into \pi^{*}R$.
\begin{nlemma}
\label{thm:KIsC}
As a left $\pi^{*}C$-module, we have $\cK \otimes \pi^{*}(\wedge^{2n}E^{\vee}) \cong \pi^{*}C $.
\end{nlemma}
\begin{proof}
Forgetting the differential on $\cK$ for the moment, let $\psi : \pi^{*}(\wedge^{2n}E) \to \cK$ be the inclusion into the left-most factor in \eqref{eqn:DefOfK}.
The composition 
\[
\phi : \pi^{*}C \to \sHom(\cK^{\vee},\cK^{\vee}) \stackrel{\psi^{\vee}\circ}{\to} \sHom(\cK^{\vee}, \pi^{*}(\wedge^{2n}E^{\vee})) = \cK \otimes \pi^{*}(\wedge^{2n}E^{\vee})
\]
is a map of left $\pi^{*}C$-modules, which we claim is an isomorphism.

The map $\phi$ is a map of $\CC^{*}_{R}$-equivariant locally free sheaves.
Since $\CC^{*}_{R}$ scales the fibres of $\pi$ positively, it suffices to check that the map is an isomorphism after restriction to the 0-section $\cY$.

After restriction to $\cY$, the map factors as
\[
\phi|_{\cY} : C \to \pi_{*}\sHom(\cK^{\vee},\cK^{\vee}) \to \pi_{*}\sHom(\cK^{\vee}, i_{*}(\wedge^{2n}E)) = \cK|_{\cY} \otimes \wedge^{2n}E^{\vee}.
\]
Putting the natural differentials on the objects in this sequence of maps, we see that every map is a quasi-isomorphism by Lemma \ref{thm:CliffordResolutionOfK}.
As the differentials vanish on both source and target, $\phi|_{\cY}$ must be an isomorphism.
\end{proof}

\begin{nlemma}
\label{thm:CliffordGBounded}
The functor $G$ sends $D^{b}(\cY,R)$ to $D^{b}(\cZ_{-},W)$.
\end{nlemma}
\begin{proof}
We have an equivalence $\Phi : D^{b}(\cY,C) \to D^{b}(\cY,R)$, by Lemmas \ref{thm:CliffordIsCohomology} and \ref{thm:QuasiIsoImpliesEquivalence}, where $\Phi$ is given by $M \mapsto M \otimes^{L}_{C} R$.

Let $U\Phi$ be the underived functor $- \otimes_{C} R : K(\cY,C) \to K(\cY,R)$, and let $UH = UG \circ U\Phi$.
Then for any $M \in K(\cY,C)$, we have
\begin{align*}
UH(M) &= \pi^{*}(M \otimes_{C} R) \otimes_{\pi^{*}R} \cK \\
&= \pi^{*}(M) \otimes_{\pi^{*}(C)} \pi^{*}(R) \otimes_{\pi^{*}R} \cK = \pi^{*}(M) \otimes_{\pi^{*}C} \cK.
\end{align*}
The underlying sheaf of $\cK$ (forgetting the differential) is locally isomorphic to $\pi^{*}(C)$ as a left $\pi^{*}(C)$-module, by Lemma \ref{thm:KIsC}.

It now follows that if 
\[
M^{1} \to M^{2} \to \cdots \to M^{n}
\]
is an exact sequence of dg $C$-modules, then the induced sequence of factorisations $\pi^{*}(M^{\bullet}) \otimes_{\pi^{*}C} \cK$ is exact as well.

On $K^{b}(\cY,C)$, every complex $M$ has a left replacement by a bounded, coherent complex $N$.
If $N^{\pr}$ is a different such complex, then $UH(N) \cong UH(N^{\prime})$ in $D(\cZ_{-},W)$ by the above calculations.
Therefore we may compute the left derived functor $H$ of $UH$ on $K^{b}(\cY,C)$ by such replacements.
Now if $M$ is a coherent $C$-module, then $H(M)$ is coherent by
\[
H(M) = \pi^{*}(M) \otimes_{\pi^{*}C} \cK = \pi^{*}(M) \otimes_{\pi^{*}C} \pi^{*}C  \otimes_{\cZ_{-}} \wedge^{2n}E^{\vee} = \pi^{*}M \otimes_{\cZ_{-}} \wedge^{2n}E^{\vee},
\]
and it follows that $H(M)$ is coherent for all $M \in D^{b}(\cY,C)$.

Since $H = G \circ \Phi$ and $\Phi : D^{b}(\cY,C) \to D^{b}(\cY,R)$ is an equivalence, the claim follows.
\end{proof}

\begin{nlemma}
\label{thm:CliffordGradeRestriction}
For $\cE \in D^{b}(\cZ_{-},W)$, we have $F(\cE) \in D^{b}(\cY,C)_{\res}$ if and only if $\cE \in D(\cZ_{-},W)_{\res}$.
\end{nlemma}
\begin{proof}
By Lemma \ref{thm:ShapeOfFClifford} we have $F(\cE) \cong \wedge^{2n}E^{\vee} \otimes \cE|_{\cY}$.
Since $\wedge^{2n}E^{\vee}$ has weight 0 at every point $p \in \PP(L)$, it follows that for any point $p \in \PP(L)$, the restriction of $\cE$ to $\pt/\CC^{*}$ satisfies the grade restriction condition if and only if $F(\cE)$ does.
In other words, $F(\cE) \in D^{b}(\cY, C)_{\res}$ if and only if $\cE \in D^{b}(\cZ_{-},W)_{\res}$.
\end{proof}

\section{Geometric interpretation of $D^{b}(\cZ_{-},W)_{\res}$, odd case}
\label{sec:GeometricInterpretation}
In this section we assume that $n = \dim V$ is odd, that the corank of the quadratic form at each point of $\PP(L) \subset \PP(\Sym^{2}(V^{\vee}))$ is at most $2$, and that the loci of points of corank 1 and 2 within $\PP(L)$ are nonsingular and of codimension 1 and 3, respectively.
This assumption holds for a generic $L$ of dimension $\le 6$.

We begin by describing the variety $f : Y \to \PP(L)$.
On the open locus in $\PP(L)$ where the quadratic form has corank 0 or 1, the variety $Y$ is an étale double cover of the corank 1 locus.
The 2 points in the fibre $f^{-1}(q)$ at a corank 1 point $q \in \PP(L) \subset \PP(\Sym^{2}V^{\vee})$ correspond to the connected components of the moduli space of maximal isotropic subspaces in $(V,q)$.
The variety $Y$ is nonsingular, and ramified in the locus of corank 2 points.

We aim to prove the following proposition:
\begin{nprop}
\label{thm:GeometricInterpretationOfMFCategory}
There is an equivalence of categories 
\[
D^{b}(Y) \cong D^{b}(\cZ_{-},W)_{\res}.
\]
\end{nprop}
\begin{proof}
The strategy is the same as in the proof of Proposition \ref{thm:MainPropositionClifford}.
We define a certain factorisation $\cK$ on $(\cZ_{-},W)$ and a sheaf of dg algebras $R = \pi_{*}(\sHom(\cK,\cK))$ on $\PP(L)$, and then consider the functor $F : D^{b}(\cZ_{-},W)_{\res} \to D^{b}(\PP(L),R)$ given by $\pi_{*}(\sHom(\cK,-))$.

A somewhat involved computation gives Proposition \ref{thm:globalIso}, which says that $H(R) \cong f_{*}\cO_{Y}$. 
By Lemma \ref{thm:QuasiIsoImpliesEquivalence} we thus get $D^{b}(\PP(L),R) \cong D^{b}(Y)$.

By Lemma \ref{thm:GDerivedExists}, $F$ admits a left adjoint $G$ which is fully faithful.
As $\ker F = 0$ by Lemma \ref{thm:GeneratorDoubleCover}, we find that $G$ is an equivalence by the same argument as in the proof of Proposition \ref{thm:MainPropositionClifford}.
\end{proof}

\subsection{The generating object}
We let $\pi : \cZ_{-} \to \PP(L)$ and $\ol{\pi} : Z^{\mathrm{ss}}_{-} \to L\setminus 0$ be the projections.
We define some natural $T$-invariant subvarieties of $Z^{\mathrm{ss}}_{-}$ as follows.

For any point $q \in L \setminus 0$, the fibre $\pi^{-1}([q])$ is isomorphic to $V \oplus V/O(2)$.
The superpotential induces a bilinear form on $V$, which we abusively denote by $q$ as well.
We let $Y_{1}, Y_{2} \subset Z^{\sstable}_{-}$ be the reduced subvarieties such that $Y_{1}|_{\ol{\pi}^{-1}(q)} = \ker q \oplus V$ and $Y_{2}|_{\ol{\pi}^{-1}(q)} = V \oplus \ker q$ if $q$ is singular, and $Y_{i}|_{\ol{\pi}^{-1}(q)} = \varnothing$ if $q$ is nonsingular.
We get objects $\cO_{Y_{1}}$ and $\cO_{Y_{2}}$ in $D_{T}^{b}(Z^{\sstable}_{-},W)$, and let 
\[
K = \cO_{Y_{1}}((n-1)/2,0) \oplus \cO_{Y_{2}}(0,(n-1)/2),
\]
which is an object of $D_{G}^{b}(Z^{\mathrm{ss}}_{-},W)$, the $G$-structure being induced by the identification of $\sigma(Y_{1})$ with $Y_{2}$.

Choose a locally free resolution $\cK$ of $K \in D^{b}(\cZ_{-},W)$, and define a sheaf of dg algebras on $\PP(L)$ by
\[
R = \pi_{*}(\Rhom(\cK, \cK)).
\]
\begin{nprop}
\label{thm:globalIso}
There is an isomorphism of $\cO_{\PP(L)}$-algebras $H(R) \cong f_{*}\cO_{Y}$.
\end{nprop}
The proposition is proved by combining the local description of $H(R)$ in Corollary \ref{thm:HRIsYLocally} with the global description of $H(R)$ over the corank 1 locus in Lemma \ref{thm:Corank1Fibre}.

\subsection{Standard local form}
\label{sec:StandardForm}
Our strategy of computation is to apply the fact that, étale locally on $\PP(L)$, we can put $\cZ_{-} \to \PP(L)$ in a standard form, which we now describe.

Let $B$ be a nonsingular variety.
A quadratic vector bundle $(F,q)$ is a bundle $F$ on $B$ together with a section $q$ of $\Sym^{2}F^{\vee}$.
From this data we can define a gauged LG model as follows.

\begin{nconstruction}
\label{defi:QuadraticConstruction}
Let $X$ be the total space of $F \oplus F \to B$, let $O(2) = \CC^{*}\rtimes \ZZ_{2}$ act on $X$ by scaling the first $F$-factor by $t$, the second $F$-factor by $t^{-1}$, and let the $\ZZ_{2}$ act by permuting the $F$.
We let $\CC^{*}_{R}$ act via scaling by $t_{R}$.
Thinking of $q$ as a symmetric bilinear form on $F$, we define the superpotential by $W(f_{1},f_{2}) = q(f_{1},f_{2})$ for points $f_{1}, f_{2} \in F_{p}$.
We thus get an LG model $(X/O(2), W)$, as well as an $SO(2)$-equivariant version $(X/SO(2),W)$.
\end{nconstruction}

Conversely, if $\cX \to B$ is an LG model which locally on $B$ is of the form $F \oplus F/O(2) \to B$ with fibre-wise quadratic superpotential, then it is locally obtained by the above procedure for some $(F,q)$.

\begin{ndefn}
\label{defi:StandardSpaces}
We define some standard LG models, which will be local models for other LG models of the above form near a point where the quadratic form has corank $\le 2$:
\begin{itemize}
\item Corank 0: Let $B = \pt$, and let $(F,q)$ be given by $F = \cO_{B}^{n}$ and $q = \sum_{k=1}^{n} z_{k}^{2}$.
\item Corank 1: Let $B = \AA^{1}_{s}$, and let $(F,q)$ be given by $F = \cO_{B}^{n}$ and $q = sz_{1}^{2} + \sum_{k=2}^{n} z_{k}^{2}$.
\item Corank 2: Let $B = \AA^{3}_{s,t,u}$, and let $(F,q)$ be given by $F = \cO_{B}^{n}$ and $q = sz_{1}^{2} + 2tz_{1}z_{2} + uz_{2}^{2} + \sum_{k=3}^{n} z_{k}^{2}$.
\end{itemize}

We refer to the $O(2)$-equivariant LG model $\pi_{n} : (X_{n},W_{n}) \to B$ obtained by Construction \ref{defi:QuadraticConstruction} from the above quadratic bundles as the standard models of corank 0, 1 and 2, respectively.

We let $Y_{n,1}, Y_{n,2} \subset X_{n}$ be the reduced subvarieties such that for every $p \in B$, we have $Y_{n,1}|_{\pi_{n}^{-1}(p)} = \ker q_{p} \oplus F$ and $Y_{n,2}|_{\pi_{n}^{-1}(p)} = F \oplus \ker q_{p}$ if $q_{p}$ is singular, and $Y_{n,i}|_{\pi_{n}^{-1}(p)} = \varnothing$ if $q_{p}$ is nonsingular.
We get objects $\cO_{Y_{n,1}}, \cO_{Y_{n,2}} \in D^{b}_{SO(2)}(X_{n},W_{n})$, and define
\[
K_{n} = \cO_{Y_{n,1}}((n-1)/2) \oplus \cO_{Y_{n,2}}(-(n-1)/2) \in D^{b}_{O(2)}(X_{n},W_{n}).
\]
We define the subcategory $D^{b}_{O(2)}(X_{n},W_{n})_{\res} \subset D^{b}_{O(2)}(X_{n},W_{n})$ as in Section \ref{sec:GradeRestrictedObjects}.
\end{ndefn}

\begin{nlemma}
\label{thm:QuadraticLocalForm}
Let $(F,q) \to B$ be a quadratic vector bundle of rank $n$, such that $q$ point-wise has corank $\ge k$.
Étale locally on $B$ we may choose a trivialisation of $F$ such that $q = q_{k} \oplus q_{\triv}$, where $q_{k}$ has dimension $k$ and $q_{\triv} = z_{1}^{2} + \cdots + z_{n-k}^{2}$.
\end{nlemma}
\begin{proof}
Choose a local section $s$ of $F$ such that $q(s) = a \not= 0$.
Étale locally, we can replace $s$ with $a^{-1/2}s$ and so assume $q(s) = 1$.
We then have $F = \langle s \rangle^{\perp} \oplus \langle s \rangle$ as a quadratic bundle. 
The claim follows by repeating this procedure $n-k$ times with $F$ each time replaced by $\langle s \rangle^{\perp}$.
\end{proof}

Let now $X_{n}(i) \to B_{i}$ be the standard local LG model of corank $i$.
\begin{nlemma}
\label{thm:SmoothMorphism}
Let $p \in \PP(L)$ be a corank $i$ point. Then there exists an étale neighbourhood $U$ of $p$ and a smooth morphism $f : U \to B$, such that we have a 2-Cartesian diagram
\[
\begin{tikzcd}
\cZ_{-}|_{U} \arrow{d} \arrow{r}{g} & X_{n}(i)/O(2) \arrow{d}{\pi_{n}} \\
U \arrow{r}{f} & B
\end{tikzcd}
\]
where $g$ is compatible with the $\CC^{*}_{R}$-actions and superpotentials on $X_{n}/O(2)$ and $\cZ_{-}$.
\end{nlemma}
\begin{proof}
By restricting to some étale neighbourhood $U$ of $p$ and using Lemma \ref{thm:QuadraticLocalForm}, we may assume that $\cZ_{-} \to \PP(L)$ is constructed from $(F,q)$ where $F = \cO^{n}$ and $q = q_{i} \oplus q_{\triv}$.
The factor $q_{i}$ gives a map to $f : U \to \Sym^{2}(\CC^{i}) = B_{i}$, which gives rise to the correct Cartesian diagram.
By our genericity assumption on $L$, $f^{-1}(0)$ is nonsingular, so since $p \in f^{-1}(0)$ we find that $f$ is smooth in a neighbourhood of $p$.
\end{proof}

\begin{nlemma}
\label{thm:LocalForm}
Around every point $p \in \PP(L)$, there exists an étale neighbourhood $U \to \PP(L)$ and a morphism $f : U \to B_{2}$, such that we have a 2-Cartesian diagram
\[
\begin{tikzcd}
\cZ_{-}|_{U} \arrow{d} \arrow{r}{g} & X_{n}(2)/O(2) \arrow{d}{\pi_{n}} \\
U \arrow{r}{f} & B_{2}
\end{tikzcd}
\]
where $g$ is compatible with the $\CC^{*}_{R}$-actions and superpotentials on $X_{n}/O(2)$ and $\cZ_{-}$.
Furthermore we have $g^{*}\cO_{Y_{n,i}} = \cO_{Y_{i}}$ for $i = 1, 2$ and $g^{*}K_{n} = K$.
\end{nlemma}
\begin{proof}
Assume first that $p \in \PP(L)$ is a corank 2 point.
Then the previous lemma gives the correct $f$.
The morphism $g$ is smooth, and so since $g^{-1}(Y_{n,i}) = Y_{i}$ the claims about $g^{*}$ hold.

If $p$ is a corank 1 point, let $h : U \to \AA^{1}_{s}$ be the smooth map produced by Lemma \ref{thm:SmoothMorphism}.
Let $i : \AA^{1}_{s} \into \AA^{3}_{s,t,u}$ be the inclusion given by $s \mapsto (s,0,1)$, and let $f = i \circ h$, which gives rise to the correct Cartesian diagram.

Let $\ol{h}$ be the map $\cZ_{-}|_{U} \to \pi_{n}^{-1}(i(\AA^{1}_{s}))$.
Since the intersection $V_{i} = \pi_{n}^{-1}(i(\AA^{1}_{s})) \cap Y_{n,i}$ is regular for $i = 1,2$, we have $\cO_{Y_{n,i}}|_{\pi_{n}^{-1}(i(\AA^{1}_{s}))} = \cO_{V_{i}}$.
Then since $\ol{h}^{-1}(V_{i}) = Y_{i}$ and $\ol{h}$ is smooth, the claims about $g^{*}$ again hold.

If $p$ is a corank 0 point, let $f$ be the map to $(1,0,1) \in \AA^{3}_{s,t,u} = B$.
\end{proof}

\subsubsection{Standard local form with maximal isotropics}
We will need a version of this standard local form which includes a standardisation of maximal isotropic subbundles of the quadratic bundle.

Let $(H,q_{H}) \to \AA^{1}_{s}$ be given by $H = \cO^{n}$ and $q_{H} = sz_{1}^{2} + z_{2}z_{3} + \cdots z_{n-1}z_{n}$.
Over $0 \in \AA^{1}_{s}$ we define two maximal isotropic subspaces $L_{1},L_{2}$ of $(H|_{0},q_{H})$, where $L_{1}$ is defined by the vanishing of $z_{2k}$ and $L_{2}$ by the vanishing of $z_{2k+1}$ for all $1 \le k \le (n-1)/2$.

Let $(F,q) \to B$ be a quadratic vector bundle of rank $n$ such that $q$ pointwise has corank $\le 1$.
Let $B_{1} \subset B$ be the locus where $q$ has corank 1, and assume $B_{1}$ is a nonsingular divisor.
Let $M_{1}, M_{2} \subset F|_{B_{1}}$ be subbundles giving families of maximal isotropics over the corank 1 locus.
Assume that $M_{1} \cap M_{2} = \ker q$.
\begin{nlemma}
\label{thm:MaximalIsotropicLocalForm}
Étale locally on $B$, there is a smooth morphism $f : B \to \AA^{1}_{s}$ and an isomorphism $\phi : f^{*}(H,q_{H}) \to (F,q)$ such that $\phi(f^{*}(L_{i})) = M_{i}$.
\end{nlemma}
\begin{proof}
We work locally on $B$.
On $B_{1}$, let $K_{i} \subset M_{i}$ be subbundles with $\rk K_{i} = \rk M_{i} - 1$, not intersecting $\ker q$.
Extend these to bundles $\ol{K_{i}} \subset F$ on $B$.
Let $N = (\ol{K_{1}} \oplus \ol{K_{2}})^{\perp}$; we then get a splitting
\[
F = \ol{K_{1}} \oplus \ol{K_{2}} \oplus N.
\]
The fact that $M_{1} \cap M_{2} = \ker q$ implies that $q$ induces an isomorphism $K_{1} \to K_{2}^{\vee}$.
Choosing trivialisations of the $\ol{K_{i}}$ which respect this equivalence and trivialising $N$, we find that the splitting becomes
\[
F = \cO_{x_{i}}^{(n-1)/2} \oplus \cO^{(n-1)/2}_{y_{i}} \oplus \cO_{z},
\]
with $q = \sum x_{i}y_{i} + fz^{2}$ for some function $f$, and with $M_{1} = \{f = x_{i} = 0\}$ and $M_{2} = \{f = y_{i} = 0\}$.
The morphism defined by the function $f$ together with our chosen trivialisation give the conclusions of the lemma.
\end{proof}

Let $\pi : (\cX,W) \to B$ be the $O(2)$-equivariant LG model created from $(F,q)$ by the construction \ref{defi:QuadraticConstruction}.
Let $Z_{i} = M_{i} \oplus M_{i} \subset F \oplus F = X$, so we get $\cO_{Z_{i}} \in D^{b}(\cX,W)$.
\begin{nlemma}
\label{thm:MaximalIsotropicLocalFormMaps}
Étale locally on $B$, we have 
\[
\pi_{*}\Rhom(\cO_{Z_{1}},\cO_{Z_1}) \cong \pi_{*}\Rhom(\cO_{Z_2},\cO_{Z_2}) \cong \cO_{B_{1}}
\]
and 
\[
\Rhom(\cO_{Z_1},\cO_{Z_2}) \cong \Rhom(\cO_{Z_2},\cO_{Z_1}) \cong 0.
\]
\end{nlemma}
\begin{proof}
Suppose first that $B = \AA^{1}_{s}$ and $(F,q) = (H,q_{H})$. 
In this case, the claim is straightforward to check using a standard Koszul resolution of $\cO_{Z_{i}}$, see also \cite[4.1]{segal_quintic_2014}, \cite[A.4]{addington_d-brane_2014}.

For the general case, we use Lemma \ref{thm:MaximalIsotropicLocalForm}, by which we get a Cartesian diagram
\[
\begin{tikzcd}
\cX \arrow{d}{\pi} \arrow{r}{g} & (H \oplus H)/O(2) \arrow{d}{\pi_{H}} \\
B \arrow{r}{f} & \AA^{1}_{s},
\end{tikzcd}
\]
such that $f$ and hence $g$ is smooth, and such that $g^{-1}(L_{i} \oplus L_{i}) = M_{i} \oplus M_{i}$. 
We therefore have $g^{*}(\cO_{L_{i} \oplus L_{i}}) = \cO_{M_{i}\oplus M_{i}} = \cO_{Z_{i}}$, and the claim now follows by
\begin{align*}
\pi_{*}(\Rhom(\cO_{Z_{i}}, \cO_{Z_{j}})) &= \pi_{*}g^{*}\Rhom(\cO_{L_{i} \oplus L_{i}}, \cO_{L_{j} \oplus L_{j}}) \\
&= f^{*}(\pi_{H})_{*}(\Rhom(\cO_{L_{i} \oplus L_{i}}, \cO_{L_{j} \oplus L_{j}})) \\
&= f^{*}(\cO_{0}^{\delta_{ij}}) = \cO_{B_{1}}^{\delta_{ij}},
\end{align*}
using the smoothness of $f$.
\end{proof}

\subsection{\Knorrer{} periodicity}
Let $(X_{n},W_{n}) \to B$ denote the corank 2 standard LG model described in Definition \ref{defi:StandardSpaces}.
We choose coordinates such that $B = \AA^{3}_{s,t,u}$, $X_{n} = \AA^{n}_{x_{i}} \times \AA^{n}_{y_{i}} \times B$, and let the superpotential be
\[
W_{n} = sx_{1}y_{1} + t(x_{1}y_{2} + x_{2}y_{1}) + ux_{2}x_{2} + \sum_{k=3}^{n}x_{k}y_{k}.
\]
We will show that the category $D^{b}_{SO(2)}(X_{n},W_{n})$ (resp.\ $D^{b}_{O(2)}(X_{n},W_{n})$) is invariant under $n \mapsto n+1$ (resp.\ $n \mapsto n+2$).

\subsubsection{The $SO(2)$ case}
Let $p : X_{n+1} \to X_{n}$ be the projection which collapses the $x_{n+1}$ and $y_{n+1}$ directions.
Let $W^{\pr} = x_{n+1}y_{n+1}$, and let $\cF \in D^{b}_{SO(2)}(X_{n+2},W^{\prime})$ be the factorisation
\[
\cO(-1) \mfarrows{x_{n+1}}{y_{n+1}} \cO.
\]
Let $\Phi : D^{b}_{SO(2)}(X_{n},W_{n}) \to D^{b}_{SO(2)}(X_{n+1},W_{n+1})$ be the functor $p^{*}(-) \otimes \cF$.
\begin{nlemma}
\label{thm:SO2Knorrer}
The functor $\Phi$ is an equivalence.
The weights of $\Phi(\cE)$ satisfy $\wt(\Phi(\cE)) = \wt(\cE) + \{-1,0\}.$
We have $\Phi(\cO_{Y_{n,1}}) = \cO_{Y_{n+1,1}}$ and $\Phi(\cO_{Y_{n,2}}) = \cO_{Y_{n+1,2}}(-1)$.
\end{nlemma}
\begin{proof}
Ignoring the $SO(2)$- and $\CC^{*}_{R}$-actions, the claim that $\Phi$ is an equivalence is \Knorrer{} periodicity, see e.g.\ \cite[3.4]{shipman_geometric_2012}.
The inverse of $\Phi$ is given by $\Psi = p_{*}\Rhom(\cF, -)$.
It follows from this that the corresponding $SO(2) \times \CC^{*}_{R}$-equivariant functors are equivalences, since the isomorphisms
\[
\cE \to \Psi\Phi(\cE)\ \ \ \ \ \Phi\Psi(\cE) \to \cE
\]
are $SO(2)\times \CC_{R}^{*}$-equivariant.

The claim about the weights holds because the weights of $\cF$ are $\{-1,0\}$.

Let $A = \{x_{n+1} = 0\}$ and $B = \{y_{n+1} = 0\}$. 
Working in the category $D^{b}_{SO(2)}(X_{n+1},W^{\pr})$, we have
\[
\cO_{B}(-1) \cong \cF \cong \cO_{A},
\]
and the final claim follows from this.
\end{proof}

\subsubsection{The $O(2)$ case}
Let now $p : X_{n+2} \to X_{n}$ be the projection which collapses the $x_{n+1}$, $x_{n+2}$, $y_{n+1}$ and $y_{n+2}$ directions.
Let $W^{\pr} = x_{n+1}y_{n+1} + x_{n+2}y_{n+2}$, and let $\cF \in D^{b}_{O(2)}(X_{n+2},W^{\prime})$ be given by the $O(2)$-equivariant resolution
\[
\cO(-1) \oplus \cO(1) \mfarrows{d_{r}}{d_{l}} \cO_{+} \oplus \cO_{-},
\]
with
\[
d_{r} = \begin{pmatrix}
x_{n+1} & y_{n+1} \\
x_{n+2} & -y_{n+2}
\end{pmatrix}
\]
and
\[
d_{l} = \begin{pmatrix}
y_{n+1} & y_{n+2} \\
x_{n+1} & -x_{n+2}
\end{pmatrix}.
\]
Let $\Phi : D^{b}_{O(2)}(X_{n},W) \to D^{b}_{O(2)}(X_{n+2},W)$ be the functor $p^{*}(-) \otimes \cF$.
\begin{nlemma}
\label{thm:O2Knorrer}
The functor $\Phi$ is an equivalence, and it restricts to give an equivalence $D^{b}_{O(2)}(X_{n},W_{n})_{\res} \cong D^{b}_{O(2)}(X_{n+2},W_{n+2})_{\res}$.
Furthermore, we have $\Phi(K_{n}) = K_{n+2}$.
\end{nlemma}
\begin{proof}
The proof that $\Phi$ is an equivalence is the same as in Lemma \ref{thm:SO2Knorrer}.

The weights of $\cF$ at any point $p$ of the base are $\{-1,0, 1\}$.
Therefore $\cE \in D^{b}_{O(2)}(X_{n},W)$ has weights in $[-\lfloor \frac{n}{2} \rfloor, \lfloor \frac{n}{2} \rfloor]$ if and only if $\Phi(\cE)$ has weights in $[-\lfloor \frac{n}{2} \rfloor - 1, \lfloor \frac{n}{2} \rfloor + 1]$, which proves that $\Phi$ restricts to give $D^{b}_{O(2)}(X_{n},W)_{\res} \cong D^{b}_{O(2)}(X_{n+2},W)_{\res}$.

Working $SO(2)$-equivariantly, we have
\[
\cO_{B}(-1) \cong \cF \cong \cO_{A}(1),
\]
where $A = \{x_{n+1} = x_{n+2} = 0\}$ and $B = \{y_{n+1} = y_{n+2} = 0 \}$, and the isomorphism is by projection to the $\cO(1)$ and $\cO(-1)$ factors in the resolution of $\cF$.
It follows that we have $\Phi(K_{n}) \cong \cO_{Y_{n+2,1}}(\frac{n+1}{2}) \oplus \cO_{Y_{n+2,2}}(-\frac{n+1}{2}) = K_{n+2}$, and this isomorphism is $O(2)$-equivariant.
\end{proof}

\subsection{Computing $H(R)$ locally}
\label{sec:ComputingEndRing}
We now compute $H(R)$ on the standard corank 2 model from Definition \ref{defi:StandardSpaces}.

Let $\pi_{n} : (X_{n},W_{n}) \to B = \AA^{3}_{s,t,u}$ be the standard corank 2 model.
Let $S = \CC[s,t,u]$ be the coordinate ring of $B$, and let $P = (\pi_{n})_{*}\Rhom(K_{n},K_{n})$.
This is the local analogue of $H(R)$, in the sense that if $f : \PP(L) \to B$ is the locally defined map from Lemma \ref{thm:LocalForm}, then we have $f^{*}(P) \cong H(R)$. 

Our goal is now to compute $P$ as an $S$-algebra.
We show that $P$ is commutative, concentrated in cohomological degree 0, and that $\Spec P$ is a double cover of the corank 1 locus $su=t^{2}$, ramified in the corank 2 point $s = t = u = 0$:
\begin{nlemma}
\label{thm:localIso}
We have 
\[
P \cong \CC[s,t,u,\theta_{1},\theta_{2}]/(\theta_{1}^{2}-s, \theta_{1}\theta_{2} - t, \theta_{2}^{2}- u, su-t^{2}).
\]
\end{nlemma}

Recall that $f: Y \to \PP(L)$ is the ramified double cover of the corank 1 locus in $\PP(L)$.
Combining Lemmas \ref{thm:LocalForm} and \ref{thm:localIso} gives:
\begin{ncor}
\label{thm:HRIsYLocally}
Étale locally on $\PP(L)$, we have $H(R) \cong f_{*}\cO_{Y}$.
\end{ncor}

\begin{proof}[Proof of Lemma \ref{thm:localIso}]
By Lemma \ref{thm:O2Knorrer}, we reduce to proving the claim when $n = 3$, and now show that we may reduce further to a computation where $n = 2$.
Let $\cO_{Y_{i}} = \cO_{Y_{2,i}}$, let $K_{2}^{\pr} = \cO_{Y_{1}} \oplus \cO_{Y_{2}}(-1) \in D^{b}_{SO(2)}(X_{2},W_{2})$, and let $\widetilde{P}$ be the $S$-algebra $\RHom(K^{\prime}_{2}, K^{\prime}_{2})^{SO(2)}$. 
It decomposes as
\begin{align*}
\widetilde{P} &= \Hom(\cO_{Y_{1}}, \cO_{Y_{1}})^{SO(2)} \oplus \Hom(\cO_{Y_{1}}, \cO_{Y_{2}}(-1))^{SO(2)}\\
 &\oplus \Hom(\cO_{Y_{2}}(-1), \cO_{Y_{1}})^{SO(2)} \oplus \Hom(\cO_{Y_{2}}(-1),\cO_{Y_{2}}(-1))^{SO(2)}.
\end{align*}
Using the involution $\sigma$ of $X_{2}$, which permutes $Y_{1}$ and $Y_{2}$, we find natural isomorphisms
\begin{align*}
\Hom(\cO_{Y_{1}},\cO_{Y_{1}}) &\cong \Hom(\sigma^{*}(\cO_{Y_{2}}), \sigma^{*}(\cO_{Y_{2}})) \\
&\cong \Hom(\cO_{Y_{2}},\cO_{Y_{2}}) \cong \Hom(\cO_{Y_{2}}(-1),\cO_{Y_{2}}(-1)).
\end{align*}
In a similar way we can define an isomorphism $\Hom(\cO_{Y_{1}},\cO_{Y_{2}}(-1)) \cong \Hom(\cO_{Y_{2}}(-1), \cO_{Y_{1}})$, and this defines an action of $\ZZ_{2}$ on $\widetilde{P}$.
Lemmas \ref{thm:RTildeIsR} and \ref{thm:ComputingRTilde} now complete the proof.
\end{proof}

\begin{nlemma}
\label{thm:RTildeIsR}
We have an isomorphism of $S$-algebras $\widetilde{P}^{\ZZ_{2}} \cong P$.
\end{nlemma}
\begin{proof}
The \Knorrer{} functor of Lemma \ref{thm:SO2Knorrer} sends $K_{2}^{\pr} = \cO_{Y_{2,1}} \oplus \cO_{Y_{2,2}}(-1)$ to $K_{3} = \cO_{Y_{3,1}}(1) \oplus \cO_{Y_{3,2}}(-1)$, so we have an isomorphism of $S$-algebras $\REnd(K_{3})^{SO(2)} \cong \REnd(K^{\prime}_{2})^{SO(2)}$.

Endowing $K_{3}$ with the involution coming from its $O(2)$-structure, we obtain an involution on $\REnd(K_{3})$, and this corresponds to the involution on $H(\widetilde{R}) = \REnd(K_{2}^{\pr})$ defined above.
Hence $P = \REnd(K_{3})^{\ZZ_{2}} \cong \REnd(K^{\pr}_{2})^{\ZZ_{2}} = \widetilde{P}^{\ZZ_{2}}$, which is what we wanted.
\end{proof}

\begin{nlemma}
\label{thm:ComputingRTilde}
We have 
\[
\widetilde{P}^{\ZZ_{2}} \cong \CC[s,t,u,\theta_{1},\theta_{2}]/(\theta_{1}^{2}-s, \theta_{1}\theta_{2} - t, \theta_{2}^{2}- u, su-t^{2}).
\]
\end{nlemma}
\begin{proof}
Let $Y_{i} = Y_{2,i} \subset X_{2}$.
We introduce explicit resolutions of the objects $\cO_{Y_{i}}$ and compute.

Let 
\[
U = \Gamma(X_{2},\cO_{X_{2}}) = \CC[s,t,u,x_{1}, x_{2}, y_{1}, y_{2}].
\]
The $SO(2)$-action gives the $x_{i}$ degree 1, the $y_{i}$ degree $-1$, and gives $s,t,u$ degree 0.
The $x_{i}$ and $y_{i}$ have cohomological (i.e.\ $\CC^{*}_{R}$-) degree 1, while $s,t,u$ have cohomological degree 0.

First note that $\cO_{Y_{1}}$ has a locally free representative
\begin{equation}
\label{eqn:ResolutionOfK}
M_{1} = U(-1)^{2}[1] \overset{d_{1r}}{\underset{d_{1l}}{\rightleftarrows}} U[1] \oplus U(-1)^{2} \overset{d_{0r}}{\underset{d_{0l}}{\rightleftarrows}} U,
\end{equation}
where
\begin{align*}
\label{eqn:MFResolution}
d_{0r} &= \begin{pmatrix} su-t^2 & sx_1 + tx_2 & tx_1 + ux_2 \end{pmatrix} \\
d_{1r} &= \begin{pmatrix} x_1 & x_2 \\ -u & t \\ t & -s \end{pmatrix} \\
d_{0l} &= \begin{pmatrix} 0 \\ y_1 \\ y_2 \end{pmatrix} \\
d_{1l} &= \begin{pmatrix} sy_1 + ty_2 & -x_2y_2 & x_2y_1 \\ ty_1 + uy_2 & x_1y_2 & -x_1y_1 \end{pmatrix}.
\end{align*}
This follows by Lemma \ref{thm:Resolutions}, since the complex formed by the rightwards arrows forms a resolution of the sheaf $\cO_{Y_{1}}$.

We label the generators of the $U$-factors in \eqref{eqn:ResolutionOfK} as follows.
Let $e \in U$ be the generator of the rightmost factor, and let $g_{1},g_{2}$ be generators of the $U(-1)$-factors in the middle, ordered in the way in which they appear in the matrix.
Denote the generator of the middle $U[1]$ factor by $f_{1}$, and the generators of the two $U(-1)[1]$ factors by $f_{2}, f_{3}$.

Exchanging the $x_{i}$ and $y_{i}$ and shifting by $(-1)$, we obtain a locally free representative of $\cO_{Y_{2}}(-1)$:
\begin{equation}
\label{eqn:ResolutionOfK2}
M_{2} = U^{2}[1] \overset{}{\underset{}{\rightleftarrows}} U(-1)[1] \oplus U^{2} \overset{}{\underset{}{\rightleftarrows}} U(-1),
\end{equation}
Let $\sigma e, \sigma g_{i}, \sigma f_{i}$ be the generators of the $U$-factors in this resolution, defined as the corresponding generators for the resolution \eqref{eqn:ResolutionOfK} above.

Ignoring the algebra structure, we have 
\[
\widetilde{P}^{\ZZ_{2}} \cong \RHom(\cO_{Y_{1}}, \cO_{Y_{1}}) \oplus \RHom(\cO_{Y_{2}}(-1), \cO_{Y_{1}}).
\]
We first compute the components of this splitting as $S$-modules.

Let $Q$ be the coordinate ring of $Y_{1}$, that is
\[
Q = U/(sx_{1} + tx_{2}, tx_{1} + ux_{2}, su-t^{2}).
\]
For any factorisation $E$ we have an isomorphism
\begin{align*}
\RHom(E, \cO_{Y_{1}}) &= \RHom(E, \cO_{Y_{1}}) \\
&\cong \RHom_{Y_{1}}(E|_{Y_{1}}, \cO_{Y_{1}}) = \RGamma(Y_{1},(E|_{Y_{1}})^{\vee}).
\end{align*}
Applying this observation first to $E = \cO_{Y_{1}}$, we want to compute 
\begin{align*}
\RHom(\cO_{Y_{1}}, \cO_{Y_{1}}) &= \RGamma(Y_{1}, (\cO_{Y_{1}}^{\vee})|_{Y_{1}}) = H(M_{1}^{\vee}|_{Y_{1}})
\end{align*}
where $M_{1}$ is defined in \eqref{eqn:ResolutionOfK}.
The object $(M_{1}^{\vee})|_{Y_{1}}$ can be written as
\[
(M_{1}^{\vee})|_{Y_{1}} \cong Q[-1] \oplus Q(1)^{2} \overset{d_{r}}{\underset{d_{l}}{\rightleftarrows}} Q \oplus Q(1)^{2}[-1],
\]
where the differentials are
\[
d_{r} = \begin{pmatrix} 
0 & y_{1} & y_{2} \\
x_{1} & -u & t \\
x_{2} & t & -s
\end{pmatrix}
\]
and
\[
d_{l} = \begin{pmatrix}
0 & sy_{1} + ty_{2} & ty_{1} + uy_{2} \\
0 & -x_{2}y_{2} & x_{1}y_{2} \\
0 & x_{2}y_{1} & -x_{1}y_{1}.
\end{pmatrix}
\]
Using e.g.\ Macaulay2 \cite{M2}, we find that $\ker d_{r}/\im d_{l} = 0$ and that, as a $Q$-module, $\ker d_{l}/\im d_{r}$ is generated by $e^{\vee}$, subject to the relations 
\[
(x_{1}y_{1}, x_1y_2, x_{2}y_{1}, x_{2}y_{2}, y_1t+y_2u, y_1s+y_2t).
\]
With respect to the $SO(2)$-grading, the degree 0 part of $\ker d_{l}/\im d_{r}$ is therefore 
\[
Q_{0}/(x_{1}y_{1}, x_{1}y_{2}, x_{2}y_{1}, x_{2}y_{2})
\]
where $Q_{0}$ is the degree 0 part of $Q$.
As an $S$-module this is $S/(su-t^{2})$.
We have thus shown 
\[
\RHom(\cO_{Y_{1}},\cO_{Y_{1}}) = H(\Hom(M_{1},\cO_{Y_{1}})) = S/(su-t^{2}).
\]
This module is generated by $e^{\vee} \in \Hom(M_{1},\cO_{Y_{1}})$, which corresponds to the identity map on $\cO_{Y_{1}}$.

We can compute $\RHom(\cO_{Y_{2}}(-1), \cO_{Y_{1}})$ similarly.
Here 
\[
M_{2}^{\vee}|_{Y_{1}} = Q(1)[-1] \oplus Q^{2}  \overset{d_{l}}{\underset{d_{r}}{\leftrightarrows}} Q(1) \oplus Q^{2}[-1]
\]
with
\[
d_{r} = \begin{pmatrix} 
0 & x_{1} & x_{2} \\
y_{1} & -u & t \\
y_{2} & t & -s
\end{pmatrix}
\]
and
\[
d_{l} = \begin{pmatrix}
0 & 0 & 0 \\
sy_{1} + ty_{2} & -x_{2}y_{2} & x_{2}y_{1} \\
ty_{1} + uy_{2} & x_{1}y_{2} & -x_{1}y_{1}.
\end{pmatrix}
\]
Again we verify that $\ker d_{l}/\im d_{r} = 0$.
We further find that $\ker d_{r}$ is generated as a $Q$-module by 3 elements, $h_{1} = s\sigma g_{1}^{\vee} + t\sigma g_{2}^{\vee}$, $h_{2} = t\sigma g_{1}^{\vee} + u\sigma g_{2}^{\vee}$ and $h_{3} = x_{2}\sigma g_{1}^{\vee} - x_{1} \sigma g_{2}^{\vee}$.

We have $y_{1}h_{3}, y_{2}h_{3} \in \im d_{l}$, so the degree 0 part of $\ker d_{r}/\im d_{l}$ is generated by $h_{1}, h_{2}$.
For all $i,j \in \{1, 2\}$, we can find relations $x_{i}h_{j} = r_{ij}h_{3}$ with $r_{ij} \in Q$, and therefore $x_{i}y_{j}h_{k} \in \im d_{l}$.

It follows that the degree 0 part of $\ker d_{r}/\im d_{l}$ is generated by $h_{1}$ and $h_{2}$ as an $S$-module.
Since all elements of $\im d_{l}$ have higher cohomological degree than $h_{1},h_{2}$, in fact $(\ker d_{r}/\im d_{l})_{0}$ equals the $S$-submodule of $Q^{2} = Q(\sigma g_{1})^{\vee} \oplus Q(\sigma g_{2})^{\vee}$ generated by $h_{1}, h_{2}$. 
The only relations are then $sh_{1} - th_{2}$ and $th_{1} - uh_{2}$.
In conclusion, as an $S$-module we have 
\[
\RHom(\cO_{Y_{2}}(-1),\cO_{Y_{1}}) = H(\Hom(M_{2},\cO_{Y_{1}})) = S^{2}/((s,-t),(t,-u)).
\]

We have computed the $S$-module structure of $\widetilde{P}^{\ZZ_{2}}$; it remains to compute the algebra structure.
We choose a lifting of the elements $h_{i}$ to maps $\phi_{i} : M_{2} \to M_{1}$. 
Let $I \subset U$ be the ideal such that $Q = U/I$.
Working modulo $I$, the $\phi_{i}$ satisfy
\begin{align}
\label{eqn:PhiIEquation}
\begin{split}
e^\vee\phi_{1}(\sigma g_{1}) = s,\ e^\vee\phi_{1}(\sigma g_{2}) = t,\ e^\vee\phi_{1}(\sigma e) = e^\vee\phi_{1}(\sigma f_{i}) = 0, \\
e^\vee\phi_{2}(\sigma g_{2}) = t,\ e^\vee\phi_{2}(\sigma g_{2}) = u,\ e^\vee\phi_{2}(\sigma e) = e^\vee\phi_{2}(\sigma f_{i}) = 0.
\end{split}
\end{align}
\begin{nlemma}
\label{thm:CompositionLemma}
We have
\[
\begin{pmatrix}
e^{\vee}\phi_{1}(\sigma \phi_{1})(e) & e^{\vee}\phi_{1}(\sigma \phi_{2})(e) \\
e^{\vee}\phi_{1}(\sigma \phi_{2})(e) & e^{\vee}\phi_{2}(\sigma \phi_{2})(e)
\end{pmatrix}
= \begin{pmatrix}
s & t \\
t & u
\end{pmatrix} \mod I.
\]
\end{nlemma}
\begin{proof}
We prove $e^{\vee}\phi_{1}(\sigma \phi_{1})(e) = s$; the other cases are similar.
Note first that since $\sigma \phi_{1}$ preserves cohomological and $SO(2)$-degrees, we have 
\begin{equation*}
(\sigma \phi_{1})(e) = v_{1}\sigma g_{1} + v_{2} \sigma g_{2} + (v_{3}y_{1} + v_{4}y_{2})\sigma f_{1},\ \ \ \ v_{i} \in S.
\end{equation*}
Using \eqref{eqn:PhiIEquation} we get
\begin{equation*}
(\sigma e)^{\vee}d(\sigma \phi_{1})(e) = v_{1}(sy_{1} + ty_{2}) + v_{2}(ty_{1} + uy_{2}) \mod{(y_{1},y_{2})I}.
\end{equation*}
and
\begin{equation*}
(\sigma e)^{\vee}d(\sigma \phi_{1})(e) = (\sigma e)^{\vee}(\sigma \phi_{1})(de) = sy_{1} + ty_{2} \mod{(y_{1},y_{2})I}.
\end{equation*}
Combining these two equations we find that $v_{1} = 1$ and $v_{2} = 0$ mod $I$.
By \eqref{eqn:PhiIEquation} we then get
\begin{equation*}
e^{\vee}\phi_{1} (\sigma \phi_{1})(e) = v_{1}s + v_{2} t = s \mod{I}.
\end{equation*}
\end{proof}

Now since $\phi_{i}(\sigma\phi_{i}) \in \RHom(\cO_{Y_{1}},\cO_{Y_{1}}) \cong S/(su - t^{2})$, Lemma \ref{thm:CompositionLemma} shows
\begin{align*}
\phi_{1}\sigma\phi_{1} &\cong s\id \\
\phi_{1}\sigma\phi_{2} \cong \phi_{2}\sigma\phi_{1} &\cong t\id \\
\phi_{2}\sigma\phi_{2} &\cong u\id.
\end{align*}
Let $\theta_{i} = \phi_{i} + \sigma\phi_{i}$. 
We know that the algebra $\widetilde{P}^{\ZZ_{2}}$ is generated over $S/(su-t^{2})$ by the $\theta_{i}$.
By the above computation these satisfy
\begin{align*}
\theta_{1}\theta_{1} &= s \\
\theta_{1}\theta_{2} = \theta_{2}\theta_{1} &= t \\
\theta_{2}\theta_{2} &= u.
\end{align*}
This concludes the proof of Lemma \ref{thm:ComputingRTilde}.
\end{proof}

\subsection{The equivalence $D^{b}(\cZ_{-},W)_{\res} \cong D^{b}(\PP(L),R)$}
Let $F : D(\cZ_{-},W)_{\res} \to D(\PP(L),R)$ be given by $F(\cE) = \pi_{*}\sHom_{\cZ_{-}}(\cK,\cE)$.
The goal of this section is to show that $F$ gives an equivalence $D^{b}(\cZ_{-},W)_{\res} \cong D^{b}(\PP(L),R)$, as explained in the proof of Proposition \ref{thm:GeometricInterpretationOfMFCategory}.
\begin{nlemma}
The functor $F$ sends $D^{b}(\cZ_{-},W)_{\res}$ to $D^{b}(\PP(L),R)$.
\end{nlemma}
\begin{proof}
Let $\cE \in D^{b}(\cZ_{-},W)_{\res}$.
The cohomology of $\sHom_{\cZ_{-}}(\cK, \cE))$ is supported on the stack $\cC := \Crit(W)$ \cite[2.3.iii]{addington_pfaffian-grassmannian_2014}, \cite[Sec.\ 2]{shipman_geometric_2012}.
Let $f : \cC \to B$ be the projection.
The functor $f_{*}$ is exact, and it suffices to show that $f_{*}$ preserves coherent sheaves.

Let $C$ be the coarse space for $\cC$, i.e.\ the universal scheme admitting a map from $\cC$.
Locally on $B$, we may assume that $\cZ_{-} \to B$ has the form $(\AA_{x_{i}}^{n} \times \AA_{y_{i}}^{n})/O(2) \times B \to B$ with a superpotential and $O(2)$-action as in Construction \ref{defi:QuadraticConstruction}.
Let $S = k[B][x_{1}, \ldots x_{n}, y_{1}, \ldots y_{n}]$, and let $I \subset S$ be the Jacobi ideal of $W$.
We then get $C = \Spec (S/I)^{O(2)}$.

The morphism $f$ factors as $\cC \stackrel{g}{\to} C \stackrel{h}{\to} B$.
We first claim that $h$ is finite.
Applying Lemma \ref{thm:LocalForm}, the fact that finiteness is an fppf local property, and the fact that computing the critical locus commutes with smooth base change, we reduce to showing this for the standard models of coranks $\le 2$.
Then this claim follows from a straightforward computation.
Thus $h$ is finite and hence $h_{*}$ preserves coherent sheaves.

We next claim that $g_{*}$ preserves coherent sheaves.
This claim is Zariski local on $B$, so assume that we are in the affine setting described above.
Using this description the functor $g_{*}$ consists of taking $O(2)$-invariants.
The claim now follows from the fact that any coherent module on $\cC$ admits a surjection from a sheaf of the form $\oplus_{\rho} \cO(\rho)$, where $\rho \in \Irr(O(2))$ and that $g_{*}(\cO(\rho))$ is coherent for all such $\rho$.
It follows that $f_{*} = h_{*}g_{*}$ preserves coherent sheaves.
\end{proof}

We construct an adjoint to $F$ just as in the previous section.
Let the functor $UG: K(\PP(L),R) \to D(\cZ_{-},W)$ be given by $UG(M) = \pi^{*}(M) \otimes_{\pi^{*}(R)} \cK$.
The following lemma is shown in the same way as Lemmas \ref{thm:CliffordDerived}, \ref{thm:CliffordLeftAdjoint} and \ref{thm:CliffordFullyFaithful}.

\begin{nlemma}
\label{thm:GDerivedExists}
The functor $UG$ has a left derived functor $G : D(\PP(L),R) \to D(\cZ_{-},W)$ which is left adjoint to $F$ and is fully faithful.
\end{nlemma}

\begin{nlemma}
The functor $G$ sends $D^{b}(\PP(L),R)$ to $D^{b}(\cZ_{-},W)$.
\end{nlemma}
\begin{proof}
By Lemma \ref{thm:HRIsYLocally}, the $\cO_{\PP(L)}$-algebra $H(R)$ is commutative and $X = \Spec_{\PP(L)} H(R)$ is nonsingular.
By Lemma \ref{thm:QuasiIsoImpliesEquivalence}, there is then an equivalence $D^{b}(X) \cong D^{b}(\PP(L),R)$.
We first prove that if $\cE$ is a locally free sheaf on $X$ of finite rank, then under this isomorphism $G(\cE) \in D^{b}(\cZ_{-},W)$.
Choose an open $U \subset \PP(L)$ such that $\cE|_{U}$ is free.
Then the isomorphism $D(X|_{U}) \to D(U,R|_{U})$ sends the object $\cE|_{U}$ to a finite sum of copies of $R$.
Therefore $G(\cE|_{U})$ is a finite sum of copies of $\cK|_{U}$ -- in particular it lies in $D^{b}(\cZ_{-}|_{U},W)$.

Let now $F_{C}$ and $G_{C}$ be the functors called $F$ and $G$ in Section \ref{sec:CliffordAlgebras}.
Using Lemma \ref{thm:CliffordFBounded} we have $F_{C}G(\cE|_{U}) \in D^{b}(\cY|_{U},C)$.
It follows that the cohomology of $F_{C}G(\cE)$ is coherent over $U$, and as this holds for all $U$, the cohomology of $F_{C}G(\cE)$ is in fact coherent, so that $F_{C}G(\cE) \in D^{b}(\cY,C)$.
By Lemma \ref{thm:CliffordGBounded} and Proposition \ref{thm:MainPropositionClifford}, it follows that $G_{C}F_{C}G(\cE) = G(\cE)$ lies in $D^{b}(\cZ_{-},W)$

Since this holds for all locally free $\cE$ and since such $\cE$ generate $D^{b}(X)$, the claim follows.
\end{proof}

\begin{nlemma}
\label{thm:KIsRestricted}
The object $K \in D^{b}(\cZ_{-},W)$ is contained in $D^{b}(\cZ_{-},W)_{\res}$.
\end{nlemma}
\begin{proof}
We may check this étale locally, and therefore by Lemma \ref{thm:LocalForm} further reduce to checking the statement on a standard corank 2 LG model from Definition \ref{defi:StandardSpaces}.
By applying Lemma \ref{thm:SO2Knorrer}, we reduce to showing that the weights of $\cO_{Y_{2,1}}$ are $\{-1,0\}$, which holds because of the resolution \eqref{eqn:ResolutionOfK}.
\end{proof}

\begin{nlemma}
For any $M \in D^{b}(\PP(L),R)$, we have $G(M) \in D^{b}(\cZ_{-},W)_{\res}$.
\end{nlemma}
\begin{proof}
An object $\cE \in D^{b}(\cZ_{-},W)$ is grade restricted if and only if for every point $p \in \PP(L)$ the object $\cE|_{p}$ admits no shifted maps to $\cO_{p}(k)$ with $|k| > \lfloor\frac n 2 \rfloor$.
This means that $\RHom(\cE, i_{*}(\cO_{p}(k))) \cong 0$, where $i$ is the inclusion $p/\CC^{*} \into \cZ_{-}$.
Now for any $M \in D^{b}(\PP(L),R)$ we have
\[
\RHom(G(M), i_{*}(\cO_{p}(k))) = \RHom(M, \pi_{*}\Rhom_{\cZ_{-}}(\cK, i_{*}(\cO_{p}(k)))).
\]
Since $\cK \in D^{b}(\cZ_{-},W)_{\res}$, the complex $\pi_{*}\Rhom_{\cZ_{-}}(\cK, i_{*}(\cO_{p}(k)))$ is acyclic, and the claim follows.
\end{proof}

\begin{nlemma}
\label{thm:GeneratorDoubleCover}
If $\cE \in D^{b}(\cZ_{-},W)_{\res}$ is such that $\pi_{*}\sHom(\cK, \cE) \cong 0$, then $\cE = 0$.
\end{nlemma}
\begin{proof}
We first apply a trick taken from \cite{addington_pfaffian-grassmannian_2014}.
Assume that $\dim V = \dim L = 5$ and that $L$ is generic, so that we have $D^{b}(\cZ_{-},W)_{\res} \cong D^{b}(X)$.
Since $X$ is a smooth Calabi--Yau variety, it admits no nontrivial semiorthogonal decompositions.
Therefore in the decomposition $D^{b}(\cZ_{-},W)_{\res} = \langle \im G, \ker F\rangle$, we must have $\ker F = 0$.

Let $\cX_{n} \to \AA^{3}$ be the standard $O(2)$-equivariant corank 2 model.
Using Lemma \ref{thm:LocalForm}, the above special case implies that $\ker F = 0$ for the standard model $\cX_{5}$, since a counterexample $\cE \in D^{b}(\cX_{5},W)_{\res}$ would pull back to give a counterexample in $D^{b}(\cZ_{-},W)_{\res}$.
Applying Lemma \ref{thm:O2Knorrer}, it follows that $\ker F = 0$ for $\cX_{n} \to B$ for any odd $n$.

Now, let $V$ and $L$ be arbitrary, and assume for a contradiction that we have a counterexample $\cE$, that is $0 \not= \cE \in D^{b}(\cZ_{-},W)_{\res}$, but $F(\cE) = 0$.
There must be a point $p \in \PP(L)$ such that $\cE|_{\pi^{-1}(p)} \not \cong 0$, and by replacing $\cE$ with $\cE|_{\pi^{-1}(p)}$ we get a counterexample which is supported on $\pi^{-1}(p)$.
By Lemma \ref{thm:SmoothMorphism}, étale locally around $p$ there is a smooth morphism $f : \PP(L) \to B$, inducing a smooth morphism $\ol{f} : \cZ_{-} \to \cX_{n}$.
Then the projection $\ol{f}_{*}(\cE)$ is contained in $D^{b}(\cX_{n},W)_{\res}$, it is non-vanishing, and we have $F(\ol{f_{*}}(\cE)) = 0$.
Thus $\ol{f}_{*}(\cE)$ is a counterexample on $\cX_{n}$, which is a contradiction.
\end{proof}

\subsection{The global structure of $H(R)$}
By Corollary \ref{thm:HRIsYLocally}, we know that $H(R) \cong f_{*}(\cO_{Y})$ étale locally on $\PP(L)$.
Lemma \ref{thm:Corank1Fibre} shows that this is also true globally away from the corank 2 locus.
Applying Lemma \ref{thm:DeterminedVariety} with $B = \PP(L)$, $Y_{1} = Y$ and $Y_{2} = \Spec_{\PP(L)} H(R)$ completes the proof of Proposition \ref{thm:globalIso}.

\begin{nlemma}
\label{thm:Corank1Fibre}
Away from the corank 2 locus in $\PP(L)$, there is an isomorphism of $\cO_{\PP(L)}$-algebras $H(R) \cong f_{*}(\cO_{Y})$.
\end{nlemma}
\begin{proof}
Let $\PP(L)_{1} \subset \PP(L)$ be the locus of corank 1 points and let $p \in \PP(L)_{1}$.
We know from Proposition \ref{thm:localIso} that choosing an étale neighbourhood $U \to \PP(L)$ of $p$, we have
\begin{equation}
\label{eqn:Splitting}
\Spec_{\PP(L)} H(R)|_{U} = \PP(L)_{1}|_{U} \sqcup \PP(L)_{1}|_{U}.
\end{equation}
This induces a local isomorphism $\Spec_{\PP(L)} H(R) \cong Y$, and we must show that we can define this isomorphism globally.
It suffices to show that there is a canonical way of assigning the 2 components of the splitting \eqref{eqn:Splitting} to the 2 components of the space of maximal isotropic subspaces over $\PP(L)_{1}|_{U}$.

We may assume that $\cZ_{-}|_{U} \to U$ is given by Construction \ref{defi:QuadraticConstruction} applied to a quadratic vector bundle $(F,q) \to U$.
Locally we can choose maximal isotropic subbundles $L_{1}, L_{2} \subset F$ over $\PP(L)_{1}|_{U}$, satisfying $L_{1} \cap L_{2} = \ker q$.
We define objects $J_{1}, J_{2} \in D^{b}(\cZ_{-}|_{U},W)_{\res}$ by $J_{i} = \cO_{L_{i} \oplus L_{i}}$.

Let $p : \Spec H(R) \to U$ be the projection, and recall that we have an equivalence $F : D^{b}(\cZ_{-}|_{U},W)_{\res} \stackrel{\cong}{\to} D^{b}(H(R)|_{U})$.\footnote{The functor $F$ and its inverse are local over $\PP(L)$, so the equivalence holds after base change to $U$.}
Using Lemma \ref{thm:MaximalIsotropicLocalFormMaps}, we find that 
\[
p_{*}(\Rhom(F(J_{i}),F(J_{j}))) \cong \cO_{\PP(L)_{1}|_{U}}^{\delta_{ij}} \in D^{b}(H(R)|_{U}).
\]
This means firstly that each $F(J_{i})$ is supported on a single component of the splitting \eqref{eqn:Splitting}, since otherwise $p_{*}(\Rhom(J_{i},J_{i}))$ would be decomposable, and secondly that the $J_{i}$ must be supported on different components, since otherwise we would have
\begin{align*}
0 &\cong p_{*}(\Rhom(J_{1},J_{2}) \otimes \Rhom(J_{2},J_{1})) \\
&\cong p_{*}(\Rhom(J_{1},J_{1}) \otimes \Rhom(J_{2},J_{2})) \cong \cO_{\PP(L)_{1}}.
\end{align*}
We now define the isomorphism $\Spec H(R)|_{U} \to Y|_{U}$ by sending the component of the splitting of $H(R)$ on which $F(J_{i})$ is supported to the component of the isotropic Grassmannian which contains $L_{i}$.

We must check that this assignment is independent of our choice of $L_{i}$.
Suppose $L_{i}^{\pr}$ is a different choice, with each $L_{i}$ in the same connected component as $L_{i}^{\pr}$.
Choose a smooth curve $C \subset U$ which intersects $\PP(L)_{1}$ transversely in $p$.

Let $q_{p}$ be the quadratic form on $V$ at $p$.
Base changing to $C$, the 2 pairs of maximal isotropics $\{L_{i}|_{p}\}$ and $\{L_{i}^{\pr}|_{p}\}$ induce as above 2 bijections between the components of $\Spec H(R) \cap C \cong \pt \sqcup \pt$ and the components of the space of maximal isotropic subspaces of $(V,q_{p})$.
We may deform the pair $\{L_{i}|_{p}\}$ to $\{L_{i}^{\pr}|_{p}\}$, and since the choice of bijection is discrete, the 2 bijections are the same.
\end{proof}

\begin{nlemma}
\label{thm:DeterminedVariety}
For $i = 1,2$, let $f_{i} : Y_{i} \to B$ be a finite, dominant map of varieties, with $Y_{i}$ normal.
If there is an open subset $U \subseteq B$ such that $Y_{1}|_{U} \cong Y_{2}|_{U}$ as $U$-schemes, then $Y_{1} \cong Y_{2}$ as $B$-schemes.
\end{nlemma}
\begin{proof}
Let $K(Y_{1})$ be the function field of $Y_{1}$, considered as a constant sheaf on $B$.
Let us also consider $\cO_{Y_{1}}$ as a sheaf on $B$ via the map $f_{1}$.
We claim that $\cO_{Y_{1}}$ equals the integral closure $\ol{\cO}_{B}$ of $\cO_{B}$ in $K(Y_{1})$.
Since $f_{1}$ is finite, we have $\cO_{Y_{1}} \subseteq \ol{\cO}_{B}$.
On the other hand, since $Y_{1}$ is normal, $\cO_{Y_{1}}$ is integrally closed, hence $\ol{\cO}_{B} \subseteq \ol{\cO}_{Y_{1}} = \cO_{Y_{1}}$.

By the same argument, $\cO_{Y_{2}}$ is the integral closure of $\cO_{B}$ in $K(Y_{2})$.
But as $Y_{1}|_{U} \cong Y_{2}|_{U}$, we have $K(Y_{1}) \cong K(Y_{2})$, and the claim follows.
\end{proof}

\section{Geometric interpretation of $D^{b}(\cZ_{-},W)_{\res}$, even case}
Assume now that $n = \dim V$ is even, that the corank of the quadratic form at each point of $\PP(L) \subset \PP(\Sym^{2}V^{\vee})$ is at most $1$, and that the locus of corank 1 points is a nonsingular divisor.
This assumption holds for a generic $L$ of dimension $\le 3$.

We define the variety $f : Y \to \PP(L)$ as the nonsingular double cover of the corank 0 locus, ramified in the corank 1 locus.
At a corank 0 point $q \in \PP(L)$, the 2 points of the fibre $f^{-1}(q)$ correspond to the 2 components of the space of maximal isotropic subspaces of $(V,q)$.
We then have:
\begin{nprop}
\label{thm:InterpretationOfCliffordNEvenLater}
Under the assumptions above, $D^{b}(\cZ_{-},W)_{\res} \cong D^{b}(Y)$.
\end{nprop}

The method of proof is the same as for the case of odd $n$ in Proposition \ref{thm:GeometricInterpretationOfMFCategory}, and we only indicate the necessary changes.

We let $Y_{1} = 0 \times V \times (L \setminus 0) \subset Z^{\sstable}_{-}$ and $Y_{2} = V \times 0 \times (L \setminus 0) \subset Z^{\sstable}_{-}$.
We then get $\cO_{Y_{1}}, \cO_{Y_{2}} \in D^{b}_{T}(Z^{\sstable}_{-},W)$, and let
\[
K = \cO_{Y_{1}}(n/2,0) \oplus \cO_{Y_{2}}(0,n/2) \in D_{G}^{b}(Z^{\sstable}_{-},W).
\]
Choosing a locally free resolution $\cK$ of $K$, we get a dg algebra $R = \pi_{*}(\sHom(\cK,\cK))$ and a functor $F = \pi_{*}(\Rhom(\cK,-)) : D^{b}(\cZ_{-},W)_{\res} \to D^{b}(\PP(L),R)$.

\begin{nprop}
There is an isomorphism of $\cO_{\PP(L)}$-algebras $H(R) \cong f_{*}(\cO_{\PP(L)})$.
\end{nprop}
The proof of this proposition is carried out in the same way as that for Proposition \ref{thm:globalIso}, and the computations are simpler in the even case.

The proof of Proposition \ref{thm:InterpretationOfCliffordNEvenLater} now goes the same way as that of Proposition \ref{thm:GeometricInterpretationOfMFCategory}, except for one difficulty: In proving that $0 = \ker F \subset D^{b}(\cZ_{-},W)_{\res}$ (Lemma \ref{thm:GeneratorDoubleCover}), we no longer have an equivalence $D^{b}(X) \cong D^{b}(\cZ_{-},W)_{\res}$ for a variety $X$, and so the initial step in the proof of Lemma \ref{thm:GeneratorDoubleCover} does not work.
Instead we prove directly that in a standard corank 1 model, the appropriate object is a local generator for the grade restricted category.
The rest of the proof of Lemma \ref{thm:GeneratorDoubleCover} then goes through.

Let $X_{n} = \AA^{n} \times \AA^{n} \times \AA^{1} \stackrel{\pi}{\to} \AA^{1}$ be the standard corank 1 model, let $Y_{1} = 0 \times \AA^{n} \times \AA^{1}$, $Y_{2} = \AA^{n} \times 0 \times \AA^{1}$, and let $K \in D^{b}_{O(2)}(X_{n}, W)$ be given by
\[
K = \cO_{Y_{1}}(n/2) \oplus \cO_{Y_{2}}(-n/2).
\]
\begin{nlemma}
If $\cE \in D^{b}_{O(2)}(X_{n},W)_{\res}$ is such that $\pi_{*}(\Rhom(K,\cE)) = 0$, then $\cE = 0$.
\end{nlemma}
\begin{proof}
Since $X_{n}$ is affine, we have $\pi_{*}(\Rhom(K,\cE)) = \RHom_{X_{n}/O(2)}(K,\cE)$.
We then have $\RHom_{X_{n}/SO(2)}(\cO_{Y_{1}}(n/2), \cE) = \RHom_{X_{n}/O(2)}(K,\cE) = 0$.

Taking a standard resolution of $\cO_{Y_{1}}$ shows that $\cO_{Y_{1}}(n/2) \cong \cO_{Y_{1}}(n/2)^{\vee}$, and so we get
\begin{align*}
\RHom_{X_{n}/SO(2)}(\cO_{Y_{1}}(n/2),\cE) &= \RHom_{X_{n}/SO(2)}(\cE^{\vee}, \cO_{Y_{1}}(n/2)) \\
&= \RGamma(Y_{1},\cE(n/2)|_{Y_{1}})^{SO(2)} = 0.
\end{align*}
Now the weights of $\cE(n/2)|_{Y_{1}}$ are contained in $[0, n]$, and if $0$ were a weight of $\cE(n/2)|_{Y_{1}}$, then by Lemma \ref{thm:MaximalsGiveMaps} we would have $\RHom_{Y_{n}/SO(2)}(\cO, \cE(n/2)|_{Y_{1}}) \not= 0$.
Therefore $0$ is not a weight of $\cE(n/2)$, which means that $-n/2$ is not a weight of $\cE$.
Since $\cE$ is $O(2)$-equivariant, it follows that $n/2$ is not a weight of $\cE$ either, and so $\wt(\cE) \subseteq [-n/2+1, n/2-1]$.

By Lemma \ref{thm:SO2Knorrer} we get an equivalence $D^{b}_{SO(2)}(X_{1},W) \to D^{b}_{SO(2)}(X_{n},W)$.
Let $\cE^{\prime} \in D^{b}_{SO(2)}(X_{1},W)$ be the object sent to $\cE$ under this equivalence.

Assume now for a contradiction that $\cE \not= 0$.
Then $\cE^{\prime} \not= 0$ and so $\wt(\cE^{\prime}) \not= \varnothing$.
Now by the statement about weights in Lemma \ref{thm:SO2Knorrer}, it follows that $\max \{k\in \wt(\cE)\} - \min \{k\in \wt(\cE)\} \ge n-1$.
This contradicts the fact that $\wt(\cE) \subseteq [-n/2+1, n/2-1]$, and we obtain the desired conclusion $\cE = 0$.
\end{proof}

\section{The case of $\PP(V)^{2}$}
Our results and proofs extend with minor changes to the case of intersections of $(1,1)$-divisors in $\PP(V)^{2}$.
Let $f : \PP(V)^{2} \to \PP(V^{\otimes 2})$ be the Segre embedding, let $L \subset (V^{\otimes 2})^{\vee}$ be a linear subspace, and let $X = f^{-1}(\PP(L^{\perp}))$.

Recall that $T = (\CC^{*})^{2} \subset G$, and let $\cZ = (V \times V \times L)/T$, where the $T$-action is induced from the $G$-action.
Let $\cY_{L}$ be the substack $0 \times 0 \times (L\setminus 0)/T$.
The natural map $\cY_{L} \to \PP(L)$ makes $\cY_{L}$ an $SO(2)$-gerbe.

Thinking of a point $p \in L \subset (V^{\otimes 2})^{\vee}$ as a bilinear function on $V$ gives a natural superpotential $W$ on $\cZ$.
The map $\cZ \to \cY_{L}$ is a vector bundle, and $W$ induces a quadratic form on this bundle.
Proceeding now in the same way as in Section \ref{sec:CliffordAlgebras}, we get a sheaf of Clifford algebras $C$ on $\cY_{L}$ and a subcategory $D^{b}(\cY_{L},C)_{\res} \subset D^{b}(\cY_{L},C)$.

With this recycling of notation, Theorem \ref{thm:MainTheorem} holds verbatim, and the proof we have given in the $\Sym^{2}\PP(V)$-case goes through with minor changes; taking the same GIT stabilities and using similar definitions for the window categories.

We can interpret this as saying that $D^{b}(\cY_{V^{\otimes 2}},C)_{\res}$ is an HP dual for $\PP(V)^{2}$ with respect to the line bundle $\cO(1,1)$ and a Lefschetz decomposition of $D^{b}(\PP(V)^{2})$ 
described as follows. 
Let
\[
\cA = \langle \cO(i,j) \rangle_{(i,j) \in S},
\]
where $S = \{(i,j)\ |\ i + j \in [0,1], |i - j| \le \lfloor \frac{n}{2} \rfloor\}$.
If $n$ is odd, we take $\cA_{0} = \cdots = \cA_{n-1} = \cA$.

If $n$ is even, we let $\cA_{0} = \cdots = \cA_{n/2-1} = \cA$.
We let $S^{\pr} = S = \{(i,j)\ |\ i + j \in [0,1], |i - j| \le \frac{n}{2} - 1\}$, let
\[
\cA^{\pr} = \langle \cO(i,j)\rangle_{(i,j) \in S^{\pr}} ,
\]
and then let $\cA_{n/2} = \cdots \cA_{n-1} = \cA^{\pr}$.

We can also describe $D^{b}(\cY_{L},C)_{\res}$ more geometrically in this case.
Thinking of $\PP(V^{\otimes 2})$ as a space of $(n \times n)$-matrices, it is stratified by the rank of the matrices.
Assume $n$ is odd, that $\PP(L)$ does not intersect the locus of matrices of corank $\ge 2$, and that the locus of corank 1 points in $\PP(L)$ is a nonsingular divisor.
Let $Y \subset \PP(L)$ be the corank 1 locus.
\begin{nprop}
\label{thm:DerivedCAtegoryIsYSO2Case}
Under the above assumptions, we have 
\[
D^{b}(\cY_{L},C)_{\res} \cong D^{b}(Y).
\]
\end{nprop}

This proposition is proved along the lines of Proposition \ref{thm:GeometricInterpretationOfMFCategory}.
We replace the local generator $K = \cO_{Y_{1}}((n-1)/2, 0) \oplus \cO_{Y_{2}}(0, (n-1)/2) \in D^{b}_{G}(Z^{\mathrm{ss}}_{-},W)$ used in Section \ref{sec:GeometricInterpretation} by the object $\cO_{Y_{1}}((n-1)/2,0) \in D^{b}_{T}(Z^{\mathrm{ss}}_{-},W)$.

Computations like those in the proof of Lemma \ref{thm:ComputingRTilde} show that 
\[
\pi_{*}\Rhom(\cO_{Y_{1}}((n-1)/2,0), \cO_{Y_{1}}((n-1)/2,0)) \cong \cO_{Y}
\]
as an $\cO_{\PP(L)}$-algebra, and the rest of the argument in Section \ref{sec:GeometricInterpretation} goes through to show
\[
D^{b}(\cY_{L},C)_{\res} \cong D^{b}(\PP(L), \cO_{Y}) \cong D^{b}(Y).
\]

\footnotesize{
\bibliographystyle{alpha-abbrv}
\bibliography{bibliography}
}
\end{document}